\documentclass[11pt]{article}

\usepackage{times} 
\usepackage{microtype}
\usepackage[utf8]{inputenc}
\usepackage{amsmath}
\usepackage{graphicx}
\usepackage[margin = 1 in]{geometry}
\usepackage{amsfonts}
\usepackage{amsthm}
\usepackage{mathrsfs}
\usepackage{bigints}
\usepackage{amssymb}
\usepackage{booktabs}
\usepackage{xspace}
\usepackage{tikz}
\usepackage{tikz-cd}
\usepackage{enumerate}
\usepackage{hyperref}
\usepackage{upgreek}
\usepackage{verbatim}
\usepackage{mathtools}
\mathtoolsset{showonlyrefs}

\newcommand{\R}{\mathbb{R}}

\newcommand{\N}{\mathbb{N}}

\newcommand{\spr}{V}

\newcommand{\KK}{[K]}
\newcommand{\ns}{K}
\newcommand{\s}{k}

\newcommand{\J}{[J]}

\newcommand\numberthis{\addtocounter{equation}{1}\tag{\theequation}}
\newcommand{\wmass}{L}
\newcommand{\awmass}{\mathcal{L}}
\newcommand{\load}{\varrho}

\newcommand{\supp}{\operatorname{supp}}

\newcommand{\ar}{\alpha} 
\newcommand{\ap}{A} 

\newcommand{\servp}{S} 
\newcommand{\p}{p} 

\newcommand{\inta}{u} 
\newcommand{\iinta}{U}
\newcommand{\pat}{\ell} 
\newcommand{\pd}{\vartheta} 
\newcommand{\ser}{v} 
\newcommand{\tm}{Z} 
\newcommand{\ur}{\kappa} 
\newcommand{\dt}{\eta} 
\newcommand{\ssp}{\mathcal{Z}} 
\newcommand{\fl}{\zeta} 
\newcommand{\flm}{z} 

\newcommand{\avg}{H} 
\newcommand{\mart}{Y} 
\newcommand{\othermart}{\mathcal{Y}}

\newcommand{\sr}{\mu} 
\newcommand{\flcdf}{M}
\newcommand{\pdcdf}{N}

\newcommand{\M}{\mathbf{M}}

\theoremstyle{definition}

\theoremstyle{definition}

\theoremstyle{definition}
\newtheorem{defi}{Definition}[section]
\theoremstyle{plain}
\newtheorem{lem}{Lemma}[section]
\theoremstyle{plain}
\newtheorem{prop}{Proposition}[section]
\theoremstyle{plain}
\newtheorem{thm}{Theorem}[section]
\theoremstyle{plain}
\newtheorem{cor}{Corollary}[section]
\theoremstyle{definition}
\newtheorem{rem}{Remark}[section]
\theoremstyle{definition}

\theoremstyle{definition}
\newtheorem{assumption}{Assumption}

\title{Diffusion Limits for Measure-Valued Queueing Models, With Application to General Random Order of Service Queues\footnote{The research reported in this paper was supported in part by NSF RTG grant DMS-2134107.}}
\begin{document}

\author{Eva Loeser\footnote{Department of Statistics and Operations Research, University of North Carolina at Chapel Hill, 204 E Cameron Ave, Chapel Hill, NC 27514. Email: ehloeser@unc.edu.} } 

\maketitle
\begin{abstract}
Currently, there is no general theory for deriving diffusion approximations of queueing systems with high- or infinite-dimensional state descriptors. In this paper, we explore one path for deriving diffusion limit equations of queueing models. The method hinges on a martingale decomposition of dynamics driven by time-changed renewal processes, which are a common feature of many queueing models. 
We then prove a central limit theorem for models decomposed in this way, which gives the form of stochastic differential equations (SDEs) that will be satisfied in the diffusion limit of such a system.
We demonstrate the approach on a multiclass, multi-server random order of service queue with reneging and generally distributed interarrival, service, and patience times. In this setting, the state descriptor is measure-valued and the workload is nonlinear in the fluid limit. We prove tightness and uniqueness of limit points, and derive SDEs satisfied by the diffusion limit. Our results offer a broadly applicable method for diffusion approximations in complex queueing systems without relying on dimension reductions or model-specific structure.

\end{abstract}

\section{Introduction}

In this paper, we develop a general framework for establishing tightness and deriving stochastic differential equations (SDEs) satisfied by the diffusion limits of queueing systems. Our methods apply broadly across a wide class of models, as they exploit a structural feature common to most queueing systems: state changes either occur at the jump times of time-changed renewal processes (e.g., service completions, arrivals) or are driven by deterministic dynamics.
Both types of dynamics are covered by the Central Limit Theorem (CLT) for Renewal-Driven Systems (Theorem \ref{diffusionclt}) proved in this paper, which derives SDEs that will be satisfied by diffusion limits of such models.

The use of diffusion approximations in analyzing queueing networks dates back several decades~\cite{kingman1961single,halfinwhitt1981many,bramson,williamsssc,miyazawa2014review,ruthsurveypaper}.
However, these approximations can be quite technical, and there is currently no general theory for obtaining diffusion approximations for queueing models.
This gap is especially pronounced in the context of non-Markovian systems with features such as reneging, non-Head-of-the-line (non-HL) service disciplines, and infinite-dimensional state descriptors ~\cite{ruthsurveypaper}.
Martingale methods, such as those used in this paper, for obtaining such approximations have become increasingly common.
Pang, Talreja, and Whitt~\cite{pangetal2007martingale} outline a general framework for using martingale decompositions to derive diffusion approximations for queueing networks.
However, their analysis focuses on Markovian models, primarily addresses finite-dimensional performance processes, and does not account for reneging.

The martingale decomposition used here is closely related to, and can be viewed as a
reformulation of, Lemma 2.1 of Daley and Miyazawa~\cite{martingaledecomppaper}.
Thus, the decomposition itself is not the main novelty of the paper. Rather, the
queueing-specific contribution lies in how this decomposition is applied to the ROS
state descriptor. In particular, it leads to a family of multi-index mass-transport
martingales that capture the fluctuations of the system. Using the tightness criterion
developed in Lemma~\ref{ornsteinuhlenbecktightnesscondition}, we take the limit of this
family and obtain a multiparameter noise process that is \emph{not} a Brownian sheet in
the usual sense, since it has nontrivial correlations between disjoint rectangles; see
Theorem~\ref{Ydefthm} for a full description of the limiting martingale field.
We then use a method-of-characteristics argument, following appropriately chosen paths
through this multiparameter martingale noise, to characterize the limiting behavior of
the system. This provides a systematic route to tightness and subsequential SDE
characterizations for a measure-valued non-HL queue with reneging.

Ward and Glynn~\cite{wardglynn2005balking} obtained the first diffusion approximation for a system with both generally distributed primitives and reneging, establishing an Ornstein--Uhlenbeck representation for the workload and queue-length processes of a GI/GI/1 queue with reneging.
Huang, Zhang, and Zhang~\cite{huangzhangzhang2012unified} present a unified theory of diffusion approximations for queues with general patience-time distributions, but again, their analysis is restricted to finite-dimensional performance processes, rather than the infinite-dimensional state descriptors that are the focus of this paper.

In order to understand the limiting behavior of the infinite-dimensional state descriptors, we formulate the limit as an SPDE on the dual of the Schwartz space.
SPDEs have been used to characterize the limiting behavior of infinite-dimensional state descriptors of complex queueing models \cite{KaspiRamanan2013SPDE,AghajaniRamanan2019Ergodicity,AghajaniRamanan2020StationaryDistributions}. There are strong connections between this work and \cite{KaspiRamanan2013SPDE}, where an SPDE limit on the dual of the Schwartz space is established for a model with generally distributed primitives, and martingale methods are used.
However, the model in \cite{KaspiRamanan2013SPDE} has a large-server limit, whereas the model we study here has a fixed-server limit. Furthermore, the model in this paper has a random order of service discipline with reneging.

A major advantage of our approach is that it accommodates queueing systems with generally distributed primitives and high-dimensional (often infinite-dimensional) state descriptors--systems increasingly prominent in contemporary research \cite{gromollpuhawilliams, gromoll, krukrecent, puhabanerjeebudhiraja}. Furthermore, our framework uses classical FCLT methods to obtain a diffusion approximation of the queueing system. 
The scaling considered here should be distinguished from heavy-traffic regimes in which
state space collapse may reveal strong correlations among components of the state descriptor.
Our focus is instead on fluctuations about a deterministic fluid limit in an overloaded ROS
system. In this regime, centering about the fluid model is necessary to obtain a nontrivial
limit, and one should not generally expect the same collapse phenomena that arise in
critically loaded systems. Thus, our goal is not to replace state-space-collapse methods,
but to develop tools for a different scaling regime that is more relevant to queueing systems with reneging.

To illustrate the power of the method, we apply it to a multiclass, multi-server, random order of service queue with generally distributed interarrival, service, and patience times, and with reneging.
We establish tightness, prove uniqueness of limit points, and identify the diffusion
limit through a system of SDEs (Theorems \ref{tightnessresult}, \ref{lhatconvergencethm}, and \ref{diffusionapproximationresult}).
This model features a measure-valued state descriptor and a workload process that is nonlinear in the fluid limit \cite{loeserwilliams}.

\subsection{Contributions}

Currently, there is no general theory for obtaining diffusion approximations for queueing models with generally distributed primitives and non-head-of-the-line (non-HL) service disciplines \cite{ruthsurveypaper}. Such approximations are technically challenging and often rely heavily on model-specific structures or dimension-reduction techniques.

We demonstrate our framework on a model whose diffusion approximation was previously unknown: the multiclass, multi-server random order of service (ROS) queue with reneging and generally distributed interarrival, service, and patience times (see \S\ref{modelsetupintro}). However, these methods can be used to derive and generalize many of the fluid- and diffusion-limit equations that are currently known. 

As an example, consider the processor-sharing (PS) queue with impatience studied in \cite{gromollkruk}. There, a state space collapse argument is used to lift the workload process into the infinite-dimensional space where the state descriptor lives. To make this argument tractable, the authors assume \emph{soft deadlines}, meaning that jobs do not leave the system when their patience times expire. This assumption allows the diffusion limit of the workload process to coincide with that of a classical PS queue, described by a reflected Brownian motion. The case with \emph{firm deadlines} (i.e., reneging) is left for future work and described as exhibiting ``very different behavior.''

The same formal decomposition can be applied to a PS queue with impatience and reneging, and it suggests candidate fluid and diffusion equations for that model.
As an example, we derive candidate fluid and diffusion equations for the model described in that paper but with the addition of reneging. 
This is done in \S \ref{gromollkrukexamplesect}.
This is only meant as an illustrative example, as proofs of convergence, uniqueness of solutions, and tightness are left for future work.

Furthermore, continuously varying families of martingales arise naturally when martingale decompositions are used with measure-valued processes.
In the proof of tightness \S \ref{proofoftightnesssect}, uniform tightness of one such continuously varying family of martingales is proved.
These methods are novel and interesting in their own right.

\subsection{The Random Order of Service Queue as a Model for Enzymatic Processing}
\label{modelsetupintro}

Our methodology was developed, in part, to analyze a multiclass, multi-server random order of service (ROS) queue with reneging and generally distributed primitives. In this setting, classical state space collapse techniques—such as those developed by Bramson \cite{bramson} and Williams \cite{williamsssc}—do not apply due to the system's non-head-of-line (non-HL) service discipline and non-exponential patience times.

Because patience times are generally distributed, no finite-dimensional Markovian state descriptor exists. We thus model the system using a \emph{measure-valued state descriptor} and prove convergence in the Skorokhod space $D([0,\infty), \mathscr{S}')$, where $\mathscr{S}'$ denotes the space of tempered distributions. Our framework does not require the state process to be real-valued or finite-dimensional, only vector-valued.

While ROS queues have been studied in operations research for decades \cite{flatto, borst, rogiest}, an especially compelling application for this model comes from enzymatic processing \cite{criticalityadaptivity, experimentpaper, factorizedtimedependence}. We briefly describe this application to build intuition for the model setup.

Consider an enzyme operating in an enclosed biological environment, such as a cell nucleus. The cell produces various molecules that need to be processed (i.e., broken down and removed) by the enzyme. In queueing terms, this is not a head-of-the-line (HL) system: the enzyme is equally likely to select any molecule of a given type for processing, regardless of its arrival time.
In the case of one enzyme and a single molecule type, this setup corresponds to a single-server ROS queue. 
Allowing for unprocessed molecules to degrade or dilute out of the system over time, as well as the possibility of multiple enzyme copies and molecule types, this motivates a multiclass, multi-server ROS queue with reneging, where:
\begin{itemize}
    \item Molecules have type-dependent selection weights $p_j$ that reflect their binding affinity with the enzyme;
    \item The selection probability for type-$j$ molecules is proportional to $p_j \flm_j$ when $\flm_j$ molecules of type $j$ are present. To be specific, if $\boldsymbol{\flm} = (\flm_1, \ldots, \flm_J)$ is nonzero, then the probability of an enzyme selecting a type-$j$ molecule is given by
\[
\frac{p_j \flm_j}{\sum_{i=1}^J p_i \flm_i}.
\]
    \item Molecules leave the system upon expiration of a finite patience time.
\end{itemize}
Visualizing each type of molecule in its own virtual queue, the schematic for $J\in \N$ molecule types and $K\in \N$ identical enzymes is as pictured below.
\begin{center}
\includegraphics[width=0.5\textwidth]{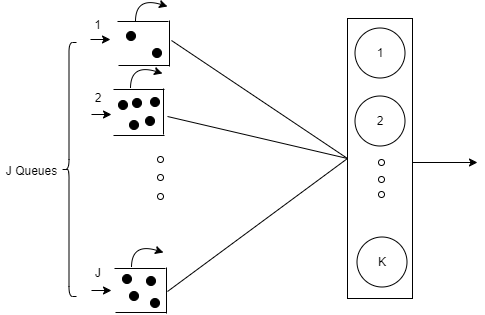}
\end{center}
While this example is couched in biochemical terms, ROS queues are broadly applicable in queueing theory. For instance, servers may represent identical processors in a data center, and molecule types may correspond to different job classes with priority weights. 

\subsection{Notation} We shall use the following notation throughout the paper.
Let $\N$ denote the set of strictly positive integers, $\{1,2,....\},$ and let $\N_0= \N \cup \{0\}$. 
For a positive integer $N,$ let $[N]$ denote the set $\{1,..., N\}.$ 
For $x\in \R$ we denote the positive part of $x$ by $x^+:=x\vee 0.$ 
For a finite set $A \subset\R_+,$ we denote the $i$th smallest element of $A$ by $A_{\{i\}}$. 
Let $\chi(x):=x$ for $x \geq 0.$ 
We denote the zero vector in any vector space by $\mathbf{0}$. 
For a vector $\boldsymbol{x}\in \R^d,$ we write $\boldsymbol{x}>\boldsymbol{0}$ if and only if $x_i>0$ for $i=1,...,d.$  
For $X=\R$ or $X= \R_+$, we denote the set of bounded continuous functions defined on $X$ and taking values in $\R$ by $\mathbf{C}_b(X).$
The set of functions in $\mathbf{C}_b(X)$ that have bounded continuous derivatives up to order $n\geq 1$ is denoted by $\mathbf{C}_b^n(X).$
For $T \geq 0$ and a bounded continuous function $f:\R_+\rightarrow \R,$ we write $||f||_T$ for $\sup_{t\in[0,T]}|f(t)|.$ 
We take $\sup\emptyset$ to be $0$ and $\inf \emptyset$ to be $+\infty.$
Let $\R_+=[0,\infty),$ and consider it with the Borel $\sigma$-algebra $\mathscr{B}(\R_+).$ 
We denote the set of signed, finite measures on $(\R_+, \mathscr{B}(\R_+))$ by $\M.$ 
We endow $\M$ with the topology of weak convergence of measures.
If $\xi \in \M$ and $f$ is a Borel measurable function on $\R_+$ that is integrable with respect to $\xi$, we let $\langle f, \xi\rangle := \int_{\R_+}fd\xi.$ 
If $F$ is a function of bounded variation and $g$ is integrable with respect to $\mu_F,$ the Lebesgue-Stieltjes measure associated to the function $F,$ then we denote $\int_{(s,t]} g d\mu_F$ as $\int_s^t gdF.$
We denote the Schwartz space on $[0,\infty)$ as $\mathscr{S}.$
We denote the space of functions from $[0,\infty)$ to $\R^d$ that are right continuous with finite left limits by $D([0,\infty),\R^d)$. 
We endow $D([0,\infty),\R^d)$ with the Skorokhod-$J_1$ topology, under which it is a Polish space. 
We denote $\delta_x^+:= 1_{\{x>0\}}\delta_x.$
We will commonly denote a vector by using a bold symbol. 
For example, if we have introduced $x_1,...,x_d,$ then $\boldsymbol{x}$ will be $(x_1,...,x_d)^{\bot}.$ Similarly, if we have also introduced $y_1,...,y_d,$ then $\boldsymbol{xy}$ will be $(x_1y_1,...,x_dy_d)^{\bot},$ and so on.
If $\boldsymbol{\nu}\in \M^d$ for some $d \in \N,$ and $\boldsymbol{f}\in \mathscr{B}(\R_+)^d,$ then we denote the vector $(\langle f_1, \nu_1\rangle, ...,\langle f_d, \nu_d\rangle )^{\bot}$ as $\langle \boldsymbol{f}, \boldsymbol{\nu}\rangle.$
Similarly, we denote $(\langle 1, \nu_1\rangle, ...,\langle 1, \nu_d\rangle )^{\bot}$ as $\langle \boldsymbol{1}, \boldsymbol{\nu}\rangle$ and $(\langle \chi, \nu_1\rangle, ...,\langle \chi, \nu_d\rangle )^{\bot}$ as $\langle \boldsymbol{\chi}, \boldsymbol{\nu}\rangle.$

\section{Multiclass Random Order of Service Queue}
\label{modeldescriptionsect}

In this section, we describe the multiclass, multi-server, random order of service (ROS) queue that motivates the development of the methods in this paper.

\subsection{Model Setup}

We now give a more formal description of the model introduced in \S \ref{modelsetupintro}.
There are $J$ job classes, each with a separate queue, and $K$ identical servers. Jobs of class $j \in [J]$ arrive according to a delayed renewal process $\ap_j(\cdot)$, and the vector of arrival processes is denoted by
\[
\boldsymbol{\ap}(\cdot) = (\ap_1(\cdot), \ldots, \ap_J(\cdot)).
\]
Service and patience times are generally distributed. A job can renege if not selected for service within its patience time. Once in service, jobs cannot renege.

\begin{rem}
The arrival and service time distributions are assumed to have no atoms. This precludes the occurrence of simultaneous arrivals, reneging, or service entries almost surely. See \S 2.4 of \cite{loeserwilliams} for details.
\label{simultaneouseventsremark}
\end{rem}

To describe the system’s state, we use a measure-valued process that tracks the remaining patience times of all jobs in each queue. We define key random variables below.

\paragraph{(i) Arrivals:} For each $j \in [J]$, let $\inta_0^j$ be the first arrival time to class-$j$, and $\{\inta_i^j\}_{i=1}^\infty$ be the i.i.d.\ interarrival times. Then $\iinta_i^j := \sum_{l=0}^{i-1} \inta_l^j$ is the arrival time of the $i$th job. The cumulative arrival process is
\[
\ap_j(t) := \sup\left\{ i \in \N : \iinta_i^j \leq t \right\}.
\]
Let $\tau_i^{\ap,j} := \iinta_i^j$ denote the $i$th arrival time to queue $j$.

\paragraph{(ii) Patience Times:} Let $\pat_i^j$ denote the patience time of the $i$th class-$j$ job. We take $\{\pat_i^j\}_{i=1}^{\infty}$ to be i.i.d.. We define the remaining patience at time $t \geq 0$ as
\[
\pat_i^j(t) := \pat_i^j + \iinta_i^j - t.
\]

\paragraph{(iii) Service Times:} Service times are indexed by job class and server. Let $\ser_i^{k,j}$ denote the service time of the $i$th class-$j$ job served at server $k$. 
We take $\{\ser_i^{k,j}\}_{i=1}^{\infty}$ to be i.i.d..
For $t\geq 0,$ let
\begin{equation}
\spr_j^k(t) := \sup \left\{n : \sum_{i=1}^{n} \ser_i^{k,j} \leq t \right\}, \qquad
g_j^k(t) := \int_0^t c_j^k(s)\,ds,
\label{gjkdef}
\end{equation}
where $c_j^k(s) = 1$ if server $k$ is serving a class-$j$ job at time $s$, and 0 otherwise. Then service completions occur at the jump times of $V_j^k(g_j^k(\cdot))$, and we let $\tau_i^{\spr,k,j}$ denote the $i$th such time.
We index our sequences of service times differently than in \cite{loeserwilliams};
however, the resulting system and related processes have the same distribution as those
in that paper.

\paragraph{(iv) Server Assignments:} When a job enters service at server $k$ at time $\tau_i^{\spr,k,j}< \infty$, selection is determined via i.i.d. random variables $\ur_i^{k,j} \sim \mathrm{Uniform}(0,1)$. Define the class-selection interval
\[
I_l(\boldsymbol{\flm}) := \left[ \frac{\sum_{n=1}^{l-1} p_n \flm_n}{\sum_{n=1}^{J} p_n \flm_n}, \frac{\sum_{n=1}^{l} p_n \flm_n}{\sum_{n=1}^{J} p_n \flm_n} \right),
\]
and the within-class subinterval
\[
I_{l,m}(\boldsymbol{\flm}) := \left[ \frac{\sum_{n=1}^{l-1} p_n \flm_n + p_l(m-1)}{\sum_{n=1}^J p_n \flm_n}, \frac{\sum_{n=1}^{l-1} p_n \flm_n + p_l m}{\sum_{n=1}^J p_n \flm_n} \right),
\]
so that the job selected is the one with the $m$th smallest patience time in queue $l$ if $\ur_i^{k,j} \in I_{l,m}(\boldsymbol{\flm})$.

\paragraph{(v) Additional State Variables:} Let $a_j(t)$ be the time until the next arrival to class $j$ at time $t \geq 0$, and $s^k(t)$ be the remaining service time at server $k$ at time $t \geq 0$. If server $k$ is idle, $s^k(t) = 0$. Let $\servp_j^k(t)$ denote the number of class-$j$ jobs served at server $k$ by time $t$, and define
\[
\servp^k(t) := \sum_{j=1}^J \servp_j^k(t), \qquad
\servp(t) := \sum_{k=1}^K \servp^k(t).
\]

\paragraph{(vi) Initial Conditions:} Let $\boldsymbol{\tm}_0 = (\tm_{0,1}, \ldots, \tm_{0,J})$ denote the initial queue lengths. Let $\tilde{\pat}_{-i}^j$ be the remaining patience time of the $i$th initial job in queue $j$, where $\{\tilde{\pat}_{-i}^j\}_{i=1}^{\infty}$ are i.i.d..
Let $s_0^k$ be the remaining service time at server $k$ at $t=0$. If $s_0^k = 0$, then server $k$ is idle. We assume $s_0^k = 0$ only if $\boldsymbol{\tm}_0 = \boldsymbol{0}$ because the system is non-idling, and $s_0^k \neq \inta_0^j$ for any $j,k$ to avoid simultaneous events. 
We define $\tau_0^{\spr,k,1} := s_0^k$ and $\tau_0^{\spr,k,j} := 0$ for $j \neq 1$.

\subsection*{State Descriptor}

The measure-valued process $\ssp_j(t)$ for $j \in [J]$ tracks the remaining patience times of jobs in queue $j$ at time $t$:
\begin{align*}
\ssp_j(t) &:= \sum_{i=1}^{\tm_{0,j}} \delta_{\tilde{\pat}_{-i}^j - t}^+
+ \sum_{i=1}^{\ap_j(t)} 1_{\{ s^k(\iinta_i^j-) \neq 0 \,\, \forall k \in [K] \}} \delta_{\iinta_i^j + \pat_i^j - t}^+ \\
&\quad - \sum_{k \in [K]} \sum_{l \in [J]} \sum_{\tau_i^{\spr,k,l} \in (0,t]} 1_{\{\boldsymbol{\ssp}(\tau_i^{\spr,k,l}-) \neq 0\}} \delta^+_{T_{i,j}^{k,l} - t + \tau_i^{\spr,k,l}},
\numberthis
\label{statespacedescriptorequation}
\end{align*}
where
\[
T_{i,j}^{k,l} := \sum_{n=1}^{\tm_j(\tau_i^{\spr,k,l}-)} 1_{\{ \ur_i^{k,l} \in I_{j,n}(\boldsymbol{\tm}(\tau_i^{\spr,k,l}-)) \}} (\text{supp}(\ssp_j(\tau_i^{\spr,k,l}-)))_{\{n\}},
\]
if a class-$j$ job enters service at $\tau_i^{\spr,k,l}$, and $T_{i,j}^{k,l} = 0$ otherwise. If fewer than $i$ jobs of class $l$ enter service at server $k$, then $\tau_i^{\spr,k,l} = \infty$ and $T_{i,j}^{k,l} := \infty$.

This state descriptor is distributionally equivalent to that in \cite{loeserwilliams}, but expressed in terms of $\ser_i^{k,j}$ and $\tau_i^{\spr,k,j}$, consistent with the indexing used in this paper. The full system state at time $t \geq 0$ is given by
\[
\boldsymbol{X}(t) := (\boldsymbol{\ssp}(t), \boldsymbol{a}(t), \boldsymbol{s}(t)) \in \M^J \times \mathbb{R}^J \times \mathbb{R}^K.
\]

\subsection{Fluid- and Diffusion-Scaled Models}
\label{sequenceofmodelssection}

Because we will be focusing on the overloaded regime, the model will not be balanced. 
Therefore, we will need to center the model before attempting a FCLT-type limit. The model will be centered around the unique fluid model solution, established in \cite{loeserwilliams}, associated to the limiting initial condition. 
The fluid model solution is defined in that paper as follows:
\begin{defi}[Fluid model parameters] A vector $(\boldsymbol{\ar},\boldsymbol{\sr},\boldsymbol{\p},\boldsymbol{\pd})\in \R^J_+\times\R^J_+\times (0,1)^J\times \M^J$ is a set of fluid model parameters if $\boldsymbol{\ar}>\boldsymbol{0},$ $\boldsymbol{\sr}>\boldsymbol{0},$ $\sum_{j=1}^Jp_j=1,$ and $\pd_j$ is a probability measure with $\pd_j(\{0\})=0$ for each $j \in \J.$
\label{parametersdef}
\end{defi}
\begin{defi}[Fluid Model Solution]
\label{notoverdef}
Let $\boldsymbol{\fl}: [0,\infty) \rightarrow \M^J$ be a continuous function. 
Then we say that $\boldsymbol{\fl}$ is a fluid model solution for fluid model parameters $(\boldsymbol{\ar},\boldsymbol{\sr},\boldsymbol{p},\boldsymbol{\pd})$ satisfying Definition \ref{parametersdef} and initial condition $\boldsymbol{\fl}_0 = (\fl_{0,1},...,\fl_{0,J}),$ a vector of continuous measures, if 

\begin{enumerate}[(i)]
\item $\boldsymbol{\fl}(0) = \boldsymbol{\fl}_0,$ \label{startsatICdef}
\item $\langle 1_{\{0\}}, \fl_j(t) \rangle = 0$ for each $t\geq0,j \in \J,$
\label{doesntchargeorigindef}
\item \label{fluidequationdef} for each $f \in  \mathbf{C}^1_b(\R_+) $ such that $ f(0) =0,  j \in \J, t\geq 0,$
\begin{align*}
\langle f, \fl_j(t)\rangle &= \langle f, \fl_j(0) \rangle -\int_0^t \left\langle f', \fl_j(s) \right\rangle ds  \quad - \int_0^t    \ns1_{\{\boldsymbol{\fl}(s) \neq \boldsymbol{0}\}}\frac{p_j\langle f, \fl_j(s)\rangle}{\sum_{i=1}^J\frac{p_i}{\sr_i}\langle 1, \fl_i(s) \rangle } ds \\&\hspace{8mm}+ \ar_j  \langle f,  \pd_j\rangle \int_0^t 1_{\{\boldsymbol{\fl}(s) \neq \boldsymbol{0}\}}ds, \numberthis \label{fluidlimiteqn}
\end{align*}
\item and when $\load >1,$ at each $t >0,$ $\langle 1, \fl_j(t)\rangle >0$ for some $j \in [J].$ 
\label{overloadedconditiondef}
\end{enumerate}
\label{fmsleq}
\end{defi}

The fluid model is obtained as the limit of a sequence of fluid-scaled models, described in \S\ref{modeldescriptionsect}. For a sequence of models indexed by $m \in \N$, the fluid-scaled state descriptor for class $j$ in the $m$th model is defined as follows: for any Borel set $B \subseteq \R_+$,
\begin{equation}
\bar{\ssp}^{m}_j(t)(B) := \frac{1}{m} \ssp^{m}_j(mt)(mB), \qquad t \geq 0,
\label{scalingwithborelsets}
\end{equation}
or, equivalently, for any bounded Borel measurable function $f: \R_+ \to \R$,
\begin{equation}
\langle f, \bar{\ssp}^{m}_j(t) \rangle = \frac{1}{m} \left\langle f\left(\frac{1}{m} \cdot \right), \ssp^{m}_j(mt) \right\rangle.
\label{scalingwithfunctions}
\end{equation}

Throughout the paper, we append the superscript $m$ to denote quantities associated with the $m$th model. For example, the interarrival times for class $j$ are denoted $\{\inta_i^{j,m}\}_{i=1}^{\infty}$. We also use an overbar to denote fluid-scaled processes; for instance, the fluid-scaled arrival process for class $j$ is $\bar{\ap}^m_j(t) := \frac{1}{m} \ap^m_j(mt)$.

Some parameters remain fixed across the sequence of models. Specifically, for $j \in \J,k \in [K],$ the selection probabilities $p_j$, $j \in [J]$, are constant in $m$; the choosing variables $\{\ur_i^{k,j}\}_{i=1}^{\infty}$ are also held fixed; patience times scale with $m$ so that there exists a sequence $\{\pat_i^j\}_{i=1}^{\infty}$ satisfying $\pat_i^{j,m} = m \pat_i^j$ for all $i, m \in \N$; the service rate $\sr_j := 1/\mathbb{E}[\ser_1^{j,m}]$ is independent of $m$.

We now list the assumptions imposed on this sequence of models:

\begin{assumption}
\label{assumptions}
The following conditions are assumed throughout:
\begin{enumerate}[(i)]
\item \label{basicassumptions}
For all $j \in [J]$, $k \in [K]$, and $m \in \N$, the arrival rate $\ar_j^m := 1/\mathbb{E}[\inta_1^{j,m}]$, reneging rate $\gamma_j^m := 1/\mathbb{E}[\pat_1^j]$, and service rate $\sr_j := 1/\mathbb{E}[\ser_1^{k,j,m}]$ are all positive and finite.
The expected initial number of jobs in the queue for class $j$, $\mathbb{E}[\tm_{j,0}^m]$, is finite.
The distributions of $\inta_0^{j,m}$, $\inta_1^{j,m}$, $\ser_1^{k,j,m}$, and $s_0^{k,m}$ have no atoms.
The patience time distribution $\pd_j$ for class $j$ is independent of $m$, and the service rate $\sr_j$ is also independent of $m$. 
Additionally, for each $t \geq 0$, $j \in [J]$, and $k \in [K]$,
\[
\sup_{m \in \N} \mathbb{E}[\bar{\ap}_j^m(t)] < \infty, \qquad
\sup_{m \in \N} \mathbb{E}[\bar{\spr}_j^{k,m}(t)] < \infty.
\]

\item For each $m \in \N, j \in [J],k\in [K]$, the sequences $\{\inta_i^{j,m}\}_{i=1}^{\infty}$, $\{\ser_i^{k,j,m}\}_{i=1}^{\infty}$, $\{\pat_i^j\}_{i=1}^{\infty}$, and $\{\tilde{\pat}_{-i}^j\}_{i=1}^{\infty}$ are mutually independent and independent of the initial condition $(\boldsymbol{\tm}^m(0), \boldsymbol{a}^m(0), \boldsymbol{s}^m(0))$.

\item \label{parametersassumption}
There exists a limiting arrival rate vector $\boldsymbol{\ar} > 0$ such that $\boldsymbol{\ar}^m \to \boldsymbol{\ar}$ as $m \to \infty$. The system is overloaded in the limit, meaning the load parameter
\[
\load := \sum_{j=1}^J \frac{\ar_j}{\ns \sr_j} > 1.
\]

\item \label{fllnassumption1}
For all $j \in [J]$, $k \in [K]$, we assume that
\[
\frac{\mathbb{E}[\inta_0^{j,m}]}{\sqrt{m}} \to 0, \qquad
\frac{\mathbb{E}[s_0^{k,m}]}{\sqrt{m}} \to 0, \qquad \text{as } m \to \infty.
\]

\item \label{fllnassumption2}
For each $j \in [J]$, $k \in [K]$,
\[
\mathbb{E}[\inta_1^{j,m} \cdot 1_{\{\inta_1^{j,m} > m\}}] \to 0, \quad
\mathbb{E}[\ser_1^{k,j,m} \cdot 1_{\{\ser_1^{k,j,m} > m\}}] \to 0, \quad \text{as } m \to \infty,
\]
and the third moments are uniformly bounded:
\[
\sup_{m \in \N} \mathbb{E}[|\ser_1^{k,j,m}|^3] < \infty, \qquad
\sup_{m \in \N} \mathbb{E}[|\inta_1^{j,m}|^3] < \infty.
\]
\item The patience time distribution, $\pd_j$, has a H\"older continuous CDF.
\item \label{initialconditionsassumption}
There exists a deterministic vector of nonzero measures $\bar{\boldsymbol{\ssp}}_0 = (\bar{\ssp}_{0,1}, \dots, \bar{\ssp}_{0,J})$ such that $\langle \chi, \bar{\ssp}_{0,j} \rangle < \infty$ for all $j \in [J]$, and
\[
(\bar{\boldsymbol{\ssp}}^m(0), \langle \boldsymbol{\chi}, \bar{\boldsymbol{\ssp}}^m(0) \rangle)
\Rightarrow
(\bar{\boldsymbol{\ssp}}_0, \langle \boldsymbol{\chi}, \bar{\boldsymbol{\ssp}}_0 \rangle), \qquad \text{as } m \to \infty.
\]
Furthermore, we assume that the generalized cdf of $\bar{\boldsymbol{\ssp}}_0,$ $\bar{F}(x):= \langle 1_{[0,x]},\bar{\boldsymbol{\ssp}}_0\rangle, $ is H\"older continuous.
In addition, there exists a random variable $\hat{\boldsymbol{\ssp}}_0 \in (\mathscr{S}')^J$ such that for any $f_1,\dots,f_J \in \mathscr{S}$,
\[
(\langle f_1, \hat{\ssp}_1^m(0) \rangle, \dots, \langle f_J, \hat{\ssp}_J^m(0) \rangle)
\Rightarrow
(\langle f_1, \hat{\ssp}_{0,1} \rangle, \dots, \langle f_J, \hat{\ssp}_{0,J} \rangle),
\]
and for any $f \in \mathscr{S} \cup \{1_{(0,\infty)}\}$, the functions
\[
F_f^{j,m}(x) := \left\langle f((\cdot - x)^+), \hat{\ssp}_j^m(0) \right\rangle
\]
converge in distribution to a random function $F_f^{j}(x)$ that is continuous almost surely.

\item \label{cltforrenewalassumption}
Let $\sigma_{\ap,j} > 0$ and $\sigma_{\spr,j} > 0$ be the limiting standard deviations of $\inta_1^{j,m}$ and $\ser_1^{1,j,m}$, respectively. Then for each $j \in [J]$ and $k \in [K]$, the diffusion-scaled processes
\[
\hat{\ap}_j^m(\cdot) := \frac{1}{\sqrt{m}} (\ap_j^m(m\cdot) - \ar_j(\cdot)), \qquad
\hat{\spr}_j^{k,m}(\cdot) := \frac{1}{\sqrt{m}} (\spr_j^{k,m}(m\cdot) - \sr_j (\cdot))
\]
are such that the central limit theorem for renewal processes holds, i.e. the diffusion-scaled renewal process converges in distribution to a Brownian motion with quadratic variation $\iota^3 \sigma^2 t$, where $\iota$ is the limiting rate of the renewal process and $\sigma$ is the limiting standard deviation of its interevent times of the renewal process (see \cite{chenandyao}, Theorem 5.11).
\end{enumerate}
\end{assumption}

\medskip

\noindent
We note that even though Assumption~\ref{cltforrenewalassumption} is stated explicitly for the diffusion scaling, Assumptions~\ref{fllnassumption1} and \ref{fllnassumption2} also ensure a functional law of large numbers (FLLN) for the fluid-scaled renewal processes $\{\bar{\ap}^m_j(\cdot)\}_{m=1}^{\infty}$ and $\{\bar{\spr}^{k,m}_j(\cdot)\}_{m=1}^{\infty}$, as shown in Lemma A.2 of \cite{gromollpuhawilliams}. Together, these assumptions ensure that the limiting parameters of the model sequence described in \S\ref{modeldescriptionsect} satisfy the conditions of Definition~\ref{parametersdef}.

This fluid-scaling was analyzed by the authors of \cite{loeserwilliams}, and the results that will be central to this paper can be summarized as follows:
\begin{thm}[Loeser--Williams]
\label{fluidlimittheorem}
    Under the conditions in Assumption \ref{assumptions}, a sequence of fluid-scaled models of a multiclass, multi-server random order of service queue with reneging as described in \S \ref{modeldescriptionsect} and \S \ref{sequenceofmodelssection} is tight, and all subsequential limits are fluid model solutions. If either $\load= \sum_{j=1}^J \frac{\ar_j}{\ns \sr_j}\leq 1$ or $\boldsymbol{\fl}_0 \neq \boldsymbol{0}$, then fluid model solutions are unique, and thus the original sequence converges to a fluid model solution.
    \end{thm}

Centering around this fluid model, we define our diffusion scaling
\begin{equation}
\hat{\boldsymbol{\ssp}}^m(\cdot) = \sqrt{m}(\bar{\boldsymbol{\ssp}}^m-\boldsymbol{\fl}(\cdot)), 
\label{diffusiondefinitioneqn}
\end{equation}
where for $\omega \in \Omega,$ $\boldsymbol{\fl}(\omega)$ is the unique fluid model solution with initial condition $\bar{\boldsymbol{\ssp}}_0(\omega).$
Similar to the bar denoting fluid-scaling, we use a hat to denote diffusion-scaling for relevant processes.

\section{Main Results}
\label{diffresultsect}
Before presenting our results, it is important to remark on some key properties of \eqref{diffusiondefinitioneqn} which will inform us on what types of convergence we can and cannot attain.
The key observation is that, for each $t \geq 0, j \in [J]$, $\fl_j(t)$ is a continuous measure by Lemma 4.2 of \cite{loeserwilliams}. Furthermore, $\bar{\ssp}_j^m(t)$ is the sum of weighted delta masses.
Thus, the two measures are singular with respect to each other.
It follows that the total variation of $\hat{\ssp}^m_j(t)$
$$\|\hat{\ssp}^m_j(t)\|_{\mathrm{TV}}
=
\sqrt m\,\|\bar{\ssp}_j^m(t)\|_{\mathrm{TV}}
+
\sqrt m\,\|\fl_j(t)\|_{\mathrm{TV}} \rightarrow^m \infty.$$
Thus, one cannot hope to achieve tightness of measure with respect to the topology of weak convergence of signed measures.
For this reason, we work towards convergence of the sequence $\{\hat{\boldsymbol{\ssp}}^m(\cdot)\}_{m=1}^{\infty}$ in the space $D([0,\infty), \mathscr{S}')^J,$ where $\mathscr{S}$ is the Schwartz space on $[0,\infty)$ and $\mathscr{S}'$ is its dual, as explored in \cite{mitoma}.
Following Theorem 5.3 2) of \cite{mitoma} and the extension of that theorem to the interval $[0,\infty)$ in Remark (R.2.2) of the same work, we see that in order to obtain the limit of $\{\hat{\boldsymbol{\ssp}}^m(\cdot)\}_{m=1}^{\infty},$ we need to
\begin{enumerate}
    \item show that $\{\langle f, \hat{\boldsymbol{\ssp}}^m(\cdot)\rangle \}_{m=1}^{\infty}$ is tight for each $f \in \mathscr{S}$ and
    \label{firstmainpartmitoma}
    \item show that for any finite collection $f_1,...,f_n \in \mathscr{S}$ and $t_1,...,t_n \in [0,\infty),$
    $$ (\langle f_1, \hat{\boldsymbol{\ssp}}^m(t_1)\rangle ,...,\langle f_n, \hat{\boldsymbol{\ssp}}^m(t_n) \rangle )$$ converges in law to some $n$-dimensional probability distribution.
    \label{secondmainpartmitoma}
\end{enumerate}
Together, these two steps imply convergence in distribution of
\(\{\hat{\boldsymbol{\ssp}}^m(\cdot)\}\) in \(D([0,\infty),(\mathscr S')^J)\) to a limit process $\hat{\boldsymbol{\ssp}}(\cdot) \in D([0,\infty),\mathscr{S}')^J$ whose finite dimensional distributions are equal in law to the limits established in part \ref{secondmainpartmitoma} above.
Theorem \ref{tightnessresult} of this paper supplies \ref{firstmainpartmitoma}.
Theorem \ref{lhatconvergencethm} provides uniqueness in distribution of limit points.
Theorem \ref{diffusionapproximationresult} provides SDEs that will be satisfied by subsequential limits of $\langle \boldsymbol{f}, \hat{\boldsymbol{\ssp}}^m(\cdot) \rangle$ for a large class of test functions $\boldsymbol{f}$.
In order to write this theorem, it will be helpful to introduce some notation for certain variables and processes, much of which was also used in \cite{loeserwilliams}. 
We define the class of functions,
\begin{equation}
\mathscr{C}:=\{f \in \mathbf{C}^1_b(\R_+) | f(0) =0\} .
\label{classoffunctions}
\end{equation}
In the following, $(\boldsymbol{\ar}, \boldsymbol{\sr}, \boldsymbol{p}, \boldsymbol{\pd})\in \R_+^J \times \R_+^J\times (0,1)^J\times \M^J$ are parameters satisfying Definition \ref{parametersdef}.
Given $\boldsymbol{\flm} \in \R_+^J$, define a weighted mass
\begin{equation}\wmass(\boldsymbol{\flm}):= \sum_{j=1}^J p_j \flm_j,  \label{wmassequation} \end{equation}  and an adjusted weighted mass \begin{equation}\awmass(\boldsymbol{\flm}):= \sum_{j=1}^J \frac{p_j}{\sr_j} \flm_j.
\label{aloadequation}\end{equation}
We denote the load parameter
\begin{equation}
    \load:= \sum_{j=1}^J \frac{\ar_j}{\sr_j \ns}
    \label{load},
\end{equation}
 and an analogue to a total mass vector for our diffusion-scaled state descriptor as
 \begin{equation}
     \hat{\boldsymbol{\tm}}^m(\cdot):= \langle 1, \hat{\boldsymbol{\ssp}}^m(\cdot)\rangle.
     \label{totalmassnotation}
 \end{equation}
We also take some time to define our driving diffusive noise field.
\begin{thm}
    \label{Ydefthm}
    Let $\boldsymbol{\fl}$ be a fluid model solution with a nonzero initial condition $\boldsymbol{\fl}_0.$
    For each finite collection of functions \(f_1,\ldots,f_n\in\mathscr C\), there exists a unique in distribution, continuous multiparameter $\R^{n(J+J^2K)}$-valued process
    \[
\left\{
\hat{\mart}_{f_a}^{\ap_j,j}(r,t),
\hat{\mart}_{f_a}^{\spr_j^k,i}(r,t):
a\in[n],\ j,i\in[J],\ k\in[K],\ 0\le r\le t\le T
\right\}.
\]
    \normalsize
    with finite dimensional distributions that are equal to those of the mean-zero gaussian process with covariances 
\begin{equation}
    Cov(\hat{\mart}_{f_a}^{\ap_j,j}(r_1,t_1),\hat{\mart}_{f_b}^{\ap_j,j}(r_2,t_2))= \int_0^{r_1\wedge r_2} \ar_j \left(\langle t_{t_1-u}f_at_{t_2-u}f_b,\pd_j\rangle -\langle t_{t_1-u}f_a, \pd_j\rangle \langle t_{t_2-u}f_b, \pd_j \rangle\right) du,
\end{equation}
\begin{align*}
  &Cov( \hat{\mart}_{f_a}^{\spr_j^k,c}(r_1,t_1),\hat{\mart}_{f_b}^{\spr_j^k,d}(r_2,t_2))\\&= \int_0^{r_1 \wedge r_2} \left(1_{\{c=d\}}\frac{p_c \langle t_{t_1-u}f_at_{t_2-u}f_b,\fl_c(u)\rangle}{\wmass(\boldsymbol{\flm}(u))}- \frac{p_cp_d\langle t_{t_1-u}f_a,\fl_c(u)\rangle \langle t_{t_2-u}f_b,\fl_d(u)\rangle  }{\wmass(\boldsymbol{\flm}(u))^2}\right)\frac{p_j\flm_j(u)}{\awmass(\boldsymbol{\flm}(u))}du,
\end{align*}
and all other covariances equal to zero.
Furthermore, if $\pd_j$ and $\fl_j(0)$ have H\"older continuous CDFs for each
$j\in\J$, the same construction applies to finite collections
$f_1,\ldots,f_n\in\mathscr C\cup\{1\}$.
\end{thm}
\begin{defi}
    \label{Udefdef}
    Let $\boldsymbol{\fl}$ be a fluid model solution with a nonzero initial condition $\boldsymbol{\fl}_0.$
    Assume also that $\pd_j,\fl_j(0)$ have H\"{o}lder continuous cdfs for each $j \in \J.$
    Then, take any collection of test functions
\(f_1,\ldots,f_n\in\mathscr C\cup\{1\}\) and a jointly specified realization of the
limiting driving noise processes
\[
\{\hat{\spr}_j^k\}_{j \in \J, k \in \KK},\quad \{\hat{\ap}_j\}_{j \in \J},\quad
\{F_{f_a}^{j}\}_{a\in[n]j\in\J},
\quad
\{\hat{\mart}_{f_a}^{\ap_j,j}\}_{a\in[n],i\in \J,},
\quad
\{\hat{\mart}_{f_a}^{\spr_j^k,i}\}_{a\in[n],i,j\in \J,k\in \KK},
\]
where the martingale field has the joint law specified in
Theorem~\ref{Ydefthm} and the other processes are as defined in Assumption \ref{assumptions} \eqref{initialconditionsassumption}-\eqref{cltforrenewalassumption}.
Furthermore, assume that the five processes above form a mutually independent collection of processes.
   Then we may jointly define the driving noise field for the test functions $f_1,...,f_n$ and $j \in \J$ as follows:
    \begin{align*}
    U_{f_a}^{j}(r,t)&:= F_{f_a}^{j}(t)+ \hat{{\mart}}^{\ap_j,j}_{f_a}(r,t)- \sum_{i=1}^J\sum_{k=1}^K \left(\hat{\mart}^{\spr_i^k,j}_{f_a}(r,t)-\int_0^r \frac{p_j \langle f_a, \fl_j(s)\rangle }{\awmass({\boldsymbol{\flm}}(s))} d\sum_{l=1}^J\frac{1}{\sr_l} \hat{\mart}^{\spr_i^k,l}_1(s,t)\right)\\
    &-\int_0^r \frac{p_j {{\flcdf}}^{j,c}_{f_a}(s,t-s)}{\awmass({\boldsymbol{\flm}}(s))} d\sum_{k=1}^K \sum_{l=1}^J\frac{1}{\sr_l} \hat{\spr}^{k}_l\left(\int_0^s \frac{\frac{p_l}{\sr_l}\flm_l(u)}{\awmass(\boldsymbol{\flm}(u))}du\right)\\
     &+\int_0^r\pdcdf^{j,c}_{f_a}(t-s) d\hat{\ap}_j(s), \hspace{5mm} 0 \leq r \leq t.
     \numberthis
     \label{ujmlimit}
\end{align*}
\end{defi}
\begin{rem}
    We remark that, while it may not seem immediate from the definition, for each $a \in [n], j \in [J],$ the process $U_{f_a}^j(r,t)$ is continuous as a multiparameter process.
    This is an immediate consequence of the fact that it is equal in distribution to the limit of a C-tight sequence of processes, which will be proved in Lemma \ref{Uconvlemma}.
\end{rem}
 \begin{thm}[Tightness of the ROS diffusion-scaled state descriptors]
\label{tightnessresult}
Under Assumption~\ref{assumptions}, suppose that there exist constants
$C<\infty$ and $\alpha>1/2$ such that, for each $j\in[J]$ and $m\in\mathbb N$,
the CDF of $\pd_j$ and the functions
\[
G^{j,m}(x):= E\!\left[\left\langle 1_{[0,x]}, \bar{\ssp}^m_j(0) \right\rangle\right],
\qquad
\tilde{G}^{j,m}(x):=
E\!\left[\left\langle 1_{[0,x]}, \bar{\ssp}^m_j(0) \right\rangle^2\right]
\]
are H\"older continuous with exponent $\alpha$ and H\"older constant at most $C$.
Then the sequence
\[
    \{\hat{\boldsymbol{\mathcal Z}}^m(\cdot)\}_{m=1}^{\infty}
\]
of diffusion-scaled ROS state descriptors is tight in
$D(\mathbb R_+,(\mathscr S')^J)$. Moreover, for every
$\boldsymbol f\in \boldsymbol{C}_1^b(\R_+)^J$, the sequence
\[
\left\{
\left(
    \langle f_1,\hat{{\ssp}}_1^m(\cdot)\rangle,
    \ldots,
    \langle f_J,\hat{{\ssp}}_J^m(\cdot)\rangle,
    \awmass(\hat{\boldsymbol{\tm}}^m(\cdot))
\right)
\right\}_{m=1}^{\infty}
\]
is $C$-tight in $D(\mathbb R_+,\mathbb R^{J+1})$.
\end{thm}
\begin{rem}
This regularity assumption is natural. For example, it is satisfied if the initial data
\(\ssp_j^m\) are constructed by first choosing a random total mass \(\tm_j^m\) with
finite second moment and then independently assigning i.i.d. initial remaining patience
times whose distribution function is H\"older continuous with exponent
\(\alpha>1/2\), provided the resulting H\"older bounds are uniform in \(m\).
\end{rem}
Next, we introduce a system of stochastic differential equations that will be related to the diffusion limit of $\{\hat{\boldsymbol{\mathcal Z}}^m(\cdot)\}_{m=1}^{\infty}.$
\begin{defi}[Diffusion Model Solution for Total Mass]
    Let $\boldsymbol{U}_{1}(r,t)$ be a realization of the multiparameter process defined in Definition \ref{Udefdef}.
    Then, any continuous, $\R^J$-valued, multi-index process $\{\boldsymbol{H}_{1}(r,t): 0 \leq r \leq t\}$ that satisfies
\begin{align*}
    {H}_{1}^{j}(r,t)& = U_{1}^{j}(r,t)- \int_0^r \frac{Kp_jH_{1}^{j}(s,t)}{\awmass(\boldsymbol{\flm}(s))}ds\\
    &+\int_0^r\frac{Kp_j {{\flcdf}}^{j,c}_1(s,t-s)}{\awmass(\boldsymbol{\flm}(s))^2}\sum_{i=1}^J \frac{p_i}{\sr_i}H_{1}^i(s,s)ds, \hspace{5mm} j \in \J,0 \leq r\leq t, 
    \numberthis \label{Llimitdef}
\end{align*}
is a diffusion model solution for total mass driven by the noise $\boldsymbol{U}_1(r,t)$.
\end{defi}
\begin{thm}[Uniqueness of Solutions for Total Mass]
\label{uniquenessoftotalmassthm}
For each realization of $\boldsymbol{U}_1(r,t)$ as defined in Definition \ref{Udefdef}, there is a unique continuous solution to \eqref{Llimitdef}. Furthermore, any subsequential limit of $\{\awmass(\hat{\boldsymbol{\tm}}^m(\cdot))\}_{m=1}^{\infty}$ is equal in distribution to the single-index process
\begin{equation}
    \awmass(\hat{\boldsymbol{\tm}}(t)):=\sum_{j=1}^J \frac{p_j}{\sr_j}H_{1}^j(t,t), \hspace{5mm} t\geq 0,
    \label{Lhatlimit}
\end{equation} where $\boldsymbol H_{1}$ is a
solution for total mass driven by any realization of $\boldsymbol{U}_1(r,t)$, which is unique in distribution.  
\end{thm}
\begin{defi}[Mass Transport Diffusion Model Solution]
    Let $f\in\mathscr{C}$. Let $(\boldsymbol{U}_f(r,t),\boldsymbol{U}_1(r,t))$ be a joint realization of the multiparameter process defined in Definition~\ref{Udefdef}.
    Then, any continuous, $\R^J$-valued, multi-index process $\{\boldsymbol{H}_f(r,t): 0 \leq r \leq t\}$ that satisfies
\begin{align*}
    {H}_f^{j}(r,t)& = U_f^{j}(r,t)- \int_0^r \frac{Kp_jH_f^{j}(s,t)}{\awmass(\boldsymbol{\flm}(s))}ds\\
    &+\int_0^r\frac{Kp_j {{\flcdf}}^{j,c}_f(s,t-s)}{\awmass(\boldsymbol{\flm}(s))^2}\awmass(\hat{\boldsymbol{\tm}}(s))ds, \hspace{5mm} j \in \J,0 \leq r\leq t 
    \numberthis \label{rjlimitdef}
\end{align*}
where $\awmass(\hat{\boldsymbol{\tm}}(s))$ is as defined in \eqref{Lhatlimit}, is a mass transport diffusion model solution for that realization of $(\boldsymbol{U}_f(r,t),\boldsymbol{U}_1(r,t))$.
\end{defi}

\begin{thm}[Uniqueness and characterization of projected ROS diffusion limits]
\label{lhatconvergencethm}
Under the assumptions of Theorem~\ref{tightnessresult}, for each
$f\in\mathscr C$ and realization of
$(\boldsymbol U_f,\boldsymbol U_1)$, equation \eqref{rjlimitdef} has a unique
continuous solution $\boldsymbol H_f$.
Furthermore, for every finite collection
$f_1,\ldots,f_n\in \mathbf{C}_b^1(\mathbb R_+)$, the joint law of any subsequential limit of
\[
\left\{
    \left(\langle f_1,\hat{\boldsymbol{\ssp}}^m(\cdot)\rangle,
    \ldots,
    \langle f_n,\hat{\boldsymbol{\ssp}}^m(\cdot)\rangle
,\awmass(\hat{\boldsymbol{\tm}}^m(\cdot)) 
\right)\right\}_{m=1}^{\infty}
\]
is uniquely determined. More precisely, any such subsequential limit is equal in
distribution to
\[
\left(
     \left(\boldsymbol{H}_{f_1-f_1(0)}(\cdot,\cdot)
     +f_1(0)\boldsymbol{H}_{1}(\cdot,\cdot)\right)\big|_{r=t},
    \ldots,
     \left(\boldsymbol{H}_{f_n-f_n(0)}(\cdot,\cdot)
     +f_n(0)\boldsymbol{H}_{1}(\cdot,\cdot)\right)\big|_{r=t}, \sum_{j=1}^J \frac{p_j}{\sr_j}H_1^j(\cdot,\cdot)\big|_{r=t}
\right),
\]

where the process $\boldsymbol{H}_1$ above is the unique continuous solution to \eqref{Llimitdef}, and the remaining processes are the unique continuous solutions to \eqref{rjlimitdef} with $\awmass(\hat{\boldsymbol{\tm}}(\cdot)):=\sum_{j=1}^J \frac{p_j}{\sr_j}H_1^j(\cdot,\cdot)\big|_{r=t},$ all driven by the jointly
specified noise
\[
\left(
\boldsymbol{U}_{f_1-f_1(0)},\ldots,
\boldsymbol{U}_{f_n-f_n(0)},\boldsymbol{U}_{1}
\right).
\]

\end{thm}
 We now define a diffusion model solution in a more classical sense, as a system of SDEs satisfied by the projections of $\hat{\boldsymbol{\ssp}}(\cdot)$ onto a class of functions that includes the Schwartz functions with the appropriate boundary condition.
 \begin{defi}[Diffusion Model Solution]
     Let $\mathscr{Z}(\cdot)\in D(\mathbb R_+,(\mathscr S')^J).$
     Then $\mathscr{Z}(\cdot)$ will be a diffusion model solution if, for each $\boldsymbol{f} \in \mathscr{S}^J\cap \mathscr{C}^J,$ $(\langle \boldsymbol{f}, \mathscr{Z}(\cdot) \rangle, \langle \boldsymbol{f}', \mathscr{Z}(\cdot) \rangle)$ has a.s. continuous sample paths that satisfy the following system of SDEs for $j \in \J, t\geq 0,$
\begin{align*}
    \langle f_j, \mathscr{Z}_j(t) \rangle &= \langle f_j, \mathscr{Z}_j(0) \rangle- \int_0^t \langle f_j', \mathscr{Z}_j(s) \rangle ds\\
    &+\int_0^t \sqrt{\ar_j(\langle f_j^2, \pd_j\rangle-\langle f_j,\pd_j\rangle^2)}dW_{\pat}^j(s)-\sum_{k=1}^K\sum_{l=1}^J\int_0^t  \bigg[\sqrt{D_{k,l}^{\boldsymbol{f}}(s)}\cdot d\boldsymbol{W}_{\kappa}^{k,l}(s) \bigg]_j\\
    &-\sum_{k=1}^K\sum_{l=1}^J\int_0^t \frac{p_j \langle f_j,\fl_j(s)\rangle }{\awmass(\boldsymbol{\flm}(s))}\sigma_{V,l}\sqrt{\frac{p_l \flm_l(s)}{\awmass(\boldsymbol{\flm}(s))}}dW_{\ser}^{k,l}(s)\\
    &+\int_0^t\ar_j^{3/2}\sigma_{A,j}\langle f_j, \pd_j\rangle dW_{\inta}^j(s)\\
    &-\int_0^t \frac{Kp_j}{\awmass(\boldsymbol{\flm}(s))}\left(\langle f_j, \mathscr{Z}_j(s)\rangle -\frac{\langle f_j, \fl_j(s)\rangle}{\awmass(\boldsymbol{\flm}(s)) }\awmass(\hat{\boldsymbol{\tm}}(s))\right)ds.         \numberthis
        \label{diffusionlimitequation}
\end{align*}
 In the above equation, $W_{\pat}^j$, $\boldsymbol W_{\kappa}^{k,l}$,
$W_{\ser}^{k,l}$, and $W_{\inta}^j$, $j,l\in[J]$, $k\in[K]$, are independent
standard Brownian motions, with $\boldsymbol W_{\kappa}^{k,l}$ being
$J$-dimensional. Furthermore, $\awmass(\hat{\boldsymbol{\tm}}(\cdot))$ is as defined in Theorem \ref{uniquenessoftotalmassthm}. Furthermore, $D_{k,j}^{\boldsymbol{f}}(\cdot),$ $j \in \J,k \in [K]$ is the matrix with values 
\begin{align*}
   &(D^{\boldsymbol{f}}_{k,j})_{a,b}(\cdot)=\\
   &\left(
   1_{\{a=b\}}\frac{p_a\langle f_a^2, {{\fl}}_a(\cdot)\rangle}{\wmass({\boldsymbol{\flm}}(\cdot))}
   -\frac{p_a\langle f_a, \fl_a(\cdot)\rangle p_b\langle f_b, \fl_b(\cdot)\rangle}{\wmass({\boldsymbol{\flm}}(\cdot))^2}
   \right)\frac{p_j\flm_j(\cdot)}{\awmass(\boldsymbol{\flm}(\cdot))}\\
   &-\sum_{n=1}^J
   \frac{p_a \langle f_a, \fl_a(\cdot)\rangle }{\awmass({\boldsymbol{\flm}}(\cdot))}\frac{1}{\sr_n}
   \left(
   1_{\{n=b\}}\frac{p_b\langle f_b, {\boldsymbol{\fl}}_b(\cdot)\rangle}{\wmass({\boldsymbol{\flm}}(\cdot))}
   -\frac{p_n\flm_n(\cdot)p_b\langle f_b, \fl_b(\cdot)\rangle}{\wmass({\boldsymbol{\flm}}(\cdot))^2}
   \right)\frac{p_j\flm_j(\cdot)}{\awmass(\boldsymbol{\flm}(\cdot))}\\
   &-\sum_{n=1}^J
   \frac{p_b \langle f_b, \fl_b(\cdot)\rangle }{\awmass({\boldsymbol{\flm}}(\cdot))}\frac{1}{\sr_n}
   \left(
   1_{\{n=a\}}\frac{p_a\langle f_a, {\boldsymbol{\fl}}_a(\cdot)\rangle}{\wmass({\boldsymbol{\flm}}(\cdot))}
   -\frac{p_n\flm_n(\cdot)p_a\langle f_a, \fl_a(\cdot)\rangle}{\wmass({\boldsymbol{\flm}}(\cdot))^2}
   \right)\frac{p_j\flm_j(\cdot)}{\awmass(\boldsymbol{\flm}(\cdot))}\\
   &+ \sum_{n=1}^J \sum_{x=1}^J
   \frac{p_a \langle f_a, \fl_a(\cdot)\rangle }{\awmass({\boldsymbol{\flm}}(\cdot))}\frac{1}{\sr_n}
   \frac{p_b \langle f_b, \fl_b(\cdot)\rangle }{\awmass({\boldsymbol{\flm}}(\cdot))}\frac{1}{\sr_x}
   \left(
   1_{\{n=x\}}\frac{p_n\flm_n(\cdot)}{\wmass({\boldsymbol{\flm}}(\cdot))}
   -\frac{p_n\flm_n(\cdot)p_x\flm_x(\cdot)}{\wmass({\boldsymbol{\flm}}(\cdot))^2}
   \right)\frac{p_j\flm_j(\cdot)}{\awmass(\boldsymbol{\flm}(\cdot))},
   \numberthis
\end{align*} 
and the matrix square-root above is the unique symmetric square root. 
 \end{defi}
 \begin{rem}
     In the above, instead of simply numbering the independent Brownian motions, we label them using the stochastic primitives that generated them.
In particular, it may be intuitive to the reader to think of the integrator $dW_{\pat}^j$ (resp. $\bigg[\sqrt{D_{k,l}} dW_{\kappa}^{k,l}(s)\bigg]_j,dW_{\ser}^{k,l},dW_{\inta}^j$) as the contribution to the limiting fluctuation of the randomness introduced by the sequence $\{\pat_i^{j}\}_{i=1}^{\infty}$ (resp, $\{\ur_i^{k,l}\}_{i=1}^{\infty},\{\ser_i^{k,l}\}_{i=1}^{\infty},\{\inta_i^{j}\}_{i=1}^{\infty}$).
 \end{rem}
 \begin{thm}     \label{diffusionapproximationresult} Let $\{\hat{\boldsymbol{\ssp}}^m(\cdot)\}_{m=1}^{\infty}$ be a sequence of diffusion-scaled state descriptors as described in \S \ref{modeldescriptionsect} and \S \ref{sequenceofmodelssection} that satisfy Assumption \ref{assumptions}. Then any subsequential limit of $\{\hat{\boldsymbol{\ssp}}^m(\cdot)\}_{m=1}^{\infty},$ $\hat{\boldsymbol{\ssp}}(\cdot),$ is a diffusion model solution.
 \end{thm}

In \S\ref{pathtoadiffusionapproxsection}, we develop the main tools and methodology for obtaining diffusion approximations. In \S\ref{renewaldrivensystemsect}, we show how the sequence of models described in \S\ref{modeldescriptionsect}–\ref{sequenceofmodelssection} can be decomposed into a “good sequence of diffusion-scaled renewal-driven systems,” making them amenable to our central limit theorem for such systems (Theorem \ref{diffusionclt}). In \S\ref{proofoftightnesssect}, we establish tightness (Theorem \ref{tightnessresult}) and prove Theorem \ref{Ydefthm}. Finally, in \S\ref{proofofdiffusiontheoremsection}, we obtain the limit of $\awmass(\hat{\boldsymbol{\tm}}^m(\cdot))$ (Theorems \ref{uniquenessoftotalmassthm} and \ref{lhatconvergencethm}), and we apply Theorem \ref{diffusionclt} to the Random Order of Service model to derive Theorem \ref{diffusionapproximationresult}.

\section{Path to a Diffusion Approximation}
\label{pathtoadiffusionapproxsection}
\subsection{Outline of the Method}
\label{outlineofmethod}

The core idea of our methodology originates from a simple yet powerful observation: in many queueing models, randomness enters the system only through events occurring at the jump times of time-changed renewal processes.
We illustrate this now.
Let $X(\cdot)$ denote the state descriptor of a given queueing model, and let $E_l(g_l(\cdot))$, for $l = 1,2,\dots,N$, denote time-changed renewal processes driving the system (e.g., arrival processes, service processes, etc.). Let $\{\tau_i^l\}_{i\geq 1}$ be the jump times of $E_l(g_l(\cdot))$, and let $\{x_i^l\}$ be an i.i.d.\ sequence of random variables representing information received at each jump time (e.g., $x_i^l$ could be the $i$th service time for some class of jobs).

At time $\tau_i^l$, the system updates via a measurable function $f_l$ of the time, $\tau_i^l,$ the state of the system just before that time, $X(\tau_i^l-)$, and the new information received at that time, $x_i^l$:
\[
f_l(\tau_i^l, x_i^l, X(\tau_i^l-)).
\]
The cumulative change that occurs at jump times of the process $E_l(g_l(\cdot))$ up to time $t$ is:
\begin{equation}
\Delta^l(t) := \sum_{i=1}^{E_l(g_l(t))} f_l(\tau_i^l, x_i^l, X(\tau_i^l-)). \label{deltaintro}
\end{equation}
Define $\phi_l(t, x) := {E}[f_l(t, x_1^l, x)].$
Then $\phi_l(t,x)$ represents the expected change to the system that would occur from a jump in $E_l(g_l(\cdot))$ at time $t$ if the system was in state $x,$ averaged over the noise variable $x_1^l$.
We can decompose \eqref{deltaintro} into a martingale term
\begin{equation}
    \sum_{i=1}^{E_l(g_l(t))} \left( f_l(\tau_i^l, x_i^l, X(\tau_i^l-)) - \phi_l(\tau_i^l, X(\tau_i^l-)) \right), \label{martingalepartintro}
    \end{equation}
and a compensator
\begin{equation}
    \sum_{i=1}^{E_l(g_l(t))} \phi_l(\tau_i^l, X(\tau_i^l-)). \label{compensatorpartintro}
    \end{equation}
Under mild assumptions, Proposition \ref{martingaledecompositionprop} shows that \eqref{martingalepartintro} is a martingale. 
We thus obtain a prelimit equation:
\begin{equation}
X(t) = X(0) + \int_0^t b(s, X(s)) \, ds + \sum_l \Delta^l(t), \qquad t \ge 0, \label{martingaledecompeqnintro}
\end{equation}
where $b(s, X(s))$ accounts for any deterministic dynamics.
(For example, in the random order of service model we'll discuss, $b(s,X(s))$ will represent the change to the system due to all patience times decreasing at rate $1.$)
Under functional law of large numbers (FLLN) scaling, if $\bar{X}(\cdot)$ and $\bar{E}_l(\bar{g}_l(\cdot))$ are subsequential limits of $X(\cdot)$ and $E_l({g}_l(\cdot))$, we expect the compensator (\eqref{compensatorpartintro}, rescaled) to converge to
\[
\int_0^t \phi(s, \bar{X}(s)) \, d\bar{E}_l(\bar{g}_l(s)),  \qquad t \ge 0, 
\]
and the martingale term (\eqref{martingalepartintro}, rescaled), which is centered, to vanish. Therefore, the fluid limit should satisfy:
\begin{equation}
\bar{X}(t) = \bar{X}(0) + \int_0^t b(s, \bar{X}(s)) \, ds + \sum_{l} \int_0^t \phi(s, \bar{X}(s)) \, d\bar{E}_l(\bar{g}_l(s)),  \qquad t \ge 0.
\label{easyfluidintro}
\end{equation}
If this equation has a unique solution, we then center around that limit and diffusion-scale to study fluctuations (the diffusion limit) and derive a system of SDEs using the following steps:

\begin{enumerate}
    \item Decompose each diffusion-scaled renewal-driven term as in \eqref{martingaletermtoyexample}–\eqref{integralagainstrenewaldifftoy}, expressing the diffusion-scaled process as a \emph{renewal-driven system} (Definition \ref{renewaldrivensystemdef}).
    \item Use Lemma~\ref{ornsteinuhlenbecktightnesscondition} to establish tightness of the diffusion-scaled sequence.
    \item Apply our CLT for renewal-driven systems (Theorem \ref{diffusionclt}) to derive the SDEs satisfied by subsequential diffusion limits.
\end{enumerate}

We now describe steps 1-3 above in more detail.
Beginning with the decomposition of \eqref{deltaintro}, \eqref{martingalepartintro}-\eqref{compensatorpartintro},
\begin{align*}
\Delta^l(t)
&= \sum_{i=1}^{E_l(g_l(t))} \left(f_l(\tau_i^l, x_i^l, X(\tau_i^l-)) - \phi_l(\tau_i^l, X(\tau_i^l-))\right)
+ \sum_{i=1}^{E_l(g_l(t))} \phi_l(\tau_i^l, X(\tau_i^l-)) \\
&= \sum_{i=1}^{E_l(g_l(t))} \left(f_l(\tau_i^l, x_i^l, X(\tau_i^l-)) - \phi_l(\tau_i^l, X(\tau_i^l-))\right)
+ \int_0^t \phi_l(s, X(s-)) \, dE_l(g_l(s)). \numberthis \label{toydecomp}
\end{align*}
Recall that $\phi_l(t, x) := {E}[f_l(t, x_1^l, x)]$ is the expected change at time $t$, averaged over the noise variable $x_1^l$. The integral term follows from the definition of the Lebesgue–Stieltjes integral. Under appropriate conditions, the first sum is a martingale (see Proposition~\ref{martingaledecompositionprop}).

This decomposition provides a convenient path to characterize both fluid and diffusion limits. Consider the following rescalings for $t,x,y \geq 0, m \in \N$:
\[
\bar{X}^m(\cdot) := \frac{1}{m} X^m(m\cdot), \quad
f_l^m(t,x,y) := f_l\left(\frac{t}{m}, x, \frac{y}{m}\right), \quad
\bar{g}_l^m := \frac{1}{m} g_l(m\cdot),
\]
\[
\bar{E}_l^m(t) := \frac{1}{m} E_l(m t), \quad
\bar{\Delta}^{l,m}(t) := \frac{1}{m} \Delta^{l,m}(m t).
\]
Suppose $\bar{X}^m(\cdot) \Rightarrow \bar{X}(\cdot)$ and $\bar{g}_l^m(\cdot) \rightarrow \bar{g}_l(\cdot)$. Then the centered (fluid-scaled) martingale part vanishes in the limit, and the fluid limit of $\bar{\Delta}^{l,m}$ is given by:
\[
\bar{\Delta}^l(t) = \int_0^t \phi_l(s, \bar{X}(s-)) \, d\bar{E}_l(\bar{g}_l(s)), \qquad t \geq 0,
\]
where $\bar{E}_l(t) = \mu_l t$ and $\mu_l$ is the mean rate of the original renewal process $E_l(\cdot)$. 
We then center each $\Delta^l(\cdot)$ around this fluid limit and study the diffusion-scaled fluctuations, $\hat{\Delta}^{l,m}(t)$:
\begin{align}
\hat{\Delta}^{l,m}(t) &:= \sqrt{m} \left( \bar{\Delta}^{l,m}(t) - \bar{\Delta}^l(t) \right) \nonumber \\
&= \frac{1}{\sqrt{m}} \sum_{i=1}^{m \bar{E}_l(\bar{g}_l^m(t))} \left(
f_l\left(\frac{\tau_i^l}{m}, x_i^l, \bar{X}^m\left(\frac{\tau_i^l}{m}-\right)\right)
- \phi_l\left(\frac{\tau_i^l}{m}, \bar{X}^m\left(\frac{\tau_i^l}{m}-\right)\right) \right) \label{martingaletermtoyexample} \\
&\quad + \int_0^t \hat{\phi}_l^m(s, \bar{X}^m(s-)) \, d\bar{E}_l^m(\bar{g}_l^m(s)) \label{phihatmtoy} \\
&\quad + \int_0^t \phi_l(s, \bar{X}(s-)) \, d\hat{E}_l^m(\bar{g}_l^m(s)) \label{errortermtoy} \\
&\quad + \int_0^t \phi_l(s, \bar{X}(s-)) \, d(\mu_l \hat{g}_l^m(s)), \label{integralagainstrenewaldifftoy}
\end{align}
where $\hat{\phi}^m_l(\cdot, \bar{X}^m(\cdot-)):=\sqrt{m}(\phi_l(\cdot, \bar{X}^m(\cdot-))-\phi_l(\cdot,\bar{X}(\cdot-))),$ and $\hat{g}^m_l(\cdot) = \sqrt{m}(\bar{g}^m_l(\cdot)-\bar{g}_l(\cdot)).$
\begin{rem}
\label{timechangeremark}
A key subtlety lies in the time change: the presence of the term \eqref{integralagainstrenewaldifftoy} requires handling integrals against $\hat{g}_l^m(\cdot)$, which introduces nontrivial complexity. As we will see in the ROS queue application, any diffusion-scaled time change appearing in the equation must also be decomposed if one wishes to apply the central limit theorem for renewal-driven systems (Theorem~\ref{diffusionclt}) and obtain the correct form of the limiting SDEs.
\end{rem}

Ultimately, this leads to a diffusion-scaled prelimit equation:
\[
\hat{X}^m(t) = \hat{X}^m(0) + \int_0^t b(s, \hat{X}^m(s)) \, ds + \sum_l \hat{\Delta}^{l,m}(t),
\]
where $b(s, \hat{X}^m(s))$ captures deterministic system dynamics, and each $\hat{\Delta}^{l,m}(t)$ is decomposed as above. This form enables the application of Theorem~\ref{diffusionclt} to characterize the limiting process as a solution to a system of SDEs.
\subsection{An Example: The Processor Sharing Queue with Reneging}
\label{gromollkrukexamplesect}
To illustrate the method, we derive the fluid and diffusion equations for the model described in \cite{gromollkruk}, but with reneging.
This is only meant as an illustrative example, as proofs of convergence, uniqueness of solutions, and tightness are left for future work.
As discussed in the introduction, we examine a processor sharing queue with impatience.
Let $\{(v_i, l_i)\}_{i=1}^\infty$, with $(v_i, l_i) \sim \vartheta$, be the i.i.d.\ sequence of initial service and patience times of jobs in the queue, and let $E(t)$ be the exogenous arrival process to the queue.
Define the measure-valued state descriptor $\mathscr{Z}(t)$ on $\mathbb{R}_+^2$ to place a Dirac mass at the point (remaining service time, remaining patience time) for each job in the system.

The contribution to $\mathscr{Z}(t)$ that occurs at jump times of $E(\cdot)$ is given by:
\[
\Delta(t) = \sum_{i=1}^{E(t)} \delta_{(v_i, l_i)}, \qquad t\geq 0.
\]
Between jumps, each point mass evolves deterministically at rate $-\left(1 / \langle 1, \mathscr{Z}(t) \rangle, 1\right)$, reflecting decreases in remaining service and patience times.

Substituting these dynamics into the general form of the evolution equation from \S\ref{outlineofmethod}, \eqref{martingaledecompeqnintro}, we obtain the prelimit equation:
\[
\langle f, \mathscr{Z}(t) \rangle = \langle f, \mathscr{Z}(0) \rangle + \langle f, \Delta(t) \rangle - \int_0^t 1_{\{\langle 1, \mathscr{Z}(s-) \rangle\neq 0\}} \frac{\langle f_x, \mathscr{Z}(s-) \rangle}{\langle 1, \mathscr{Z}(s-) \rangle}ds -\int_0^t \langle f_y, \mathscr{Z}(s-) \rangle  ds,
\]
for each $t\geq 0$ and $f \in \mathbf{C}_b^1(\mathbb{R}_+^2)$ such that $f(0, y) \equiv 0$ and $f(x, 0) \equiv 0$.
The last two terms above reflect the deterministic dynamics, movement of each particle at rate $-\left(1 / \langle 1, \mathscr{Z}(t) \rangle, 1\right)$.
For the rigorous justification of the form of the last terms above, see, e.g., the proof of Lemma 7.2 in \cite{loeserwilliams}.

Passing to the fluid limit $\bar{\mathscr{Z}}(\cdot)$ as in \eqref{easyfluidintro}, we obtain for $0 \leq t \leq t^*,$ where $t^*:= \inf \{{t \geq 0: \bar{\mathscr{Z}}(t) =\boldsymbol{0}}\},$
\[
\langle f, \bar{\mathscr{Z}}(t) \rangle = \langle f, \bar{\mathscr{Z}}(0) \rangle + \ar \langle f, \vartheta \rangle t - \int_0^t  \frac{\langle f_x, \bar{\mathscr{Z}}(s) \rangle}{\langle 1, \bar{\mathscr{Z}}(s) \rangle} ds -\int_0^t \langle f_y, \bar{\mathscr{Z}}(s) \rangle  ds,
\]
where $\alpha$ is the rate of the renewal process $E(\cdot).$ 
We note that, as in \cite{gromollkruk}, we have restricted to the set of times up until the fluid model solution hits zero.
One can check that this fluid equation is consistent with that found in \cite{gromollkruk} (see, e.g., the proof of Lemma 4.3 in \cite{gromollpuhawilliams}), but incorporates reneging.

By applying the decomposition described in \eqref{martingaletermtoyexample}–\eqref{integralagainstrenewaldifftoy}, writing the diffusion-scaled system as a ``good sequence of renewal-driven systems," Theorem~\ref{diffusionclt} yields a stochastic differential equation satisfied by subsequential limits of the diffusion-scaled system under mild conditions: 
\begin{align*}
    \langle f, \hat{\mathscr{Z}}(t)\rangle &= \langle f, \hat{\mathscr{Z}}(0)\rangle + \int_0^t \sqrt{\ar (\langle f, \vartheta \rangle +\langle f^2,\vartheta \rangle - \langle f,\vartheta \rangle^2) }dW(s)\\
    &- \int_0^t \left(\frac{\langle f_x, \hat{\mathscr{Z}}(s) \rangle}{\langle 1, \bar{\mathscr{Z}}(s) \rangle}-\frac{\langle f_x, \bar{\mathscr{Z}}(s)\rangle}{\langle 1, \bar{\mathscr{Z}}(s) \rangle^2  }\langle 1, \hat{\mathscr{Z}}(s)\rangle\right)ds -\int_0^t \langle f_y, \hat{\mathscr{Z}}(s) \rangle  ds, \qquad 0 \leq t \leq t^*.
\end{align*}
In the above, $W(s)$ is a brownian motion, and the form of the quadratic variation displayed in the second term of the right hand side above can be calculated using Corollary \ref{martingaleconvcor} once the martingale decomposition is complete. 
\subsection{Relevant Definitions}
\label{relevantdefinitionssect}

In this section, we introduce definitions and notation related to diffusion-scaled renewal processes and martingales. These tools are foundational for our analysis of renewal-driven queueing systems.

We begin by describing a decomposition of a (delayed) renewal process $E(\cdot)$ into a martingale part $O(\cdot)$ and a remainder term $R(\cdot)$. This decomposition is based on Theorem 2.1 of \cite{martingaledecomppaper} (see equations (2.8) and (1.3)), although we will re-establish the martingale property using Proposition~\ref{martingaledecompositionprop}, which is more aligned with our framework.

Let $\{x_l\}_{l=1}^{\infty}$ be an i.i.d.\ sequence of interarrival times, with an optional initial delay $x_0$ (set to zero in the non-delayed case). Then the renewal process satisfies:
\[
E(t) = O(t) + R(t), \qquad t \geq 0,
\]
where
\begin{equation}
O(t) := \sum_{l=1}^{E(t)}\left(1 - \frac{x_l}{{E}[x_l]}\right), \qquad t \geq 0, \label{martingalepartofE}
\end{equation}
\begin{equation}
R(t) := \frac{1}{{E}[x_l]}\left(r(t) + t - x_0\right), \qquad t \geq 0, \label{remainderpartofE}
\end{equation}
and
\begin{equation}
r(t) := x_0 + \sum_{l=1}^{E(t)} x_l - t, \qquad t \geq 0. \label{rtdef}
\end{equation}

Here, $r(t)$ is the residual time until the next jump after $t$. The process $E(\cdot)$ is said to be delayed if $x_0 > 0$.

\begin{rem}
\label{delayednote}
In \cite{martingaledecomppaper}, any renewal process $E(t)$ with $E(0)=0$ is treated as delayed. Therefore, in our service process martingales, we treat $\ser_1^{k,j,m}$ as the initial delay $x_0$, and we sum from $l = 2$ to $\spr_j^k(t) + 1$ in order to remain consistent with that convention.
\end{rem}

\begin{defi}[Fluid-scaled Renewal Process]
\label{fluidscaledrenewaldef}
A process $\bar{E}^m(\cdot)$ is called a \emph{fluid-scaled (delayed) renewal process} for the parameter $m>0$ if
\[
\bar{E}^m(t) := \frac{1}{m} E(m t), \qquad t \geq 0,
\]
for some (delayed) renewal process $E(\cdot)$.
\end{defi}

\begin{defi}[Diffusion-scaled Renewal Process]
\label{diffscaledrenewaldef}
A process $\hat{E}^m(\cdot)$ is called a \emph{diffusion-scaled (delayed) renewal process} for the parameter $m>0$ if
\[
\hat{E}^m(t) := \frac{1}{\sqrt{m}} \left(E(m t) - \mu t\right), \qquad t \geq 0,
\]
where $E(\cdot)$ is a (delayed) renewal process with rate $\mu = {E}[x_1]^{-1}$.
\end{defi}

\begin{defi}[Fluid-scaled Martingale]
\label{fluidscaledmartdef}
A process $\bar{Y}^m(\cdot)$ is called a \emph{fluid-scaled martingale} for the parameter $m>0$ if
\[
\bar{Y}^m(t) := \frac{1}{m} Y(m t),
\]
for some martingale $Y(\cdot)$.
\end{defi}

\begin{defi}[Diffusion-scaled Martingale]
\label{diffscaledmartdef}
A process $\hat{Y}^m(\cdot)$ is called a \emph{diffusion-scaled martingale} for the parameter $m>0$ if
\[
\hat{Y}^m(t) := \frac{1}{\sqrt{m}} Y(m t),
\]
for some martingale $Y(\cdot)$.
\end{defi}

\begin{defi}[Time Change]
A process $g(\cdot) \in D([0,\infty), \mathbb{R}_+)$ is called a \emph{time change} if $g(0) = 0$ and $g(\cdot)$ is non-decreasing. A process is said to be \emph{time-changed} if its time index has been replaced by a time change.
\end{defi}

We now define the corresponding scaled versions of $O(\cdot)$ and $R(\cdot)$:

\begin{align}
\bar{O}^m(t) &:= \frac{1}{m} \sum_{l=1}^{m \bar{E}^m(t)}\left(1 - \frac{x_l}{{E}[x_1]}\right), \label{martingalepartofEbar} \\
\bar{R}^m(t) &:= \frac{1}{{E}[x_1]} \cdot \frac{1}{m} \left(r(m t) + m t - x_0\right), \label{remainderpartofEbar} \\
\hat{O}^m(t) &:= \frac{1}{\sqrt{m}} \sum_{l=1}^{m \bar{E}^m(t)}\left(1 - \frac{x_l}{{E}[x_1]}\right), \label{martingalepartofEhat} \\
\hat{R}^m(t) &:= \frac{1}{{E}[x_1]} \cdot \frac{1}{\sqrt{m}} \left(r(m t) - x_0\right). \label{remainderpartofEhat}
\end{align}

Later, in Lemma~\ref{martingalesaremartingaleslem}, we will verify—using Proposition~\ref{martingaledecompositionprop}—that even with a time change $\bar{g}^m(\cdot)$, the process $\hat{O}^m(\bar{g}^m(\cdot))$ remains a martingale with respect to an appropriate filtration.

\begin{defi}[Good Sequence of Diffusion-Scaled Renewal Driven Systems]
\label{renewaldrivensystemdef}
Let $(\Omega^m, \mathscr{F}^m, \{\mathscr{F}_t^m\}, {P}^m)$ be a sequence of filtered probability spaces satisfying the usual conditions. A sequence
\begin{align*}
(\hat{\boldsymbol{X}}^m(\cdot),\hat{\boldsymbol{J}}^m(\cdot),\boldsymbol{Y}^m_1(\cdot), \dots, \boldsymbol{Y}^m_A(\cdot), 
\boldsymbol{b}^1(\cdot), \dots, \boldsymbol{b}^A(\cdot), 
\boldsymbol{h}^{1,m}(\cdot), \dots, \boldsymbol{h}^{A,m}(\cdot), \\
E^m_1(g^m_1(\cdot)), \dots, E^m_A(g^m_A(\cdot)), 
\boldsymbol{r}^1, \dots, \boldsymbol{r}^A, c_1^m, \dots, c_A^m)
\end{align*}
for some $A \in \N$ is called a \emph{good sequence of diffusion-scaled renewal-driven systems} if:
\begin{itemize}
    \item The processes are $\mathscr{F}^m_t$-adapted and belong to $D(\mathbb{R}_+, (\mathbb{R}^d)^{2 + 3A}) \times C_b(\mathbb{R},\mathbb{R})^{d A} \times \mathbb{R}^A$;
    \item The system satisfies the stochastic integral equation
    \begin{align}
    \hat{\boldsymbol{X}}^m(t) &= \hat{\boldsymbol{X}}^m(0) 
    + \sum_{i=1}^A \hat{\boldsymbol{Y}}^m_i(t) 
    + \sum_{i=1}^A \int_0^t \boldsymbol{b}^i(s)\, dc_i^m \hat{E}_i^m(\bar{g}_i^m(s)) 
    + \sum_{i=1}^A \int_0^t \boldsymbol{r}^i(\hat{\boldsymbol{X}}^m(s))\, ds \nonumber \\
    &\quad + \sum_{i=1}^A \int_0^t \boldsymbol{h}^{i,m}(s) \hat{\boldsymbol{X}}^m(s-) \, d\bar{E}_i^m(\bar{g}_i^m(s)) 
    + \hat{\boldsymbol{J}}^m(t); \label{renewaldrivenequation}
    \end{align}
    \item Each $E^m_i(\cdot)$ is a (possibly delayed) renewal process with rate $\iota_i^m$, and the $E^m_i(\cdot)$ are mutually independent;
    \item The processes $\bar{E}^m_i(\bar{g}^m_i(\cdot))$ are $\mathscr{F}_t^m$-predictable, and their jump times are distinct across $i$;
    \item At the $n$th jump time $\tau_n^{i,m}$ of $\bar{E}_i^m(\bar{g}_i^m(\cdot))$, the interevent time $x_n^{i,m}$ is independent of $\mathscr{F}_{\tau_n^{i,m}-}$ but measurable with respect to $\mathscr{F}_{\tau_n^{i,m}}$;
    \item The martingales $\hat{\boldsymbol{Y}}^m_i(\cdot)$ and $\hat{O}^m_i(\cdot)$ are pure-jump processes that only jump at the jump times of $\bar{E}_i^m(\bar{g}_i^m(\cdot))$, and the increments of $\hat{\boldsymbol{Y}}^m_i(\cdot)$ are independent of $x_n^{i,m}$.
\end{itemize}
\end{defi}
In order to obtain convergence for a ``Good Sequence of Diffusion-Scaled Renewal Driven Systems," the following natural assumptions prove useful.
\begin{assumption}
    \label{assumptionsgeneral}
    Let   $ (\hat{\boldsymbol{X}}^m(\cdot)$ $,\hat{\boldsymbol{J}}^m(\cdot),$ ${\boldsymbol{Y}}^m_1(\cdot),...,{\boldsymbol{Y}}^m_A(\cdot),$ $\boldsymbol{b}^1(\cdot),...,\boldsymbol{b}^A(\cdot),$ $\boldsymbol{h}^{1,m}(\cdot),...,\boldsymbol{h}^{A,m}(\cdot),$ ${{E}}^m_1({g}^m_1(\cdot)),...,{{E}}^m_A({g}^m_A(\cdot)),$ $ \boldsymbol{r}^{1},...,\boldsymbol{r}^{A} $,
    $c_1^m,...,c^m_A)$ in $D([0,\infty), (\R^d)^{2+3A}\times\R^{A}) )\times\left({C_b(\R,\R)}^{d}\right)^{A}\times \R^A$ for some $A \in \N$ be a ``good sequence of diffusion-scaled renewal-driven systems" whose renewal processes have rates $(\iota_1^m,..,\iota_A^m)$.
    We make the following further assumptions
    \begin{enumerate}
         \item \label{martingalecltsecondassumption}
         For each sequence of vector-valued martingales $\{\hat{\boldsymbol{Y}}_i^m(\cdot)\}_{m=1}^{\infty},$ the associated quadratic covariation matrix $C^m_i(\cdot)$ converges in distribution as $m \rightarrow \infty$ to some continuous deterministic matrix-valued function $C_i(\cdot)$, where each of the components can be written $c_{i}^{j,l}(\cdot) = \int_0^{\cdot} d_{i}^{j,l}(s)d(\gamma_i(s))$ for some deterministic matrix-valued function $D_{i}$ and deterministic, real-valued continuous function $\gamma_i$. Furthermore, for $T>0,$ $\lim_{m \rightarrow \infty}E[\sup_{t \leq T}|\hat{\boldsymbol{Y}}_i^m(t)-\hat{\boldsymbol{Y}}_i^m(t-)|^2]=0$ and $\lim_{m\rightarrow \infty}E[\sup_{t\leq T}|C_{i,j}^m(t)-C_{i,j}^m(t-)|]=0.$
          \item The fluid-scaled time changes $(\bar{g}^m_1(\cdot),...,\bar{g}^m_A(\cdot))= (\frac{1}{m}{g}^m_1(m\cdot),...,\frac{1}{m}{g}^m_A(m\cdot))$ converge in distribution to deterministic, continuously differentiable functions $(\bar{g}_1(\cdot),..,\bar{g}_A(\cdot))$ as $m \rightarrow \infty$ and the constants $(c_1^m,...,c_A^m)$ converge as $m \rightarrow \infty$ to some $(c_1,...,c_A)\in \R^A.$
         \item \label{locallyfinitevarassumption} The functions $(\boldsymbol{b}^1(\cdot),...,\boldsymbol{b}^A(\cdot))$ are deterministic and of locally finite variation.
         \item \label{countingprocessesintegrableassumption} For each $i \in [A]$ and $t \geq 0$,  $\sup_mE\!\left[|\bar{E}^m_i(\bar{g}^m_i(t))|\right]<\infty.$
        \item The processes $(E^m_1(\cdot),...,E^m_A(\cdot))$ are such that the functional law of large numbers for renewal processes holds, i.e., $(\bar{E}^m_1(\cdot),...,\bar{E}^m_A(\cdot))\Rightarrow (\iota_1 (\cdot),...,\iota_A (\cdot))$ (see, e.g. \cite{chenandyao}, Theorem 5.10) where $(\iota_1 ,...,\iota_A )$ are the limiting rates of the renewal processes $(\bar{E}^m_1(\cdot),...,\bar{E}^m_A(\cdot)).$ 
        \item \label{fcltdiffassumption} The function $\boldsymbol{h}^{i,m}(\cdot)$ converges in distribution to a some path $\boldsymbol{h}^i(\cdot)\in D(\R_+, \R^d)$ for each $i \in [A].$
        \item The processes $(\hat{E}^m_1(\cdot),...,\hat{E}^m_A(\cdot))$ are such that the functional central limit theorem for renewal processes holds, i.e., $(\hat{E}^m_1(\cdot),...,\hat{E}^m_A(\cdot))\Rightarrow (\iota_1\sigma_1W_1(\iota_1\cdot),...,\iota_A \sigma_A W_A(\iota_A \cdot))$ for $(W_1(\cdot),...,W_A(\cdot)),$ a vector of independent Brownian Motions (see, e.g. \cite{chenandyao}, Theorem 5.11) where $(\iota_1 ,...,\iota_A) $ are the limiting rates of the renewal processes $({E}^m_1(\cdot),...,{E}^m_A(\cdot))$ and $(\sigma_1 ,...,\sigma_A) $ are the limiting standard deviations of the interevent times of the renewal processes $({E}^m_1(\cdot),...,{E}^m_A(\cdot))$.
        \item Furthermore, if $\{x_n^{i,m}\}_{n=1}^{\infty}$ are the interevent times for the renewal process $E^m_i(\cdot)$, we assume that $\sup_{m \in \N} E[|x_1^{i,m}|^3]<\infty.$ \label{thirdmomentbound}
    \end{enumerate}
\end{assumption}

\subsection{Our Toolbox}
\label{toolssection}
In this section we introduce three main results that will be used to obtain a diffusion approximation for our model. The first formalizes the martingale decompositions alluded to in the toy example.
The second provides a condition under which tightness can be obtained for an equation such as \eqref{renewaldrivenequation}. The last provides a stochastic differential equation that will be satisfied by limits of a system written in the form of \eqref{renewaldrivenequation} under mild assumptions.
These results will be proved in \S \ref{proofsoftoolboxresults}.
\begin{prop}
\label{martingaledecompositionprop}
    Let $(\Omega, \mathscr{F},\mathscr{F}_t,P)$ be a filtered probability space. Let $E(\cdot)$ be counting process with jump times $\tau_1<\tau_2<\tau_3...$ such that each $E(\cdot)$ is adapted to $\mathscr{F}_t$, $E(t)$ is integrable for $t \geq 0,$ and $\tau_k$ is a predictable stopping time for each $k.$ Let $\{a_k\}_{k=1}^{\infty}$ be a sequence of random variables such that $a_k \in \mathscr{F}_{\tau_k}$, but $a_k$ is independent of $\mathscr{F}_{\tau_k-}.$
    Let $X(\cdot)$ be an adapted process that takes values on some Polish space $S$, and $f(t,y,x): \R \times \R \times S\rightarrow \R$ be a $\mathscr{B}(\R)\times\mathscr{B}(\R) \times \mathscr{F}$-measurable function such that $\sup_{t>0}\sup_{x\in S}E[|f(t,a_1,x)|]<\infty$. Then the process defined    $$Y(\cdot):=\sum_{i=1}^{E
(\cdot)}(f(\tau_i,a_i,X(\tau_i-))-\phi(\tau_i,X(\tau_i-))),$$ where the function $\phi:\R\times\R \rightarrow \R$ is defined $\phi(t,x):= E[f(t,a_1,x)],$ is an $\mathscr{F}_t$-martingale.
When the above conditions hold, we call $Y(\cdot)$ a \textbf{counting martingale for the counting process} $\mathbf{E(\cdot)}$ \textbf{and sequence} ${\{a_i\}}_{i=1}^{\infty}.$
\end{prop}

Using Proposition \ref{martingaledecompositionprop}, we apply the Martingale Central Limit Theorem to terms of the form \eqref{martingaletermtoyexample}. This gives us the following important corollary.
\begin{rem}
    As discussed in the beginning of \S \ref{relevantdefinitionssect}, the proposition \ref{martingaledecompositionprop} is not a new result, but rather a reframing of the result in \cite{martingaledecomppaper}.
    In \S \ref{proofoftightnesssect}, we will use this decomposition to construct a family of multi-index mass-transport martingales that characterize the fluctuations of our system.
    We will then take the limit of this family using Lemma \ref{ornsteinuhlenbecktightnesscondition} and the corollary below in order to obtain a multiparameter noise process that is \emph{not} a Brownian sheet in the usual sense.
\end{rem}
\begin{cor}
\label{martingaleconvcor}
  Assume one has a sequence of filtered probability spaces $(\Omega^m,\mathscr{F}^m, \mathscr{F}_t^m,P^m)$ on which there is a sequence of counting processes $E^m(\cdot)$ with jump times $\tau_1^m < \tau_2^m <...,$ a sequence of adapted processes $X^m(\cdot),$ and an array of random variables $\{a_n^m\}_{n,m=1}^{\infty},$ such that there are counting martingales $Y_1^m(\cdot),...,Y_d^m(\cdot)$ for each $m \in \N$, as defined in Proposition \ref{martingaledecompositionprop}:
  $${Y}_{i}^m(\cdot):=\sum_{n=1}^{{E}^m(\cdot)}\left(f_i(\tau_n^m,a_n^m,X(\tau_n^m-))- \phi_{f_i}(\tau_n^m,X(\tau_n^m-))\right).$$
Then, define
$$\hat{Y}_{i}^m(\cdot):=\frac{1}{\sqrt{m}}\sum_{n=1}^{m\bar{E}^{m}(\cdot)}\left(f_i(\tau_n^m/m,a_n^m,X(\tau_n^m-)/m)- \phi_{f_i}(\tau_n^m/m,X(\tau_n^m-)/m)\right),$$
  where $\phi_{f_i}(t,x) = E[f_i(t,a_1,x)].$
  Assume also that $\bar{E}^m(\cdot):=\frac{1}{m}E^m(m\cdot)$ converges in distribution to some continuously differentiable process $\bar{E}(\cdot),$ $f_i$ is bounded for $i \in [d],$ and that for $\bar{X}^m(\cdot):=\frac{1}{m}X^m(m\cdot),$ $\phi_{f_i}(\cdot,\bar{X}^m(\cdot)),\phi_{f_if_l}(\cdot,\bar{X}^m(\cdot))\Rightarrow \phi_{f_i}^X(\cdot),\phi_{f_if_l}^X(\cdot)$ in $D([0,\infty), \R)$ for some processes $\phi_{f_i}^X(\cdot),\phi_{f_if_l}^X(\cdot)$ for each $i,l \in [N]$.
  Then for $i,l \in [n],$ the predictable quadratic variation
  \begin{equation}
\langle\hat{Y}_i^m,\hat{Y}_l^m\rangle_{\cdot} \Rightarrow \int_0^{\cdot}(\phi^X_{f_if_l}(s-)-\phi^X_{f_i}(s-)\phi^X_{f_l}(s-))E'(s)ds,
      \label{limitingcov}
  \end{equation}
  and it will follow from the Martingale Central Limit Theorem that
$$(\hat{Y}_1^m(\cdot),...,\hat{Y}^m_d(\cdot)) \Rightarrow \int_0^{\cdot} \sqrt{B(s)}d\boldsymbol{W}(s)$$
where the $d\times d$ matrix $B$ 
 has $B_{il}(\cdot) = (\phi_{f_if_l}^X(\cdot-)-\phi_{f_i}^X\phi_{f_l}^{X}(\cdot-))E'(\cdot)$ for $i,l \in [N],$ the square-root is the unique positive semi-definite matrix square root, and $\boldsymbol{W}$ is a standard $d$-dimensional Brownian motion.
\end{cor}
\begin{proof}
    We would like to apply the Martingale Central Limit Theorem (see e.g., \cite{ethierandkurtz}, Chapter 7, Theorem 1.4 (b)) to obtain the limit of these terms.
    We first calculate the compensator matrix.
    Following the Remark 1.5 in the same reference \cite{ethierandkurtz}, in the expository commentary just after the cited Martingale Central Limit theorem (Theorem 1.4 of Chapter 7), we observe that for $i,l \in [d], t\geq 0$
    \begin{align*}
\langle \hat{Y}_i^m,\hat{Y}_l^m\rangle_t &= \frac{1}{m}\sum_{n=1}^{m \bar{E}^m(t)}E[\xi_n^i\xi_n^l|\mathscr{F}_{\tau_n-}^m]
\end{align*}
   where
   $$\xi_n^i=f_i(\tau_n^m/m,a_n^m,X(\tau_n^m-)/m)- \phi_{f_i}(\tau_n^m/m,X(\tau_n^m-)/m)).$$
Expanding out term-by-term and using both the assumed independence of $a_n^m$ from $\mathscr{F}^m_{\tau_n-}$ and the fact that the stopping times are predictable, we see that
\begin{align*}
    E[f_i(\tau_n^m/m,a_n^m,X(\tau_n^m-)/m)f_l(\tau_n^m/m,a_n^m,X(\tau_n^m-)/m)|\mathscr{F}^m_{\tau_n-}] = \phi_{f_if_l}(\tau_n^m/m,X(\tau_n^m-)/m),
\end{align*}
\begin{align*}
    &E[f_i(\tau_n^m/m,a_n^m,X(\tau_n^m-)/m)\phi_{f_l}(\tau_n^m/m,X(\tau_n^m-)/m)|\mathscr{F}^m_{\tau_n^m-}]\\& = \phi_{f_i}(\tau_n^m/m,X(\tau_n^m-)/m)\phi_{f_l}(\tau_n^m/m,X(\tau_n^m-)/m),
\end{align*}
\begin{align*}
    &E[\phi_{f_i}(\tau_n^m/m,X(\tau_n^m-)/m)\phi_{f_l}(\tau_n^m/m,X(\tau_n^m-)/m)|\mathscr{F}^m_{\tau_n^m-}] \\&= \phi_{f_i}(\tau_n^m/m,X(\tau_n^m-)/m)\phi_{f_l}(\tau_n^m/m,X(\tau_n^m-)/m).
    \numberthis
    \label{quadraticvarcalc}
\end{align*}
Ultimately, we obtain
\begin{align*}
\langle \hat{Y}_i^m,\hat{Y}_l^m\rangle_{\cdot} &= \frac{1}{m}\sum_{n=1}^{m \bar{E}^m(\cdot)}\phi_{f_if_l}(\tau_n^m/m,X(\tau_n^m-)/m)-\phi_{f_i}(\tau_n^m/m,X(\tau_n^m-)/m)\phi_{f_l}(\tau_n^m/m,X(\tau_n^m-)/m)\\
& = \int_0^{\cdot} (\phi_{f_if_l}(s,\bar{X}^m(s-))-\phi_{f_i}(s,\bar{X}^m(s-))\phi_{f_l}(s,\bar{X}^m(s-)))d\bar{E}^m(s). 
\numberthis
\label{martingalepredquadvar}
\end{align*}
We will be applying the theory in \cite{protterkurtz} to obtain convergence of the stochastic integral above. 
In particular, applying Theorem 7.10 of that paper, we see that if $\bar{{E}}^m(\cdot) \Rightarrow \bar{{E}}(\cdot),$ and $\bar{E}^m(\cdot)$ satisfies their UT condition, then the stochastic integral above will converge in distribution to \eqref{limitingcov}.
The UT condition in that paper, which is given in Definition 7.4, is as follows:
\begin{defi}[Definition 7.4 from \cite{protterkurtz}]
    A sequence of semimartingales $(U^m)_{m\geq 1}$, with $U^m$ defined on a filtered probability space $(\Omega^m,\mathscr{F}^m,\mathscr{F}^m_t,P^m)$ that satisfies the usual hypothesis for each $m \geq 1,$ is said to be uniformly tight, denoted UT, if for each $t>0,$ the set
    $$ \left\{ \int_0^t H_{s-}^mdU_{s}^m, H^m \text{ is simple and predictable }, |H^m|\leq 1, m \geq 1\right\}$$ is stochastically bounded (uniformly in $m$).
\end{defi}
It is straightforward to check this definition for $\bar{{E}}^m(\cdot).$ 
We see that for and $m \in \N$ and such an $H^m,$ 
$$\left|\int_0^t H_{s-}^md\bar{E}_{s}^m\right| = \left|\frac{1}{m}\sum_{l=1}^{m\bar{{E}}^m(\cdot)}H_{\tau_l/m-}^m\right|
      \leq \frac{1}{m}\sum_{i=1}^{m\bar{E}^m(\cdot)}1= \bar{E}^m(\cdot)$$
(recalling that $\tau_l$ is the $l$th jump time of the counting process $E(\cdot)$). Because $\{\bar{E}^m(\cdot)\}_{m=1}^{\infty}$ is C-tight, the result holds.
Lastly, note that the bounded jumps conditions, (1.16) and (1.17) of (b) in the cited martingale central limit theorem will be satisfied because $\sup_{0\leq t< \infty}|\hat{\mart}_i^m(t)-\hat{\mart}_i^m(t-)|^2 \leq \frac{1}{m}4
    ||f||^2$ and for $i,l \in [d],$ $\sup_{0\leq t< \infty}|\langle \hat{\mart}_i^m,\hat{\mart}_l^m\rangle_t-\langle \hat{\mart}_i^m,\hat{\mart}_l^m\rangle_{t-}|^2 \leq \frac{1}{\sqrt{m}}
    (||f_if_l||_{\infty}+||f_i||_{\infty}||f_l||_{\infty }).$
\end{proof}
\begin{lem}
\label{ornsteinuhlenbecktightnesscondition}
    Let $X^m(\cdot)$ be a sequence of stochastic processes in $D([0,\infty),\R)$. If for each $t \geq 0,$
    \begin{equation}
        X^m(t) \leq \int_0^{t}  f^m(s,t)X^m(s)dR^m(s) + U^m(t) \label{tightnesslemmaineq}
    \end{equation}
    for some sequence of processes $\{U^m(\cdot)\}_{m=1}^{\infty},$ sequence of random functions $\{f^m(\cdot,\cdot)\}_{m=1}^{\infty}$, and a sequence of increasing processes $\{R^m(\cdot)\}_{m=1}^{\infty}$ that are compactly contained, where a process $H^m(\cdot)$ is said to be compactly contained if the following condition holds,
    \begin{enumerate}[i]
    \item (Compact Containment)\label{jacodboundedness} For each $M\in\N,$ $\epsilon>0,$ there exists $m_0\in \N$ and $K_{\epsilon} \in \R_+$ such that $$m\geq m_0\implies P^m(\sup_{t\leq M}|H^m(t)|\geq K_{\epsilon})\leq \epsilon.$$
\end{enumerate}
(For $f,$ we replace $\sup_{t\leq M}$ above with $\sup_{t\leq M} \sup_{s\leq t}.$)
    Then $X^m(\cdot)$ is compactly contained.
    If equality holds, as in, for $t\geq 0,$
    \begin{equation}
        X^m(t) = \int_0^{t}  f^m(s,t)X^m(s)dR^m(s) + U^m(t), \label{tightnesslemmaeq}
    \end{equation}
    and we assume that $U^m(\cdot)$ and $R^m(\cdot)$ are C-tight, then we have that $X^m(\cdot)$ is
    also $C$-tight.
\end{lem}

\begin{thm}{CLT for Renewal Driven Systems}\\
\label{diffusionclt}
Let   $ (\hat{\boldsymbol{X}}^m(\cdot)$ $,\hat{\boldsymbol{J}}^m(\cdot),$ ${\boldsymbol{Y}}^m_1(\cdot),...,{\boldsymbol{Y}}^m_A(\cdot),$ $\boldsymbol{b}^1(\cdot),...,\boldsymbol{b}^A(\cdot),$ $\boldsymbol{h}^{1,m}(\cdot),...,\boldsymbol{h}^{A,m}(\cdot),$ ${{E}}^m_1({g}^m_1(\cdot)),...,{{E}}^m_A({g}^m_A(\cdot)),$ $ \boldsymbol{r}^{1},...,\boldsymbol{r}^{A} $,
    $c_1^m,...,c^m_A)$ in $D([0,\infty), (\R^d)^{2+3A}\times\R^{A}) )\times\left({C_b(\R,\R)}^{d}\right)^{A}\times \R^A$ for some $A \in \N$ be a ``good sequence of diffusion-scaled renewal-driven systems" whose renewal processes have rates $(\iota_1^m,..,\iota_A^m)$.
    Assume that Assumption \ref{assumptionsgeneral} holds.
    Then if $(\hat{\boldsymbol{X}}(0),\hat{\boldsymbol{X}}(\cdot), \hat{\boldsymbol{J}}(\cdot))$ is a subsequential limit in distribution of $\{(\hat{\boldsymbol{X}}^m(0),\hat{\boldsymbol{X}}^m(\cdot), \hat{\boldsymbol{J}}^m(\cdot))\}_{m=1}^{\infty}$ that has continuous sample paths, it will satisfy the equation
     \begin{align*}
    \hat{\boldsymbol{X}}(\cdot) &= \hat{\boldsymbol{X}}(0) +\sum_{i=1}^A\int_0^{\cdot}\sqrt{D_i(s)} d\boldsymbol{W}_i(\gamma(s))+\sum_{i=1}^A \int_0^{\cdot}\boldsymbol{b}(s)c_i\iota_i^{3/2}\sigma_i\sqrt{g_i'(s)} d\tilde{W}_i(s) +  \sum_{i=1}^A\int_0^{\cdot} \boldsymbol{r}_{i}(\hat{\boldsymbol{X}}(s))ds \\&
    + \sum_{i=1}^{A}\int_0^{\cdot}\boldsymbol{h}^i(s)\hat{\boldsymbol{X}}(s-)\iota_i \bar{g}_i'(s)ds+\hat{\boldsymbol{J}}(\cdot) 
    \numberthis\label{renewaldrivenequationlimit}
   \end{align*} 
   where $\boldsymbol{W}_i(\cdot)\in D([0,\infty), \R^d),\tilde{W}_i\in D([0,\infty),\R)$ for $i \in [A],$ are independent, standard Brownian motions, and all matrix square roots are taken to be the unique symmetric square roots.

\end{thm}
\begin{rem}
Although the list of conditions in Definition~\ref{renewaldrivensystemdef} may appear extensive, they are, in fact, quite natural for systems of this type. 
For instance, the condition on martingale jumps is typically satisfied in practice, as the jump sizes are scaled by $1/\sqrt{m}$. When the system is decomposed following the outline in \S\ref{outlineofmethod}, the time changes $g_i^m(\cdot)$ generally take the form of integrals of system-dependent rate processes.
The functions $\boldsymbol{h}^{i,m}$ are often fluid-scaled and, as such, one expects them to converge to deterministic vector-valued functions.
In particular, $g_i^m(\cdot)$ typically converges to a differentiable limit.
Likewise, the functions $\boldsymbol{b}_i(\cdot)$ usually arise from the fluid model of the system and are expected to have locally finite variation.
Finally, the remaining assumptions—tightness and convergence in FCLTs and FSLLNs for the stochastic primitives—are standard requirements in the study of scaling limits for stochastic processing networks.
\end{rem}

   \subsection{Proofs of Main Toolbox Results}
   \label{proofsoftoolboxresults}
   \subsubsection{Proof of Proposition \ref{martingaledecompositionprop}}
   
  \begin{proof}
By construction, we see that $Y(t)$ is adapted to the filtration $\mathscr{F}_t.$ 
We observe that because $a_k$ is independent of $\mathscr{F}_{\tau_{k}-}$ and $X(\tau_k-),\tau_k$ are $\mathscr{F}_{\tau_{k}-}$-measurable, $\phi(\tau_k,X(\tau_k-))= E[f(\tau_k,a_k,X(\tau_k-))|\mathscr{F}_{\tau_k-}]$ (for proof of this elementary but sometimes forgotten property of conditional expectation, see e.g., example 4.1.7 of \cite{durrett}).
To prove that $E[|Y(t)|]< \infty$ for each $t \geq 0,$ note that, defining $|\phi|(t,x):= E[|f(t,a_1,x)|],$
\begin{align*}
    E[|Y(t)|]&\leq \sum_{i=1}^{\infty} E[1_{\tau_i \leq t}|(f(\tau_i,a_i,X(\tau_i-)|]+E[1_{\tau_i \leq t}|(|\phi|(\tau_i,X(\tau_i-)]\\
    & \leq \sum_{i=1}^{\infty} E[1_{\tau_i \leq t}E[|(f(\tau_i,a_i,X(\tau_i-)||\mathscr{F}_{\tau_i-}]]+E[1_{\tau_i \leq t}(|\phi|(\tau_i,X(\tau_i-)]\\
    & = \sum_{i=1}^{\infty} 2E[1_{\tau_i \leq t}|(|\phi|(\tau_i,X(\tau_i-)|]\\
    &\leq \sum_{i=1}^{E(t)}2\sup_{t\geq0}\sup_{x \in S}E[|f(t,a_1,X)|] \leq 2 E[E(t)] \sup_{t\geq0}\sup_{x \in S}E[|f(t,a_1,X)|] <\infty.
\end{align*}
Thus, we continue to the martingale property. 
By the tower property for stopping times, we see that for $0 \leq s \leq t,$
\begin{align*}
    E[\mart(t)|\mathscr{F}_s] &= E\bigg[\sum_{k=1}^{\infty} 1_{\{\tau_k \leq t\}}(f(\tau_k,a_k,X(\tau_k-))-\phi(\tau_k,X(\tau_k-)))\bigg|\mathscr{F}_s\bigg]\\&=E\bigg[\sum_{k=1}^{\infty} 1_{\{s<\tau_k \leq t\}}(f(\tau_k,a_k,X(\tau_k-))-\phi(\tau_k,X(\tau_k-)))\bigg|\mathscr{F}_s\bigg]  
    \numberthis \label{vanishingpart}\\
    &+E\bigg[\sum_{k=1}^{\infty} 1_{\{\tau_k \leq s\}}(f(\tau_k,a_k,X(\tau_k-))-\phi(\tau_k,X(\tau_k-)))\bigg|\mathscr{F}_s\bigg]
    \numberthis \label{yspart}
    \end{align*}
    The proof is concluded when one notes that \eqref{yspart} is simply $\mart(s)$ and \eqref{vanishingpart} is zero:
    \begin{align*} &E\bigg[\sum_{k=1}^{\infty} 1_{\{s<\tau_k \leq t\}}(f(\tau_k,a_k,X(\tau_k-))-\phi(\tau_k,X(\tau_k-)))\bigg|\mathscr{F}_s\bigg]\\&=E\bigg[E\bigg[\sum_{k=1}^{\infty} 1_{\{s<\tau_k \leq t\}}(f(\tau_k,a_k,X(\tau_k-))-\phi(\tau_k,X(\tau_k-)))\bigg|\mathscr{F}_{\tau_k-}\bigg]\bigg|\mathscr{F}_s\bigg]\\&=E\bigg[\sum_{k=1}^{\infty} 1_{\{s<\tau_k \leq t\}}E\bigg[(f(\tau_k,a_k,X(\tau_k-))-\phi(\tau_k,X(\tau_k-)))\bigg|\mathscr{F}_{\tau_k-}\bigg]\bigg|\mathscr{F}_s\bigg]=0.
\end{align*}
\end{proof}
   \subsubsection{Proof of Lemma \ref{ornsteinuhlenbecktightnesscondition}}
   \begin{proof}
    Applying the C-tightness criterion (see, e.g., \cite{Skorokhodtopologychapter}, Proposition 3.26), condition \ref{jacodboundedness} from the statement of the Lemma along with the following condition imply C-tightness in our case:
\begin{enumerate}[i]
\setcounter{enumi}{1}
    \item \label{jacodcontinuity} (Controlled Oscillations) For each $M\in \N$, $\epsilon>0,$ $\eta>0,$ there exists some $m_0 \in \N$ and $\theta>0$ such that $$m\geq m_0\implies P^m(\sup_{t\in [0,M-\theta]}\sup_{\delta \in [0,\theta)}|X^m(t+\delta)-X^m(t)|>\eta)\leq \epsilon.$$
\end{enumerate}
We will first prove \ref{jacodboundedness} when \eqref{tightnesslemmaineq} holds, and then prove \ref{jacodcontinuity} when \eqref{tightnesslemmaeq} also holds.
We will use the integral form of the Gr{\"o}nwall Inequality (see, e.g. \cite{corlay}, Lemma 3.1) for locally finite measures to prove \ref{jacodboundedness}.
 Let $m,M\in \N.$ Define 
 \begin{equation}
C^m_M:=e^{R^m(M)\sup_{t\leq M}\sup_{x\leq t}|f^m(x,t)|}\vee \sup_{0 \leq t \leq M}|U^m(t)| \vee \sup_{t\leq M}\sup_{s\leq t}|f^m(s,t)|\vee R^m(M). \label{CMdef}
\end{equation}
It follows from continuity of the exponential function and compact containment of $f^m(\cdot,\cdot), R^m(\cdot),U^m(\cdot)$ that $C^m_M$ will also be compactly contained. 
 The Gr{\"o}nwall Inequality for locally finite measures says that if the integral $\int_{[a,t)} |u(s)|d \mu(s)$ is well-defined on $[0,T]$ and 
 $$0 \leq  u(t) \leq x(t)+\int_{[a,t)} u(s) \mu(ds),$$
on $[0,T]$, and the function $x(\cdot)$ is nonnegative, then $u(\cdot)$ satisfies
$$ u(t) \leq x(t) + \int_{[a,t)} x(s)e^{\mu(s,t)}\mu(ds).$$
Substituting $|X(\cdot)|$ for $u(\cdot),$ $U^m(\cdot)$ for $x(\cdot),$ and the Lebesgue-Stieltjes measure induced by the function $R^m(s)\sup_{t\leq M}\sup_{x\leq t}|f^m(x,t)|$ for $\mu$, we may conclude that, in our setting,
\begin{align*}
    &\sup_{t\leq M} |X^m(t)| \\&\leq \sup_{t\leq M}U^m(t)+\sup_{t\leq M}\left(\int_0^t U^m(s) e^{\int_s^t \sup_{t\leq M}\sup_{x\leq t}|f^m(x,t)|dR^m(r)}d\sup_{t\leq M}\sup_{x\leq t}|f^m(x,t)|dR^m(s)\right) \\&\leq C_M^m + (C_M^m)^4
\numberthis \label{firstboundtightnesslemmaproof}
\end{align*}
Compact containment of $X^m(\cdot)$ follows.

We continue to the continuity condition, \ref{jacodcontinuity}. Let $M\in \N,$ $\epsilon>0,$ $\eta>0,$ $\theta>0,$  and $m\in \N.$ Then we see that, in the case of equality, using \eqref{CMdef} and applying the same Gr\"{o}nwall argument that was used to obtain \eqref{firstboundtightnesslemmaproof} to the process $X^m(t+ \cdot)-X^m(t),$
 \begin{align*}
     &\sup_{t\in [0,M-\theta]}\sup_{\delta \in [0,\theta)}|X^m(t+\delta)-X^m(t)| \\&\leq \sup_{t\in [0,M-\theta]}\sup_{\delta \in [0,\theta)}\int_t^{t+\delta} \sup_{t \leq M}\sup_{x\leq t}|f^m(x,t)||X^m(w)|dR^m(w) +\sup_{t\in [0,M-\theta]}\sup_{\delta \in [0,\theta)}|U^m(t+\delta)-U^m(t)|\\
     & \leq \sup_{t\in [0,M-\theta]}\sup_{\delta \in [0,\theta)}|R^m(t+\delta)-R^m(t)| ((C_M^m)^2+(C_M^m)^5)+  \sup_{t\in [0,M-\theta]}\sup_{\delta \in [0,\theta)}|U^m(t+\delta)-U^m(t)|.
 \end{align*}
The continuity condition \ref{jacodcontinuity} then follows from the continuity condition \ref{jacodcontinuity} holding for $U^m(\cdot)$ and $R^m(\cdot).$
\end{proof}
\subsubsection{Proof of CLT for Renewal Driven Systems}
We now prove Theorem \ref{diffusionclt}.
We will do so by proving a series of Lemmas that give convergence of each type of term in \eqref{renewaldrivenequation}.
For the remainder of this section, we will assume the assumptions of Theorem \ref{diffusionclt}.
In particular, let   $ (\hat{\boldsymbol{X}}^m(\cdot)$ $,\hat{\boldsymbol{J}}^m(\cdot),$ ${\boldsymbol{Y}}^m_1(\cdot),...,{\boldsymbol{Y}}^m_A(\cdot),$ $\boldsymbol{b}^1(\cdot),...,\boldsymbol{b}^A(\cdot),$ $\boldsymbol{h}^{1,m}(\cdot),...,\boldsymbol{h}^{A,m}(\cdot),$ ${{E}}^m_1({g}^m_1(\cdot)),...,{{E}}^m_A({g}^m_A(\cdot)),$ $ \boldsymbol{r}^{1},...,\boldsymbol{r}^{A} $,
    $c_1^m,...,c^m_A)$ in $D(\R_{+}, (\R^d)^{2+3A}\times\R^{A}) )\times\left({C_b(\R,\R)}^{d}\right)^{2A}\times \R^A$ be such a system. We first decompose the diffusion-scaled renewal processes into martingale and bounded variation parts, as described in \eqref{martingalepartofEhat} and \eqref{remainderpartofEhat}.
    This allows us to write equation \eqref{renewaldrivenequation} as
     \begin{align*}
    \hat{\boldsymbol{X}}^m(\cdot) &= \hat{\boldsymbol{X}}^m(0) +\sum_{i=1}^A \hat{\boldsymbol{Y}}^m_i(\cdot) +\sum_{i=1}^A \int_0^{\cdot}\boldsymbol{b}^i(s) dc_i^m\hat{{O}}^m_i(\bar{g}^m_i(s)) + \sum_{i=1}^A \int_0^{\cdot}\boldsymbol{b}^i(s) dc_i^m\hat{{R}}^m_i(\bar{g}^m_i(s))\\& + \sum_{i=1}^A\int_0^t \boldsymbol{r}^{i}(\hat{\boldsymbol{X}}^m(s))ds + \sum_{i=1}^{A}\int_0^{\cdot}\boldsymbol{h}^{i,m}(s)\hat{\boldsymbol{X}}^m(s-)d\bar{{E}}^m_i(\bar{g}^m_i(s))+\hat{\boldsymbol{J}}^m(\cdot)
    \numberthis\label{firsteqncltproof}
   \end{align*} 
For our situation, observe that $\{\hat{O}^m_i(\bar{g}^m_i(t)),\mathscr{F}^m_t: t \geq 0 \}$ satisfies the conditions Proposition \ref{martingaledecompositionprop} with $f(t,y,x)=\frac{1}{\sqrt{m}}\frac{1}{E[x_1^{i,m}]}y$, and is thus a martingale.
\begin{lem}
\label{dremainderconv}
    Terms of the form
  $c_j^m\int_0^{\cdot}b_{i}^j(s)d\hat{R}^m_{j}(\bar{g}^m_j(s)), $ $j \in [J],$ converge to zero in probability, uniformly on compact sets. 
\end{lem}
\begin{proof}

  Fix a $T >0.$
  Let $\bar{f}^{m}_j$ be the generalized inverse of $\bar{g}^m_j$ on $[0,T]$ as in \cite{lssubrule},
  $$\bar{f}^{m}_j(s):=\inf \{x \in [0,T]: s \leq \bar{g}^m_j(x)\}.$$
  Then, using the substitution formula for Lebesgue-Stieltjes integrals (see, e.g., \cite{lssubrule}, Proposition 1), we see that for $i,j \in [A], t\in [0,T],$ 
  \begin{align*}
      &\int_0^t b_{i}^j (s)d\hat{R}^m_j(\bar{g}_j^m(s)) = \int_0^{\bar{g}^m_j(t)} b_{i}^j (\bar{f}_j^{m}(s))d\hat{R}^m_j(s)
      \end{align*}
  Then, applying \eqref{remainderpartofEhat} and \eqref{rtdef}, the above expression equals
  \begin{align*}
\int_0^{\bar{g}^m_j(t)} b_{i}^j (\bar{f}_j^{m}(s))d\frac{1}{E[x_1^{j,m}]}\sqrt{m}\left(\sum_{l=1}^{E^m_j(ms)}\frac{x_l^{j,m}}{m}-s
      \right) .
      \end{align*}
  Using the fact that $\bar{g}^m_j(\tau_{l}^{j,m}/m), l \in \N,$ are the jump times of the process $\bar{E}_j^m(s),$ we expand to
  \begin{align*}
      &\frac{1}{E[x_1^{j,m}]}\sqrt{m}\sum_{{\tau}_l^{j,m}/m\in (0,t]}\left( b_{i}^j(\bar{f}_j^{m}(\bar{g}_j^m({\tau}_{l}^{j,m}/m)))\frac{x_l^{j,m}}{m} - \int_{\bar{g}^m_j({\tau}_{l}^{j,m}/m)}^{\bar{g}^m_j({\tau}_{l}^{j,m}/m)+x_l^{j,m}/m}b_{i}^j(\bar{f}_j^{m}(s))ds\right)\\&+o\left(\frac{x_1+x_{\bar{E}_j^m(\bar{g}^m_j(t))}}{\sqrt{m}}\right)\\
      \end{align*}
  where the error term $o\left(\frac{x_1+x_{\bar{E}_j^m(\bar{g}^m_j(t))}}{\sqrt{m}}\right)$ arises from the integral (ds) up to the first arrival and the integral (ds) after the last arrival but before time $t$ that are under- and over- covered, respectively, by the second term in the sum in the previous display.    Finally, applying  the mean value theorem for integrals for some   ${{\sigma}_l^{j,m} \in [\bar{g}^m_j({\tau}_{l}^{j,m}/m),\bar{g}^m_j({\tau}_{l}^{j,m}/m)+x_l^{j,m}/m],} $ we obtain the following bound
      \begin{align*}
       &\int_0^t b_{i}^j (s)d\hat{R}^m_j(\bar{g}_j^m(s)) \\
      & = \frac{1}{E[x_1^{j,m}]}\sum_{{\tau}_l^{j,m}/m\in (0,t]}\frac{x_l^{j,m}}{\sqrt{m}}\left( b_{i}^j(\bar{f}_j^{m}(\bar{g}_j^m({\tau}_{l}^{j,m}/m))) -  b_{i}^j(\bar{f}_j^{m}({\sigma}_{l}^{j,m}))\right)+o\left(\frac{x_1+x_{\bar{E}_j^m(\bar{g}^m_j(t))}}{\sqrt{m}}\right)\\
      &\leq \frac{1}{E[x_1^{j,m}]}\frac{\max\{{x_l^{j,m}}:l \leq E_j(mT)\}}{\sqrt{m}}(TV(b_{i}^j)_{[0,T]}+C)
      \numberthis \label{tvcalc}
  \end{align*}
  where the fourth line in follows from the TV stands for total variation.
It follows from known bounds on the maximum of sequences of i.i.d random variables, (see, e.g. \cite{downey}) and tightness of $\bar{E}_j^m(\cdot)$ that $\frac{\max\{{x_l^{j,m}}:l \leq E_j(mT)\}}{\sqrt{m}}$ will go to zero in probability if $x_l^{j,m}$ have a uniform bound on the first, second, and third moment, as is assumed in the assumptions of Theorem \ref{diffusionclt}.
In particular, applying Theorem 3 of that work \cite{downey} with $p=3$ and Markov's Inequality, we see that for $\epsilon,N >0,$
\begin{align*}
    P\left(\frac{1}{\sqrt{m}}\max_{1 \leq l \leq mN} x_{l}^{j,m}>\epsilon\right) &\leq \frac{1}{\epsilon} E\left[\frac{1}{\sqrt{m}}\max_{1 \leq l \leq mN} x_{l}^{j,m}\right]\\ &\leq \frac{1}{\epsilon}\frac{\sqrt[3]{m}}{\sqrt{m}}\left( \sup_m E[ x_{1}^{j,m}]+\sup_mE[|x_1^{j,m}-E[x_1^{j,m}]|^3]\right)\sqrt[3]{N}\rightarrow^m 0. \numberthis \label{thirdmomentultimatePbound}
\end{align*}
Fixing $\epsilon, \eta,$ it follows from compact containment of $\bar{E}^m(T)$ that there exists some $N_{\epsilon,\eta}$ such that \\$P(\sup_{m}|\bar{E}^m(T)|<N_{\epsilon,\eta}) \geq 1 - \eta/2.$
Choosing $m$ large enough that \eqref{thirdmomentultimatePbound} is less than $\eta/2$ for this choice of $\epsilon, N_{\epsilon,\eta},$ the claim is proven. 
Thus, the result follows from \eqref{tvcalc} and the assumption of locally finite variation of $(\boldsymbol{b}^1(\cdot),...,\boldsymbol{b}^A(\cdot))$ (\ref{locallyfinitevarassumption} of the assumptions of this theorem).
\end{proof}  
  We now examine the martingale terms.
  
  \begin{lem}

  \label{martingaleconv}
  Terms of the form
  $$\sum_{i=1}^A \hat{\boldsymbol{Y}}^m_i(\cdot) +\sum_{i=1}^A \int_0^{\cdot}\boldsymbol{b}^i(s) dc_i^m\hat{{O}}^m_i(\bar{g}^m_i(s)) $$ converge to
  \begin{equation}
      \sum_{i=1}^A\int_0^{\cdot}\sqrt{D_i(s)} d\boldsymbol{W}_i(\gamma(s))+\sum_{i=1}^A \int_0^{\cdot}\boldsymbol{b}(s)c_i\iota_i^{3/2}\sigma_i\sqrt{g_i'(s)} d\tilde{W}_i(s),
      \label{lemm43eqn}
  \end{equation}
  as defined in Theorem \ref{diffusionclt}.
\end{lem}
\begin{proof}
We would like to apply the Martingale Central Limit Theorem (see e.g., \cite{ethierandkurtz}, Chapter 7, Theorem 1.4 part b) to obtain the limit of these terms.
		We will view the martingale term as a vector-valued martingale:
		$(\int_0^{\cdot}b_{1}^1(s)dc_1^m\hat{O}^m_{1}(\bar{g}^m_1(s)),$ $...,$ $\int_0^{\cdot}b_{d}^1(s)dc_1^m\hat{O}^m_{1}(\bar{g}^m_1(s)), $ $...$ $\int_0^{\cdot}b_{1}^A(s)dc_A^m\hat{O}^m_{A}(\bar{g}^m_A(s))$ $...$ $\int_0^{\cdot}b_{d}^A(s)dc_A^m\hat{O}^m_{A}(\bar{g}^m_A(s)),$ $ \hat{Y}^m_{1,1},$ $...,$ $\hat{Y}^m_{1,d},$  $...,$ $\hat{Y}^m_{A,1},...,\hat{Y}^m_{A,d}).$

		In order to apply the theorem, we need to calculate the predictable quadratic covariation matrix of this vector of martingales, which we will denote $M^m(\cdot).$
		Because pairs of martingales of the form $(\hat{\boldsymbol{\mart}}^m_{i},\hat{\boldsymbol{\mart}}^m_{j}),$ $(\hat{{O}}_i,\hat{{O}}_j),$ or $(\hat{{\mart}}^m_{i,l},\hat{{O}}^m_{j}),$ for $i \neq j \in [A],l\in [d]$ are pure jump processes with no shared jumps, it follows that the predictable quadratic covariation of any such pair is zero.
		Therefore, the bottom right $dA\times dA$ portion of $M^m(\cdot)$ is simply the block diagonal of $C_1^m(\cdot),...,C_A^m(\cdot).$ 
Along similar lines, for terms of the form $\int_0^{\cdot}b_{1}^1(s)dc_1^m\hat{O}^m_{1}(\bar{g}^m_1(s)),$ $...,$ $\int_0^{\cdot}b_{d}^1(s)dc_1^m\hat{O}^m_{1}(\bar{g}^m_1(s)),$ $...,$ $\int_0^{\cdot}b_{1}^A(s)dc_A^m\hat{O}^m_{A}(\bar{g}^m_A(s)), $ $...,$ \\$\int_0^{\cdot}b_{d}^A(s)dc_A^m\hat{O}^m_{A}(\bar{g}^m_A(s))),$
we find that 
\begin{align*}
    &\left\langle\int_0^{\cdot}b_{j}^i(s)dc_i^m\hat{O}^m_{i}(\bar{g}^m_i(s)), \int_0^{\cdot}b_{k}^l(s)dc_l^m\hat{O}^m_{l}(\bar{g}^m_l(s))\right\rangle_t \\&= 1_{\{i=l\}}\int_0^t b_{j}^i(s)b_{k}^i(s)(c_i^m)^2d\langle\hat{O}_i^m(\bar{g}^m_i(\cdot))\rangle_s \hspace{8mm} t \geq 0,
\end{align*}
(see, e.g. \cite{protter}, Chapter 6, particularly Theorem 29, for background on the identities used to calculate these predictable quadratic covariations). 
The only covariations that we have not yet calculated are of the form
$$\left\langle \int_0^{\cdot}b_{j}^i(s)dc_i^m\hat{O}^m_{i}(\bar{g}^m_i(s)), \hat{Y}_{i,k}^m(\cdot)\right\rangle.$$
 Following the Remark 1.5 in the same reference \cite{ethierandkurtz}, in the expository commentary just after the cited Martingale Central Limit theorem (Theorem 1.4 of Chapter 7), we observe that
 \begin{align*}
     &\left\langle \int_0^{\cdot}b_{j}^i(s)dc_i^m\hat{O}^m_{i}(\bar{g}_i^m(\cdot)),\hat{Y}_{i,k}^m(\cdot)\right\rangle_s \\&= \sum_{n=1}^{m\bar{E}^m_i(\bar{g}^m_i(s))}E[b_{j}^i(\tau^i_n/m)c_i^m(\hat{O}_i^m(\tau^i_n/m)-\hat{O}_i^m(\tau^i_{n-1}/m))(\hat{\mart}_{i,k}^m(\tau^i_{n}/m)-\hat{\mart}_{i,k}^m(\tau^i_{n-1}/m))| \mathscr{F}_{\tau^{i}_{n}-}]\\
     & = \frac{1}{\sqrt{m}}\sum_{n=1}^{m\bar{E}^m_i(\bar{g}^m_i(s))}E\left[b_{j}^i(\tau^i_n/m)c_i^m\left(1 - \frac{x_n^{i,m}}{E[x_n^{i,m}]}\right)(\hat{\mart}_{i,k}^m(\tau^i_{n}/m)-\hat{\mart}_{i,k}^m(\tau^i_{n-1}/m))\bigg| \mathscr{F}_{\tau^{i}_{n}-}\right]\\
     & = \frac{1}{\sqrt{m}}\sum_{n=1}^{m\bar{E}^m_i(\bar{g}_i^m(s))}E\left[\left(1 - \frac{x_n^{i,m}}{E[x_n^{i,m}]}\right)\right]E\left[b_{j}^i(\tau^i_n/m)c_i^m(\hat{Y}_{i,k}^m(\tau_n^i/m)-\hat{Y}_{i,k}^m(\tau_{n-1}^i/m))\bigg|\mathscr{F}_{\tau_{n}^i-}\right]=0
 \end{align*}
where the last line follows from the assumption of independence of $x_n^{i,m}$ from $\mathscr{F}_{\tau_{n}^i-}$ and $Y_{i,k}(\tau_n^i)-Y_{i,k}(\tau_{n-1}^i)$ (see Definition \ref{renewaldrivensystemdef} for a full list of assumptions) as well as the fact that $\tau_n^i$ is $\mathscr{F}_{\tau_{n}^i-}$-measurable and $b_{j}^i(\cdot)$ is continuous. 
We conclude that $M^m(\cdot)$ is a block diagonal matrix where the first $A$ blocks are of the form
$$(U_i^{m})_{n,l}(\cdot)=\int_0^{\cdot} b_{n}^i(s)b_{l}^i(s)(c_i^m)^2d\langle \hat{O}_i^m(\bar{g}^m_i(\cdot))\rangle_s$$ for $i \in [A],$ $n,l \in [d]$ and the last $A$ blocks are of the form $C_1^m(\cdot),...,C_A^m(\cdot).$
Applying a the random time change theorem, we see that
\begin{align*}
    \langle \hat{O}_i^m(\bar{g}^m_i(\cdot))\rangle_{\cdot}
    &= \frac{1}{m}\sum_{l=1}^{m \bar{E}_i^m(\bar{g}^m_i(\cdot))}E\left[\left(1 - \frac{x_l^{i,m}}{E[x_l^{i,m}]}\right)^2\bigg| \mathscr{F}_{\tau_l^i-}\right]\\
    &=\bar{E}_i^m(\bar{g}_i^m(\cdot))Var\left(1 - \frac{x_1^{i,m}}{E[x_1^{i,m}]}\right) \Rightarrow \iota_i^3 \sigma_i^2\bar{g}_i(\cdot)
    \numberthis \label{quadraticvariationcalc}
\end{align*}
It then follows from a standard real analysis argument (one may take a Skorokhod representation to work with the pathwise limits) that $M^m(\cdot)$ converges in distribution to the block diagonal matrix where the last $A$ blocks are $C_1(\cdot),...,C_A(\cdot)$ and the first $A$ blocks are of the form
$$(U_i)_{n,l}(\cdot)=\int_0^\cdot b_{n}^i(s)b_{l}^i(s)c_i^2\iota_i^3\sigma_i^2\bar{g}'_i(s)ds$$ for $i \in [A],$ $n,l \in [d], t\geq 0.$ 
We also note that it follows from \eqref{quadraticvariationcalc} that the jumps of these entries of $M^m(\cdot)$ are all of size $\frac{1}{m}Var\left(1 - \frac{x_1^{i,m}}{E[x_1^{i,m}]}\right),$ and thus the bounded jumps condition (1.16) of the cited martingale central limit theorem is satisfied for this portion of the matrix as well.
Now that we have found the limiting behavior of the predictable quadratic covariation, following the cited theorem, the last step is to check that
$\lim_{m\rightarrow \infty} E\left[\sup_{t\leq T}|\boldsymbol{R}^m(t)-\boldsymbol{R}^m(t-)|^2\right]\rightarrow 0$ where $\boldsymbol{R}^m(\cdot)$ is the given vector of martingales.
Recall that we have assumed that, for $T>0,$ $\lim_{m \rightarrow \infty}E[\sup_{t \in [0,T]}|\hat{\boldsymbol{Y}}_i^m(t)-\hat{\boldsymbol{Y}}_i^m(t-)|^2]=0$ in bullet \ref{martingalecltsecondassumption} of the assumptions for this theorem. 
Next, we note that
\begin{align*}
   & \sup_{t\leq T}\bigg|\int_0^{t}{b}_{j}^i(s) dc_i^m\hat{{O}}^m_i(\bar{g}^m_i(s))- \int_0^{t-}{b}_{j}^i(s) dc_i^m\hat{{O}}^m_i(\bar{g}^m_i(s))\bigg|^2 \\&\leq ||b_{j}^i||_T^2 (c_i^m)^2\sup \left\{\frac{1}{m}\left(1 - \frac{x_l^{i,m}}{E[x_l^{i,m}]}\right)^2 :l \leq m \bar{E}^m_i(\bar{g}^m_i(T)) \right\}.
   \numberthis \label{analogousquant}
\end{align*}
which goes to zero in expectation by established bounds maximum of a sequence of i.i.d. random variables (see, e.g., \cite{downey}, Theorem 3, with $p=3$).
A similar argument is included in detail in the proof of Lemma \ref{dremainderconv} for the convergence of \eqref{tvcalc}.
Therefore, this martingale satisfies condition (b) of the Martingale Central Limit Theorem given in \cite{ethierandkurtz}, Theorem 1.4 of Chapter 7.
Because $M(\cdot)$ is continuous in each coordinate and, as the limit of symmetric, positive-valued, positive semidefinite matrices, it is positive semidefinite as well, the theorem applies.
Thus, we find that the martingale vector $(\int_0^{\cdot}b_{1}^1(s)dc_1^m\hat{O}^m_{1}(\bar{g}^m_1(s)),$ $...,$ $\int_0^{\cdot}b_{d}^1(s)dc_1^m\hat{O}^m_{1}(\bar{g}^m_1(s)), $ $...$ $\int_0^{\cdot}b_{1}^A(s)dc_A^m\hat{O}^m_{A}(\bar{g}^m_A(s))$ $...$ $\int_0^{\cdot}b_{d}^A(s)dc_A^m\hat{O}^m_{A}(\bar{g}^m_A(s)),$ $ \hat{Y}^m_{1,1},$ $...,$ $\hat{Y}^m_{1,d},$  $...,$ $\hat{Y}^m_{A,1},...,\hat{Y}^m_{A,d}).$ converges in distribution to a process of the form $\int_0^{\cdot} \sqrt{N(s)}d \boldsymbol{W}(\psi(s)),$ where the matrix square-root is the unique symmetric square-root and $N(\cdot)$ and $\psi(\cdot)$ are such that $\int_0^{\cdot} N_{l,j}(s)d\psi(s) = M_{l,j}(\cdot).$
Exploiting the block diagonal form of $M(\cdot),$ we obtain
$$\sum_{i=1}^A \hat{\boldsymbol{Y}}^m_i(\cdot) +\sum_{i=1}^A \int_0^{\cdot}\boldsymbol{b}^i(s) dc_i^m\hat{{O}}^m_i(\bar{g}^m_i(\cdot)) \Rightarrow \sum_{i=1}^A\int_0^{\cdot}\sqrt{D_i(s)} d\boldsymbol{W}_i(\gamma(s))+\sum_{i=1}^A \int_0^{\cdot}\sqrt{U_i(s)} d\tilde{\boldsymbol{W}}_i(s).$$
Finally, we note that we may take any matrix square-root of the matrix $U_i$ above and get a process that is the same, in distribution, as the limit process.
Thus, to obtain the final form of \eqref{lemm43eqn}, we take the most natural matrix-square root $$(\sqrt{U_i})_{n,l}:=1_{\{n=1\}}b_l(s)c_i\iota_i^{3/2}\sigma_i\sqrt{g_i'(s)}.$$
\end{proof}
\begin{lem}
\label{xhatinintconv}
For each $i \in [A],$ $$ \int_0^{\cdot}\boldsymbol{h}^{i,m}(s)\hat{\boldsymbol{X}}^m(s-) d\bar{{E}}^m_i(\bar{g}^m_i(s))\Rightarrow \int_0^{\cdot}\boldsymbol{h}^i(s)\hat{\boldsymbol{X}}(s-) d\bar{{E}}_i(\bar{g}_i(s)).$$
\end{lem}
\begin{proof}
This result follows from the theory presented in \cite{protterkurtz}, in particular, the fact that the sequence $\{\bar{{E}}^m_i(\bar{g}^m_i(\cdot))\}_{m=1}^{\infty}$ satisfies the UT condition in that paper. 
To see details, see the end of the proof of Corollary \ref{martingaleconvcor}, where the same argument is used.
\end{proof}
Finally, we prove Theorem \ref{diffusionclt}
\begin{proof}
We begin by noting that joint convergence in distribution of each term implies convergence of $\hat{\boldsymbol{X}}(\cdot)$ to a solution to the limiting SDEs (see, e.g., \cite{protterkurtz} Theorem 8.1).
Applying Lemmas \ref{xhatinintconv}, \ref{dremainderconv}, and \ref{martingaleconv} and examining \eqref{firsteqncltproof}, we see that all that is left to check is convergence of $$\int_0^{\cdot}\boldsymbol{r}^{i}(\hat{\boldsymbol{X}}^m(s))ds\rightarrow  \int_0^\cdot \boldsymbol{r}^{i}(\hat{\boldsymbol{X}}(s)) ds.$$
For this, we take a Skorokhod representation that includes all of the processes whose convergence we have established in this proof so we may work with almost sure convergence.
    This may need to take place on a different probability space, but since we are only interested in the limit in distribution, that suffices.
    We will continue to denote the Skorokhod representation using the same variables.
    Fix a realization on the almost sure set on which this convergence occurs.
    Then it follows from the continuous mapping theorem that $\boldsymbol{r}^{i}(\hat{\boldsymbol{X}}^m(\cdot))\rightarrow  \boldsymbol{r}^{i}(\hat{\boldsymbol{X}}(\cdot))$ in $D(\R_+, \R^d).$ Since the limit is continuous, this implies uniform convergence on compact sets.
    The limit of the integral term $\int_0^{\cdot}\boldsymbol{r}^{i}(\hat{\boldsymbol{X}}^m(s))ds\rightarrow  \int_0^t\boldsymbol{r}^{i}(\hat{\boldsymbol{X}}(s)) ds$ follows from uniform convergence of the integrands and a standard real analysis argument.
\end{proof}

\section{Representing the of Sequence Diffusion-Scaled Models as a Renewal-Driven System}
\label{renewaldrivensystemsect}
\subsection{Additional Fluid Model Results}
In this paper, we study each server individually, while in \cite{loeserwilliams}, the servers are studied in aggregate.
For this reason, we take a moment now to prove some results about the fluid limit of the service processes $\bar{\servp}^k(\cdot), k \in [K]$ that are analogous to the result proved for $\bar{\servp}(\cdot)$ in \cite{loeserwilliams}.
Because the servers are identical (and thus identical in distribution in the fluid limit), and they converge to deterministic functions, we will find that $\bar{\servp}^k(\cdot) = \frac{1}{K}\bar{\servp}(\cdot)$, which is the limit in the $K=1$ case of \cite{loeserwilliams}.
We also prove a useful lemma about the fluid model being bounded away from zero in the prelimit with high probability as $m\rightarrow \infty$.
\begin{lem}
    For each $T>0$,
    $$\liminf_{m \rightarrow \infty}P \{\awmass(\bar{\boldsymbol{\tm}}^m(t)) \in \R^+\setminus\{0\} \hspace{5mm} \forall t \in [0,T]\} = 1.$$
    \label{nonzerolemma}
\end{lem}
\begin{proof}
Applying the Skorokhod Representation Theorem, we may take a sequence that is equal in distribution to $\{\bar{\boldsymbol{\tm}}^m\}_{m=1}^{\infty}$ and a process that is equal in distribution to $\boldsymbol{\flm}$, possibly on a different probability space, such that $\bar{\boldsymbol{\tm}}^m\rightarrow \boldsymbol{\flm}$ almost surely.
By a slight abuse of notation, we will use the same notation for the Skorokhod representation as for the original sequence.
Fix an $\omega$ for which this convergence occurs.
Because overloaded fluid model solutions with nonzero initial conditions are nonzero for all time, $\awmass(\boldsymbol{\flm}(\cdot))$ is nonzero on $[0,T].$ 
Because $\awmass(\boldsymbol{\flm}(\cdot))$ is continuous, that means it is bounded below by some $\epsilon >0$ on $[0,T].$
Because $\awmass(\bar{\boldsymbol{\tm}}^m(\cdot)) \rightarrow \awmass(\boldsymbol{\flm}(\cdot))$ uniformly on $[0,T],$ $\awmass(\bar{\boldsymbol{\tm}}^m(\cdot))$ is eventually bounded below by $\epsilon/2.$
The result then follows for the Skorokhod representation.
Since the Skorokhod representation and our original system are the same in distribution, the result will be true for the original system as well.

\end{proof}
\begin{rem}
   It follows from Lemma \ref{nonzerolemma} that, any term of the form \\ $G^m(\cdot)=\int_0^{\cdot} 1_{\{\bar{\boldsymbol{\tm}}^m(s)=0\}}f(s, \bar{X}^m(s)) d U^m(s)$ has the property $$\liminf_{m\rightarrow \infty}P^m\{G^m(t) = 0 \hspace{2mm} \forall t \in [0,T]\}=1,$$ and the same is true with terms of the form $\tilde{G}^m(\cdot)=\int_0^{\cdot} f(s, \bar{X}^m(s)) d 1_{\{\bar{\boldsymbol{\tm}}^m(s)=0\}}U^m(s)$. 
It follows that any error terms introduced by removing the indicator functions in the decompositions and equations in this section will go to zero in probability, and thus in distribution. 
Since we focus here only on the limits in distribution, we will remove the indicator functions at this point, effectively assuming $\awmass(\bar{\boldsymbol{\tm}}^m(\cdot))\neq 0$ from here on out. 
This choice allows for a cleaner analysis, but the careful reader may observe that in the coming equations, two terms of the form above are effectively omitted from all equations, and would appear in the $\boldsymbol{J}^m(\cdot)$ term of \eqref{renewaldrivenequation}.
\label{ignoreindicators}
\end{rem}
\begin{lem}
Let $\{\bar{\servp}^{k,m}(\cdot)\}_{m=1}^{\infty}$ be a sequence of fluid-scaled service processes as described in \S \ref{modeldescriptionsect} and \S \ref{sequenceofmodelssection}.
Then $\{\bar{\boldsymbol{\ssp}}^m(\cdot),\bar{\servp}^{k,m}(\cdot)\}_{m=1}^{\infty}$ is tight, and if $(\boldsymbol{\fl}(\cdot), \bar{\servp}^{k}(\cdot))$ is a subsequential limit of $\{\bar{\boldsymbol{\ssp}}^m(\cdot),\bar{\servp}^{k,m}(\cdot)\}_{m=1}^{\infty}$, then, almost surely, $\frac{d}{dt}\bar{S}^k(t)= \frac{\awmass(\boldsymbol{\flm}(t))}{\wmass(\boldsymbol{\flm}(t)}$ for each $t$ such that $\boldsymbol{\fl}(t)>0.$
\end{lem}
\begin{proof}
This follows, with a small amount of argumentation, from Lemma 9.3 of \cite{loeserwilliams}.
Tightness of $\bar{\servp}^{k,m}(\cdot)$ follows from C-Tightness of $\bar{\servp}^m$ because the (thinned) service process for server $k$ must satisfy
$$|\bar{\servp}^{k,m}(t)-\bar{\servp}^{k,m}(s)| \leq|\bar{\servp}^{m}(t)-\bar{\servp}^{m}(s)| $$
for each $t,s\geq0,m \in \N, k \in [K].$
   Lemma 9.3 of \cite{loeserwilliams}, says that, almost surely $\frac{d}{dt}\bar{S}(t)= \frac{K\awmass(\boldsymbol{\flm}(t))}{\wmass(\boldsymbol{\flm}(t))}$ for each $t$ such that $\boldsymbol{\fl}(t)>0.$
Because each server is identical, and the initial delay for each server to start serving jobs, $\frac{s_0^k}{m}\rightarrow 0$ as $m \rightarrow \infty,$ almost surely, their limits must be identical in distribution. Because the limits are deterministic, the result follows from the fact that $\sum_{k=1}^K\bar{S}^{k,m}(
\cdot)= \bar{S}^m(\cdot)$ for each $m \in \N.$
\end{proof}
We also prove the following result.
\begin{lem}
    Let $\{\bar{\servp}^{k,m}_j(\cdot)\}_{m=1}^{\infty}$ be a sequence of fluid-scaled service processes as described in \S \ref{modeldescriptionsect} and \S \ref{sequenceofmodelssection}. Then $\{\bar{\boldsymbol{\ssp}}^m(\cdot),\bar{\servp}^{k,m}_j(\cdot),\bar{\spr}_j^{k,m}(\bar{g}_j^{k,m}(\cdot))\}_{m=1}^{\infty}$ is tight, and if $(\boldsymbol{\fl}(\cdot), \bar{\servp}^{k}_j(\cdot),\bar{\spr}_j^{k}(\bar{g}_j^{k}(\cdot)))$ is a subsequential limit of $\{\bar{\boldsymbol{\ssp}}^m(\cdot),\bar{\servp}^{k,m}_j(\cdot),\bar{\spr}_j^{k,m}(\bar{g}_j^{k,m}(\cdot))\}_{m=1}^{\infty}$, then, almost surely,  $\bar{\spr}_j^{k}(\bar{g}_j^{k}(t))=\bar{\servp}_j^k(t) = \int_0^t \frac{p_j\flm_j(s)}{\awmass(\boldsymbol{\flm}(s))}ds$ for each $t$ such that $\boldsymbol{\fl}(s)>0$ for all $s \leq t.$
    \label{sjkfluidlimit}
\end{lem}
We leave the proof of this, which will follow the proof of Lemma 9.5 of that paper, until we have done the service term decomposition used in Lemma 9.5 in our own notation. 
It will appear in \S \ref{martingalesaremartingalessect}.
\subsection{Decomposing Our Model Into Renewal Terms}
\label{decomposingrenewaltermssect}

The measure-valued process described in \S\ref{modeldescriptionsect} evolves according to three principal dynamics:
\begin{enumerate}
    \item \textit{Arrivals of class-$j$ jobs, for $j \in [J]$, governed by the renewal process $\ap_j(\cdot)$.} At the $i$th jump time $\iinta_i^{j}$ of $\ap_j(\cdot)$, two outcomes are possible. If all servers are busy at this time, i.e., on the set $\{s^{k}(\iinta_i^j-) \neq 0 \text{ for all } k \in [K]\}$, the $i$th job enters the queue, and a unit mass $\delta_{\pat_i^j}$ is added to the measure $\ssp_j(\cdot)$. Otherwise, a server is available, and the job bypasses the queue and enters service immediately.
    \label{arrivalrenewalitem}
    \item \textit{Service completions of class-$j$ jobs by server $k$, governed by the time-changed renewal processes $V^k_j(g^k_j(\cdot))$.} At the $i$th jump time $\tau_i^{\spr,k,j}$, if the system has waiting jobs, the job selected for service causes a removal of a unit mass $\delta^+_{T_{i,l}^{k,j}}$ from queue $l$ for some unique $l \in [J]$ (as described in \S\ref{modeldescriptionsect}). If no jobs are present, no change occurs in the state descriptor.
\label{seconddynamic}
    \item \textit{Deterministic aging of jobs in queue.} All point masses decrease at unit rate, corresponding to the reduction in remaining patience time.
    \label{thirddynamic}
\end{enumerate}

Here, the deterministic dynamics are described in \ref{thirddynamic}, and \ref{arrivalrenewalitem} and \ref{seconddynamic} are driven by time-changed renewal processes. To apply the methodology from \S\ref{outlineofmethod}, we observe that although the process $\ssp_j(\cdot)$ is measure-valued, it can be analyzed via a family of real-valued processes obtained by integrating bounded, measurable, and possibly time-varying functions $f: \R_+^2 \to \R$ against the measure. Specifically, we study the processes $\langle f(\cdot, x), \ssp_j(\cdot) \rangle$ for each $j \in [J]$.

Before carrying out the decomposition from \S\ref{outlineofmethod}, we rewrite the evolution equation for $\ssp_j(\cdot)$ in a form more suitable for analysis. The following representation, adapted from Lemma 7.2 in \cite{loeserwilliams}, holds almost surely for each $t \geq 0$ and $j \in [J]$ and for any $f \in \mathscr{C}$:
\begin{align*}
    \langle f, \ssp_j(t) \rangle &= \langle f, \ssp_j(0) \rangle - \int_0^t \langle f', \ssp_j(s) \rangle \, ds + \sum_{i=1}^{\ap_j(t)} 1_{\{s^k(\iinta_i^j-) \neq 0 \,\, \forall k \in [K]\}} f(\pat_i^j) \\
    &\quad - \sum_{\dt_l \in (0,t]} \sum_{i=1}^{\tm_j(\dt_l -)} 1_{\{\ur_l \in I_{j,i}(\boldsymbol{\tm}(\dt_l))\}} f\left((\supp(\ssp_j(\dt_l -)))_{\{i\}}\right),
\end{align*}
where $\dt_l$ is the $l$th time a server admits a queued job into service. That model tracks aggregate service entries, so only a single choosing sequence $\{\ur_l\}_{l=1}^\infty$ is needed.

In our model, which is equivalent in distribution but tracks service completions at the level of individual servers and job classes, the choosing variable $\ur_l$ corresponds to:
\[
\ur_l := \sum_{i=1}^\infty \sum_{n=1}^J \sum_{k=1}^K 1_{\{\dt_l = \tau_i^{\spr,k,n}\}} \ur_i^{k,n}, \quad l \in \N.
\]
Using this, we rewrite the evolution of $\ssp_j(t)$ under our notation as:
\begin{align*}
    \langle f, \ssp_j(t) \rangle &= \langle f, \ssp_j(0) \rangle - \int_0^t \langle f', \ssp_j(s) \rangle \, ds + \sum_{i=1}^{\ap_j(t)} 1_{\{s^k(\iinta_i^j-) \neq 0 \,\, \forall k \in [K]\}} f(\pat_i^j) \\
    &\quad - \sum_{n=1}^J \sum_{k=1}^K \sum_{l=1}^{\spr_n^k(g^k_n(t))} 1_{\{\boldsymbol{\tm}(\tau_l^{\spr,k,n}-) \neq \boldsymbol{0}\}} \sum_{i=1}^{\tm_j(\tau_l^{\spr,k,n}-)} 1_{\{\ur_l^{k,n} \in I_{j,i}(\boldsymbol{\tm}(\tau_l^{\spr,k,n}-))\}} \\
    &\hspace{55mm} \times f\left((\supp(\ssp_j(\tau_l^{\spr,k,n}-)))_{\{i\}}\right).
    \numberthis \label{finalprelimitequationfordecomp}
\end{align*}

Now, we decompose each renewal-driven piece of \eqref{finalprelimitequationfordecomp}. 
We formalize dynamic \ref{arrivalrenewalitem} as follows.
Examining the third term on the right hand side of \eqref{finalprelimitequationfordecomp} and adding in a time variable for the function $f$, we write
\begin{equation}
    \Delta^{\ap_j,j}_{f_j}(t) = \sum_{i=1}^{\ap_j(t)}1_{\left\{s^{{\s}}(\iinta_i^j-) \neq 0\hspace{3mm}\forall k \in \KK\right\}} f_j(\iinta_i^j,\pat_i^j), \hspace{5mm} t\geq 0,
\end{equation}
where $\Delta^{\ap_j,j}_{f_j}(t)$ represents the change to $\langle f_j(\cdot,x), \ssp_j(\cdot)\rangle$ that has occurred at the jump times of the $j$th arrival process up until time $ t\geq 0$. 
We note that, in our model, no change occurs to any $\langle f_i(\cdot,x), \ssp_i(\cdot)\rangle$ at the jump times of the $j$th arrival process for $j \neq i$. Therefore, $\Delta^{\ap_j,i}(\cdot),$ is simply zero for $j \neq i.$ Following the decomposition given in \eqref{toydecomp},
\begin{align*}
    \Delta^{\ap_j,j}_{f_j}(t) &= \sum_{i=1}^{\ap_j(t)}\left(1_{\left\{s^{{\s}}(\iinta_i^j-) \neq 0\hspace{3mm}\forall k \in \KK\right\}} f_j(\iinta_i^j,\pat_i^j)-\phi_{f_j}^{A_j,j}(\iinta_i^j,X(\iinta_i^j-))\right)\\&+\int_0^t1_{\left\{s^{{\s}}(s-) \neq 0\hspace{3mm}\forall k \in \KK\right\}} \langle f_j(s,\cdot), \pd_j\rangle d \ap_j(s) 
\end{align*}
where
\begin{equation}
    \label{phiAdef}
    \phi^{A_j,j}_{f_j}(t,X):= 1_{\left\{X_{{2J +\s}} \neq 0\hspace{3mm}\forall {\s} \in \KK\right\}}\langle  f_j(t, \cdot), \pd_j\rangle.
\end{equation}
We note that, in the above, there is a slight abuse of notation when compared with the $\phi(t,x)$ described in Proposition \ref{martingaledecompositionprop} and Corollary \ref{martingaleconvcor}.
In particular, to be consistent with the notation in Corollary \ref{martingaleconvcor}, in which
$\phi_{f_i}(t,x):=E[f_i(t, a_1,x)]$ and $f_i(\tau_n,a_n,X(\tau_n-))$ is the jump in the martingale at the time $\tau_n,$ we see that we could write $\phi_{g_i}^{\ap_j,i}(t,x) = E[g_i(t,\pat_1^j,x)]$ where
\begin{equation}
g_i(t,\pat_1^j,x)=1_{\{i=j\}}1_{\left\{x_{{2J +\s}} \neq 0\hspace{3mm}\forall {\s} \in \KK\right\}} f_j(t, \pat_1^j).
    \label{truefunctioninapmart}
\end{equation}
However, for the majority of this paper we choose to use the (inconsistent) notation $\phi_{f_i}^{\spr_j^k,i}$ so that the reader will know for which test function on the state descriptor $\boldsymbol{\ssp}(\cdot)$ the martingale was constructed.
For the service completions of jobs of class $j$ by server $k$, we have for $t \geq 0, j \in [J], k \in [K],$
\begin{equation*}
    \Delta^{\spr_j^k,i}_{f_i}(t)=\sum_{n=1}^{\spr_j^k(g^{k}_j(t))}-1_{\{\boldsymbol{\tm}(\tau_n^{\spr,k,j}-)\neq 0\}} \sum_{l=1}^{\tm_{i}(\tau_n^{\spr,k,j}-)}1_{\{\ur_{n}^{k,j}\in I_{i,l}(\boldsymbol{\tm}(\tau_n^{\spr,k,j}-))\}}(f_i(\tau_n^{\spr,k,j},\supp(\ssp_{i}(\tau_n^{\spr,k,j}-)))_{\{l\}}),
\end{equation*}
as in the fourth term in the right hand side of \eqref{finalprelimitequationfordecomp}.
This admits the decomposition, again following \eqref{toydecomp},
\begin{align*}
    \Delta^{\spr_j^k,i}_{f_i}(t)&=\sum_{n=1}^{\spr_j^k(g_j^k(t))}-1_{\{\boldsymbol{\tm}(\tau_n^{\spr,k,j}-)\neq 0\}} \sum_{l=1}^{\tm_{i}(\tau_n^{\spr,k,j}-)}1_{\{\ur_{n}^{k,j}\in I_{i,l}(\boldsymbol{\tm}(\tau_n^{\spr,k,j}-))\}}(f_i(\tau_n^{\spr,k,j},\supp(\ssp_{i}(\tau_n^{\spr,k,j}-)))_{\{l\}})\\
    &+\sum_{n=1}^{\spr_j^k(g_j^k(t))} \phi_{f_i}^{\spr_j^k,i}(\tau_n^{\spr,k,j},X(\tau_n^{\spr,k,j}-))
    -\int_0^t 1_{\{\boldsymbol{\ssp}(s-) \neq \boldsymbol{0}\}} \frac{p_i \langle f_i(s,\cdot), \ssp_i(s-) \rangle}{\sum_{n=1}^J p_n \langle 1, \ssp_n(s-) \rangle}d\spr_j^k(g_j^k(s))
\end{align*}
where
\begin{equation}
    \label{phiVdef}
    \phi_{f_i}^{\spr_j^k,i}(t,X) =1_{\{X_n \neq \boldsymbol{0} \text{ for some }n \in [J]\}} \frac{p_i \langle f_i(t,\cdot), X_i \rangle}{\sum_{n=1}^J p_n \langle 1, X_n \rangle}.
\end{equation}
Again noting that if we were to be consistent with the notation in Corollary \ref{martingaleconvcor}, we would write $\phi_{g_i}^{\spr_j^k,i}(t,x) = E[g_i(t,\kappa_1^{k,j},x)]$ where
\begin{equation}
    g_i(t,\kappa_1^{k,j},x) := \sum_{l=1}^{\infty}1_{\{l \leq \langle 1, X_i\rangle \}}1_{\{\ur_{1}^{k,j}\in I_{i,l}(\langle \boldsymbol{1}, \boldsymbol{X}\rangle)\}}(f_i(t,\supp(X_i)_{\{l\}}).
    \label{truefunctioninsprmart}
\end{equation}

We denote the ``averaged" portions as
\begin{equation}
    \avg^{\ap_j,j}_{f_j}(t):=\int_0^t1_{\left\{s^{{\s}}(s) \neq 0\hspace{3mm}\forall k \in \KK\right\}} \langle f_j(s,\cdot), \pd_j\rangle d \ap_j(s), \hspace{8mm} t\geq 0,
    \label{aavgdef}
\end{equation}
and 
\begin{equation}
    \avg^{\spr^k_j,i}_{f_i}(t):=\int_0^t 1_{\{\boldsymbol{\ssp}(s-) \neq 0\}} \frac{p_i \langle f_i(s,\cdot), \ssp_i(s-) \rangle}{\sum_{n=1}^J p_n \langle 1, \ssp_n(s-) \rangle}d\spr_j^k(g_j^k(s)), \hspace{5mm} t \geq 0.
    \label{vavgdef}
\end{equation}
We denote the terms that we will prove to be martingales using Proposition \ref{martingaledecompositionprop} as
\begin{equation}
    \mart^{\ap_j,j}_{f_j}(t):=\sum_{i=1}^{\ap_j(t)}\left(1_{\left\{s^{{\s}}(\iinta_i^j-) \neq 0\hspace{3mm}\forall {\s} \in \KK\right\}} f_j(\iinta_i^j,\pat_i^j)-\phi_{f_j}^{A,j}(\iinta_i^j,X(\iinta_i^j-))\right)
    \label{amartdef}
\end{equation}
and
\begin{align*}
     \mart^{\spr^k_j,i}_{f_i}(t)&=\sum_{n=1}^{\spr_j^k(g_j^k(t))}1_{\{\boldsymbol{\tm}(\tau_n^{\spr,k,j}-)\neq 0\}} \sum_{l=1}^{\tm_{i}(\tau_n^{\spr,k,j}-)}1_{\{\ur_{n}^{k,j}\in I_{i,l}(\boldsymbol{\tm}(\tau_n^{\spr,k,j}-))\}}(f_i(\tau_n^{\spr,k,j},\supp(\ssp_{i}(\tau_n^{\spr,k,j}-))_{\{l\}}))\\
    & -\sum_{n=1}^{\spr_j^k(g_j^k(t))} \phi_{f_i}^{\spr^k_j,i}(\tau_n^{\spr,k,j},X(\tau_n^{\spr,k,j}-)).
    \numberthis
    \label{vmartdef}
\end{align*}

Applying \eqref{finalprelimitequationfordecomp}, \eqref{aavgdef}, \eqref{vavgdef}, \eqref{amartdef}, and \eqref{vmartdef}, it follows that for $f_j \in \mathscr{C},$ $t \geq 0,$
\begin{align*}
    \langle f_j, \ssp_j(t) \rangle &= \langle f_j, \ssp_j(0) \rangle - \int_0^t \langle f_j', \ssp_j(s) \rangle ds + \Delta^{\ap_j,j}_{f_j}(t) + \sum_{l=1}^J\sum_{k=1}^{K}\Delta^{\spr^k_l,j}_{f_j}(t)\\
    &= \langle f_j, \ssp_j(0) \rangle - \int_0^t \langle f_j', \ssp_j(s) \rangle ds +\avg^{\ap_j,j}_{f_j}(t) +\mart^{\ap_j,j}_{f_j}(t)\\ &- \sum_{l=1}^J\sum_{k=1}^{K}(\avg^{\spr^k_l,j}_{f_j}(t) +\mart^{\spr^k_l,j}_{f_j}(t))
    \numberthis \label{prelimiteqn}
\end{align*}
Here we have abused notation, using $f_j \in \mathscr{C}$ rather than $f_j:\R_+^2\rightarrow \R,$ as was done in the martingale construction.
In actuality, when we use $f_j \in \mathscr{C}$, we are substituting $\tilde{f}_j \in C_b^1(\R^2_+,\R)$ such that $\tilde{f}_j(t,y):= f(y).$
\subsection{Proving the Martingale Property of Certain Terms}
\label{martingalesaremartingalessect}
We begin by explicitly defining the martingale parts of the renewal processes we are using, as described in \eqref{martingalepartofE}-\eqref{rtdef}.
\begin{equation}
	 {O}^{\spr,k,j}(t):=\sum_{l=2}^{\spr_j^k(t)+1}\left(1-\frac{\ser_l^{k,j}}{E[\ser_1^{k,j}]}\right), \hspace{5mm} t \geq0, \label{martingalepartofspr}   
	\end{equation}
    \begin{equation}
	 {O}^{\ap,j}(t):=\sum_{l=1}^{\ap_j(t)}\left(1-\frac{\inta_l^{j}}{E[\inta_1^{j}]}\right), \hspace{5mm} t \geq0, \label{martingalepartofap}   
	\end{equation}
		 and
   \begin{equation}
       {R}^{\spr,k,j}(t):=\frac{1}{E[\ser_1^{k,j}]}(r^{\spr,k,j}(t)+t-\ser_1^{k,j}), \hspace{5mm} t \geq0,
       \label{remainderpartofspr}
   \end{equation}
   \begin{equation}
       {R}^{\ap,j}(t):=\frac{1}{E[\inta_1^{j}]}(r^{\ap,j}(t)+t-\inta_0^{j}), \hspace{5mm} t \geq0,
       \label{remainderpartofap}
   \end{equation}
where 
\begin{equation}
    r^{\spr,k,j}(t)= \ser_1^{k,j} + \sum_{l=2}^{\spr_j^k(t)+1}\ser_l^{k,j}-t , \hspace{5mm} t \geq0, \label{rtdefspr}
\end{equation}
\begin{equation}
    r^{\ap,j}(t)= \inta_0^{j} + \sum_{l=1}^{\ap_j(t)}\inta_l^{j}-t , \hspace{5mm} t \geq0. \label{rtdefap}
\end{equation}

We further define a martingale term for the service \textit{entry} processes.
(We remind the reader that the $\spr_j^k(g_j^k(\cdot))$ processes count service completions).
One may observe that the jump process
$$\tilde{\spr}_j^{k}(\bar{g}^{k}_j(\cdot)-)^{rc},$$
 where $\tilde{\spr}_j^{k} = \spr_j^{k}+ 1_{[0, \infty)}$ and the superscript $rc$ indicates that we have taken the right-continuous version of the process given, is the process that jumps at each service $\textit{entry}.$
 This process would not be considered delayed in the framework given by \cite{martingaledecomppaper}, and thus, following that paper, it will admit a slightly different decomposition
 \begin{equation}
     O^{\tilde{\spr}, k,j}(t):=\sum_{l=1}^{\tilde{\spr}_j^{k}(t-)^{rc}}\left( 1-\frac{\ser_{l}^{k,j}}{E[\ser_1^{k,j}]}\right), \hspace{5mm} t \geq 0, \label{martingalepartoftildespr}
 \end{equation}
  \begin{equation}
       {R}^{\tilde{\spr},k,j}(t):=\frac{1}{E[\ser_1^{k,j}]}(r^{\tilde{\spr},k,j}(t)+t), \hspace{5mm} t \geq0,
       \label{remainderpartoftildespr}
   \end{equation}
\begin{equation}
    r^{\tilde{\spr},k,j}(t)=  \sum_{l=1}^{\tilde{\spr}_j^{k}(t-)^{rc}}\ser_l^{k,j}-t , \hspace{5mm} t \geq0.
    \label{rtdeftildespr}
\end{equation} 
\begin{lem}
    Let $f_i: \R\times \R \times (\M^J \times \R^2) \rightarrow \R$ be a bounded measurable function for each $i \in [J]$. Then, the natural filtration generated by the processes $\ap_j(\cdot),\spr_j^k(g_j^k(\cdot))$, $\sum_{n=0}^{\ap_j(\cdot)}\inta_n^j,$ $\sum_{n=1}^{\ap_j(\cdot)}\pat_n^j,$ $\sum_{n=1}^{\spr_j^k(g_j^k(\cdot))}\ur_n^{k,j},$
    $\sum_{n=1}^{\spr_j^k(g_j^k(\cdot))+1}\ser_n^{k,j},$
    $\ssp_j(\cdot),$  $a_j(\cdot),s^k(\cdot)$, $c_j^k(\cdot),$ $j \in [J], k \in [K],$ which we will denote $\mathscr{F}_t,$ is a suitable filtration for the conditions of Proposition \ref{martingaledecompositionprop} to hold for $\mart_{f_i}^{\spr_j^k,i}(t),$ $\mart_{f_i}^{\ap_i,i}(t),$  ${O}^{\spr,k,j}(g_j^k(\cdot)),$ and ${O}^{\ap,j}(\cdot)$ $i,j \in [J],k \in [K]$. In particular, $\tau_i^{\ap,j}, \tau_i^{\spr,k,j}$ are predictable stopping times, $\pat_i^j,\inta_i^j$ are measurable with respect to $\mathscr{F}_{\tau_i^{\ap,j}}$ but independent of $\mathscr{F}_{\tau_i^{\ap,j}-},$ and $\ur_i^{k,j}, \ser_{i+1}^{k,j}$ are measurable with respect to $\mathscr{F}_{\tau_i^{\spr,k,j}}$ and independent of $\mathscr{F}_{\tau_i^{\spr,k,j}-}$ for $j \in [J], k \in [K].$
    \label{martingalesaremartingaleslem}
\end{lem}
\begin{proof}
To begin, we note that because $\sum_{n=0}^{\ap_j(\cdot)}\inta_n^j,$ $\sum_{n=1}^{\ap_j(\cdot)}\pat_n^j,$ $\sum_{n=1}^{\spr_j^k(g_j^k(\cdot))}\ur_n^{k,j},$
    $\sum_{n=1}^{\spr_j^k(g_j^k(\cdot))+1}\ser_n^{k,j},$ are $\mathscr{F}_t$-measurable, $\pat_i^j,\inta_i^j$ are measurable with respect to $\mathscr{F}_{\tau_i^{\ap,j}}$ and $\ur_i^{k,j}, \ser_{i+1}^{k,j}$ are measurable with respect to $\mathscr{F}_{\tau_i^{\spr,k,j}}$ for $j \in [J], k \in [K], i \in \N.$
    Next, observe that because $f_i$ is bounded and $\{\ser_i^j\}_{j=1}^{\infty}$ and $\{\inta_i^{j,m}\}_{i=1}^{\infty}$ have uniformly bounded expectations, the condition $\sup_{t\geq 0,s\in S}E[|f(t,a_1,x)|]<\infty$ from Proposition \ref{martingaledecompositionprop} is satisfied in each case.
   Because each $\tau_i^{\ap,j}, \tau_i^{\spr,k,j}$ is a first hitting time for a measurable process, each is a stopping time.
   To prove that they are predictable stopping times,
   we see that if we let 
    \begin{equation}
        \tilde{\tau}^{\spr,k,j}_{i}= \inf\{t \geq 0: (\boldsymbol{X}(t),\boldsymbol{\ap}(t),\boldsymbol{\spr}(\boldsymbol{g}(t)), \boldsymbol{c}(t)) \in B_{i}^{\spr,k,j}\},
    \end{equation}
    where
    \begin{equation}
    B_{i}^{\spr,k,j} = \{c_j^k=1\}\cap \{\spr_j^k(g_j^k)=i-1\}
    \end{equation}
    then
    \begin{equation}
        \tau_i^{\spr,k,j}= \tilde{\tau}_i^{\spr,k,j}+ s^k(\tilde{\tau}_i^{\spr,k,j}).
    \end{equation}
Because the first term on the right hand side is a stopping time, and the second term, $s^k(\tilde{\tau}_i^{\spr,k,j}),$ which is equal to the service time of the job that entered service at server $k$ at the time $\tilde{\tau}_i^{\spr,k,j},$ is strictly positive and $\mathscr{F}_{\tilde{\tau}^{\spr,k,j}_{i}}$-measurable, it is straightforward to check that ${\tau}^{\spr,k,j}_{i}$ is a predictable stopping time. 
To show that $\tau_{i}^{\ap,j}$ is a stopping time, we follow the same steps, but with the set
$B_{i}^{\ap,j} = \{\ap_j(\cdot)=i-1\},$
and ${\tau}^{\ap,j}_{i}=\tilde{\tau}^{\ap,j}_{i}+ a_j(\tilde{\tau}^{\ap,j}_{i})=\tau_{i-1}^{\ap,j}+\inta_{i-1}^j.$

Now, we prove that $\pat_{i}^l,\inta_i^l$ are independent of $\mathscr{F}_{\tau_{i}^{\ap,l}-}.$
It suffices to show that the stopped processes $\ap_j(\cdot\wedge\tau_{i}^{\ap,l}-),$ $\spr_j^k(g_j^k(\cdot\wedge \tau_{i}^{\ap,l}-))$, $\sum_{n=0}^{\ap_j(\cdot\wedge \tau_{i}^{\ap,l}-)}\inta_n^j,$ $\sum_{n=1}^{\ap_j(\cdot\wedge \tau_{i}^{\ap,l}-)}\pat_n^j,$ $\sum_{n=1}^{\spr_j^k(g_j^k(\cdot\wedge \tau_{i}^{\ap,l}-))}\ur_n^{k,j},$\\
    $\sum_{n=1}^{\spr_j^k(g_j^k(\cdot\wedge \tau_{i}^{\ap,l}-))+1}\ser_n^{k,j},$ $\ssp_j(\cdot\wedge\tau_{i}^{\ap,l}-),a_j(\cdot\wedge \tau_{i}^{\ap,l}-),s^k(\cdot\wedge \tau_{i}^{\ap,l}-)$, $c_j^k(\cdot \wedge \tau_{i}^{\ap,l}-)$ $j \in [J], k \in [K],$ are measurable with respect to a $\sigma$-algebra that is independent of $\pat_i^l$, $\inta_i^l$ for each $t \geq 0.$
Because the natural filtration will be the smallest filtration to which these processes are adapted, it will follow that $\pat_i^l, \inta_i^l$ are also independent of $\mathscr{F}_{\tau_{i}^{\ap,l}-}.$
In order to do this, we construct an alternative model on our probability space with processes $\check{\ap}_j(\cdot),\check{\spr}_j^{k}(\check{g}_j^{k}(\cdot))$, $\sum_{n=0}^{\check{\ap}_j(\cdot)}\inta_n^j,$ $\sum_{n=1}^{\check{\ap}_j(\cdot)}\pat_n^j,$ $\sum_{n=1}^{\check{\spr}_j^k(\check{g}_j^k(\cdot))}\ur_n^{k,j},$
    $\sum_{n=1}^{\check{\spr}_j^k(\check{g}_j^k(\cdot))+1}\ser_n^{k,j},$ $\check{\ssp}_j(\cdot),\check{a}_j(\cdot),\check{s}^{k}(\cdot)$, $\check{c}_j^k(\cdot)$ $j \in [J], k \in [K],$ with one key difference: no jobs may arrive to the $l$th queue after the $i-1$st job arrives to that queue.
Then, on the set $\{t < \tau_{i}^{\ap,l}\},$ these processes are the same as their analogues in the original system for each $t \geq 0$.
However, this system is generated by only the stochastic primitives $\{\inta_n^l\}_{0\leq n \leq i-1},$ $\{\inta_n^j\}_{n\in \N_0, j\neq l},$ $\{\pat_n^{l}\}_{1\leq n \leq i-1},$ $\{\pat_n^{j}\}_{n\in \N, j\neq l},$ $\{\ser_{n}^{k,j}\}_{n,k,j\in \N},$ $\{\ur_{n}^{k,j}\}_{n,k,j\in \N},\{\tilde{\pat}_{-n}^{j}\}_{n \in \N}$ as well as the initial condition.

Therefore,
$${\mathscr{F}}_{\tau_i^{A,l}-}=\check{\mathscr{F}}_{\tau_i^{\ap,l}-} \subseteq \check{\mathscr{F}}_{\infty},$$
where $\check{\mathscr{F}}_t$ is the analogue of $\mathscr{F}_t$ in our alternative system, generated by the processes
$\check{\ap}_j(\cdot),\check{\spr}_j^{k}(\check{g}_j^{k}(\cdot))$, $\sum_{n=0}^{\check{\ap}_j(\cdot)}\inta_n^j,$ $\sum_{n=1}^{\check{\ap}_j(\cdot)}\pat_n^j,$ $\sum_{n=1}^{\check{\spr}_j^k(\check{g}_j^k(\cdot))}\ur_n^{k,j},$
    $\sum_{n=1}^{\check{\spr}_j^k(\check{g}_j^k(\cdot))+1}\ser_n^{k,j},$ $\check{\ssp}_j(\cdot),$ $\check{a}_j(\cdot),$ $\check{s}^{k}(\cdot)$, $\check{c}_j^k(\cdot)$ $j \in [J], k \in [K],$
but also 
$\check{\mathscr{F}}_{\infty}\subseteq \sigma \{\{\inta_n^l\}_{0\leq n \leq i-1},$ $\{\inta_n^j\}_{n\in \N_0, j\neq l},$ $\{\pat_n^{l}\}_{1\leq n \leq i-1},$ $\{\pat_n^{j}\}_{n\in \N, j\neq l},$ $\{\ser_{n}^{k,j}\}_{n,k,j\in \N},$ $\{\ur_{n}^{k,j}\}_{n,k,j\in \N},$ $\{\tilde{\pat}_{-n}^{j}\}_{n \in \N},  \boldsymbol{\tm}_0,$ $\boldsymbol{a}(0),\boldsymbol{s}(0)\}\wedge P_0 $
where $P_0$ are the null sets of $\mathscr{F}.$ Thus $\{u_n^l\}_{n\geq i},\{\pat_n^l\}_{n\geq i},$ are independent of $\check{\mathscr{F}}_{\infty},$ and it follows that these variables are independent of the smaller filtration $\mathscr{F}_{\tau_i^{\ap,l}-}$ is as well. The same argument applies to see that $ \ur_i^{k,l},\ser_{i+1}^{k,l}$ is independent of $\mathscr{F}_{\tau_{i}^{\spr,k,l}-}$, except that they alternative model is the one in which server $k$ stops serving jobs of type $l$ after it serves $i$ jobs of type $l.$
\end{proof}
If the reader is interested in an explicit construction of a system like the one studied in this paper with the exception that no more jobs are taken into service after a certain job, see Lemma 7.5 in \cite{loeserwilliams}.
If the reader is interested in an explicit construction of a system in which no more jobs can arrive to a queue $l$ after the $i$th one, see Lemma 7.6 in \cite{loeserwilliams}. 
Furthermore, using the same techniques with certain quantities substituted, we will show
that the following Lemma also holds.
\begin{lem}
    The processes $O^{\tilde{\spr}, k,j}(g_j^k(\cdot))$ $j \in [J],k \in [K]$ are martingales with respect to the natural filtration generated by the processes $\ap_j(\cdot),\tilde{\spr}_j^{k}({g}^{k}_j(\cdot)-)^{rc}$, $\sum_{n=1}^{\ap_j(\cdot)}\inta_n^j,$ ${\spr}_j^{k}({g}^{k}_j(\cdot)),$ $\sum_{n=1}^{\ap_j(\cdot)}\pat_n^j,$ $\sum_{n=1}^{{\spr}_j^k(g_j^k(\cdot))+1}\ur_n^{k,j},$
    $\sum_{n=1}^{\tilde{\spr}_j^{k}({g}^{k}_j(\cdot)-)^{rc}}\ser_n^{k,j},$
    $\ssp_j(\cdot),$  $a_j(\cdot),s^k(\cdot)$, $c_j^k(\cdot),$ $j \in [J], k \in [K],$ which we will denote $\tilde{\mathscr{F}}_t.$ 
    \label{weirdmgismg}
\end{lem}
\begin{proof}
This proof will be very similar to the proof of Lemma \ref{martingalesaremartingaleslem}.
We will denote the jump times of $\tilde{\spr}_j^{k}({g}^{k}_j(\cdot)-)^{rc}$ as $\tau^{\tilde{\spr},k,j}_n,$ $n \in \N.$
Fix $k\in [K], j \in [J].$
To begin, we note that because $\sum_{n=1}^{{\spr}_j^k(g_j^k(\cdot))+1}\ur_n^{k,j},$
    $\sum_{n=1}^{\tilde{\spr}_j^k(g_j^k(\cdot)-)^{rc}}\ser_n^{k,j}$ are $\tilde{\mathscr{F}}_t$-measurable, $\ur_{n+1}^{k,j},$  $\ser_n^{k,j}$ are measurable with respect to $\tilde{\mathscr{F}}_{\tau_n^{{\spr},k,j}},\tilde{\mathscr{F}}_{\tau_n^{\tilde{\spr},k,j}},$ respectively, for each $n \in \N,$ $j \in [J], k \in [K].$
   Next, we show that $\tau_i^{\tilde{\spr},k,j}, i \in \N$ are predictable stopping times with respect to $\tilde{\mathscr{F}}_t$.
   We first note that the process that tracks the next choosing variable for each server $k \in [K]$ finishing service on any given class $j \in [J]$,
   
   $$d_j^k(\cdot):=\sum_{n=1}^{\infty}1_{\{{\spr}_j^k(g_j^k(\cdot))=n\}}\ur_{n+1}^{k,j}=\sum_{n=1}^{\infty}1_{\{{\spr}_j^k(g_j^k(\cdot))=n\}}1_{\{\tau^{{\spr},k,j}_{n}\leq \cdot \}}\ur_{n+1}^{k,j},$$ 
   is measurable with respect to $\tilde{\mathscr{F}}_t$.
   This can be verified using the fact that $\kappa_{n+1}^{k,j}$ is $\tilde{\mathscr{F}}_{\tau^{{\spr},k,j}_{n}}$-measurable for each $n$ and the definition of the stopping time filtration.
   We see that if we let 
    \begin{equation}
        \tilde{\tau}^{\tilde{\spr},k,j}_{i}= \inf\{t \geq 0: (\boldsymbol{X}(t),\boldsymbol{\ap}(t),\tilde{\boldsymbol{\spr}}(\boldsymbol{g}(t)-)^{rc}, \boldsymbol{d}(t)) \in (C_{i}^{\tilde{\spr},k,j}\cup B_{i}^{\tilde{\spr},k,j}) \cap A\},
    \end{equation}
    where $B_{i}^{\tilde{\spr},k,j}$ is the set on which the next event is a service entry to server $k$ from the $j$th queue, and $C_{i}^{\tilde{\spr},k,j}$ is the set on which the next event is a service entry to the $k$th server from an arriving job of class $j$, both restricted to the set $
        A_i=\{\tilde{\spr}_j^k(g_j^k(\cdot)-)^{rc}=i-1\}.
    $

      In particular, let \begin{equation}
        B_{i}^{\tilde{\spr},k,j,1}= \{\min \{\{supp(X_1)\}\cup...\cup \{supp(X_J)\}\cup\{X_{J+1},...,X_{2J+K}\}\}=X_{2J+k}\},
    \end{equation}
        the set on which a service completion by server $k$ is the next event,
\begin{equation}
            B_{i}^{\tilde{\spr},k,j,2}= \{(X_1,...,X_J) \neq \boldsymbol{0}\}
            \end{equation}
the set where at least one queue is nonempty, and
    \begin{equation}
        B_{i}^{\tilde{\spr},k,j,3}= \left\{\sum_{i=1}^J 1_{\{c_i^k=1\}}d_i^k\in I_j((\langle {1} ,X_1\rangle,..., \langle {1} ,X_J\rangle)\right\}
    \end{equation}
      the set where the next job to be served by server $k$ is of class $j,$
      then $B_{i}^{\tilde{\spr},k,j} = B_{i}^{\tilde{\spr},k,j,1}\cap B_{i}^{\tilde{\spr},k,j,2} \cap B_{i}^{\tilde{\spr},k,j,3}.$
    Similarly, letting
    \begin{equation}
        C_{i}^{\tilde{\spr},k,j,1}= \{X_{2J+l}=0\text{ for some } l \in [K]\},
    \end{equation}
    the set on which some server is idle, \begin{equation}
        C_{i}^{\tilde{\spr},k,j,2}= \{\min \{X_{J+1},...,X_{2J}\}\cup (\{X_{2J+1},...,X_{2J+K}\}\cap \{x: x\geq 0\})\}=X_{J+j}\},
    \end{equation}
    the set where an arrival from class $j$ happens before an arrival from another class or another server becoming idle,
    and
\begin{equation}
        C_{i}^{\tilde{\spr},k,j,3}= \{\min \{l:X_{2J+l}=0\}=k\},
    \end{equation}
        the set on which $k$ is the smallest index in $\{1,...,K\}$ such that that server $k$ is idle, then $C_{i}^{\tilde{\spr},k,j}=C_{i}^{\tilde{\spr},k,j,1}\cap C_{i}^{\tilde{\spr},k,j,2}\cap C_{i}^{\tilde{\spr},k,j,3}.$ 
        It follows that
    \begin{align*}
            {\tau}^{\tilde{\spr},k,j}_{i}&=\tilde{\tau}^{\tilde{\spr},k,j}_{i}+ 1_{\{(\boldsymbol{X}(t),\boldsymbol{\ap}(t),\tilde{\boldsymbol{\spr}}(\boldsymbol{g}(t))^{rc}, \boldsymbol{d}(t))\in B_{i}^{\tilde{\spr},k,j}\}}s_k(\tilde{\tau}^{\tilde{\spr},k,j}_{i})\\&
            +1_{\{(\boldsymbol{X}(t),\boldsymbol{\ap}(t),\tilde{\boldsymbol{\spr}}(\boldsymbol{g}(t))^{rc}, \boldsymbol{d}(t))\in C_{i}^{\tilde{\spr},k,j}\}}a_j(\tilde{\tau}^{\tilde{\spr},k,j}_{i}).
    \end{align*}
Because the first term on the right hand side is a stopping time and one of the second two terms is strictly positive (see Remark \ref{simultaneouseventsremark}) and the other is zero, and the right hand side is $\tilde{\mathscr{F}}_{\tilde{\tau}^{\tilde{\spr},k,j}_{i}}$-measurable, it is straightforward to check that ${\tau}^{\tilde{\spr},k,j}_{i}$ is a predictable stopping time.

Now, we prove that $\ser_{i}^{k,j}$ is independent of $\tilde{\mathscr{F}}_{\tau_{i}^{\tilde{\spr},k,j}-}.$
Because this argument is so similar to the analogous section of the proof of Lemma \ref{martingalesaremartingaleslem}, we will be brief.
It suffices to show that the stopped processes $\ap_j(\cdot\wedge \tau_{i}^{\tilde{\spr},k,j}-),\tilde{\spr}_j^{k}({g}^{k}_j(\cdot\wedge \tau_{i}^{\tilde{\spr},k,j}-)-)^{rc}$, $\sum_{n=1}^{\ap_j(\cdot\wedge \tau_{i}^{\tilde{\spr},k,j}-)}\inta_n^j,$ ${\spr}_j^{k}({g}^{k}_j(\cdot\wedge \tau_{i}^{\tilde{\spr},k,j}-)),$ $\sum_{n=1}^{\ap_j(\cdot\wedge \tau_{i}^{\tilde{\spr},k,j}-)}\pat_n^j,$ $\sum_{n=1}^{{\spr}_j^k(g_j^k(\cdot\wedge \tau_{i}^{\tilde{\spr},k,j}-))+1}\ur_n^{k,j},$
    $\sum_{n=1}^{\tilde{\spr}_j^{k}({g}^{k}_j(\cdot\wedge \tau_{i}^{\tilde{\spr},k,j}-)-)^{rc}}\ser_n^{k,j},$
    $\ssp_j(\cdot\wedge \tau_{i}^{\tilde{\spr},k,j}-),$  $a_j(\cdot\wedge \tau_{i}^{\tilde{\spr},k,j}-),s^k(\cdot\wedge \tau_{i}^{\tilde{\spr},k,j}-),$ $c_j^k(\cdot \wedge \tau_i^{\tilde{\spr},k,j}),$ $\check{c}_j^k(\cdot \wedge \tau_i^{\tilde{\spr},k,j})$ $j\in [J], k \in [K]$ are measurable with respect to a $\sigma$-algebra that is independent of $\ser_{i}^{k,j}$.
In order to do this, we construct an alternative model on our probability space with processes $\check{\ap}_j(\cdot),\check{\tilde{\spr}}_j^{k}(\check{g}^{k}_j(\cdot)-)^{rc}$, $\sum_{n=1}^{\check{\ap}_j(\cdot\wedge \tau_{i}^{\tilde{\spr},k,j}-)}\inta_n^j,$ $\check{\spr}_j^{k}(\check{g}^{k}_j(\cdot)),$ $\sum_{n=1}^{\check{\ap}_j(\cdot\wedge \tau_{i}^{\tilde{\spr},k,j}-)}\pat_n^j,$ $\sum_{n=1}^{\check{{\spr}}_j^k(\check{g}_j^k(\cdot))+1}\ur_n^{k,j},$\\
    $\sum_{n=1}^{\check{\tilde{\spr}}_j^{k}(\check{g}^{k}_j(\cdot)-)^{rc}}\ser_n^{k,j},$
    $\check{\ssp}_j(\cdot),$  $\check{a}_j(\cdot),\check{s}^k(\cdot)$, $j\in [J], k \in [K]$ with one key difference: no jobs of class $j$ may enter service at the $k$th server after the $(i-1)$th job to do so.
Then, on the set $\{t < \tau_{i}^{\tilde{\spr},k,j}\},$ these processes are the same as their analogues in the original system for each $t \geq 0$.
However, this system is generated by only the stochastic primitives $\{\inta_n^i\}_{n\in \N_0, i \in \J},$ $\{\pat_n^{i}\}_{n\in \N,i \in \J},$ $\{\ser_{n}^{l,i}\}_{n\in \N, (l,i) \neq (k,j)},$ $\{\ser_{n}^{k,j}\}_{1\leq n \leq i-1},$ $\{\ur_{n}^{l,i}\}_{n\in \N, (l,i) \neq (k,j)},$ $\{\ur_{n}^{k,j}\}_{n \leq i}$ $\{\tilde{\pat}_{-n}^{j}\}_{n \in \N}$ as well as the initial condition $ \{\boldsymbol{\tm}_0,\boldsymbol{a}(0),\boldsymbol{s}(0)\}.$ Therefore,
${\tilde{\mathscr{F}}}_{\tau_{i}^{\tilde{\spr},k,j}-}=\check{\tilde{\mathscr{F}}}_{\tau_{i}^{\tilde{\spr},k,j}-} ,$ which is independent of $\ser_i^{k,j}.$

\end{proof}

Now that we have completed our martingale decompositions, we will prove Lemma \ref{sjkfluidlimit}.
\begin{proof}[Proof of Lemma \ref{sjkfluidlimit}]
This proof will be brief because all of the arguments are the same as in \cite{loeserwilliams} except with one server instead of the aggregate.
For more details on these arguments, please see the referenced portions of that paper.
   It follows from the decomposition in \S \ref{decomposingrenewaltermssect} that for $t \geq 0$
   $$S_j^k(t) = \sum_{n=1}^J(\avg^{\spr^k_n,j}_1(t) +\mart^{\spr^k_n,j}_1(t)),$$
   where the $1$ in the subscript above represents the constant $1$ function.
   We begin by rewriting the martingale term from \cite{loeserwilliams}, $\mart^j_t(f)$, in the notation of this paper. 
   Similar to what was done for \eqref{finalprelimitequationfordecomp}, we will be using the fact that, in that paper, $\dt_l$ is the $l$th time that a server takes a job waiting in the queues into service and $\ur_l :=\sum_{i=1}^{\infty}\sum_{x=1}^J\sum_{y=1}^K 1_{\{\dt_l = \tau_i^{\spr,y,x}\}}\ur_i^{y,x}, $ $l \in \N.$
   Using the decomposition given in Lemma 7.3 of that paper, we conclude that for $t \geq 0,$
   
   \begin{align*}
       \mart_t^{j}(f)&:= \sum_{\dt_l \in (0,t]}\sum_{i=1}^{\tm_j(\dt_l-)} 1_{\left\{\ur_l \in I_{j,i}(\boldsymbol{\tm}(\dt_l-))\right\}}f\left(\supp(\ssp_j(\dt_l-))_{\{i\}}\right)\\&
       -\int_0^t  1_{\{\awmass(\boldsymbol{\tm}(s-)) \neq 0\}}\frac{p_j\langle f, {\ssp_j}(s-)\rangle}{\wmass(\boldsymbol{\tm}(s-))} d\servp(s)\\
       & = \sum_{\dt_l \in (0,t]} \sum_{i=1}^{\tm_j(\dt_l-)} \left(1_{\left\{\ur_l \in I_{j,i}(\boldsymbol{\tm}(\dt_l-))\right\}}-\frac{p_j}{\wmass(\boldsymbol{\tm}(\dt_l-))}\right)f\left(\supp(\ssp_j(\dt_l-))_{\{i\}}\right)\\
       &=\sum_{x=1}^{K}\sum_{y=1}^{J}\sum_{\tau_{l}^{\spr,x,y}\in (0,t]}\hspace{-5mm} \sum_{i=1}^{\tm_j(\tau_{l}^{\spr, x,y}-)}\hspace{-2mm} \left(1_{\left\{\ur_l^{x,y} \in I_{j,i}(\boldsymbol{\tm}(\tau_{l}^{\spr,x,y}-))\right\}}-\frac{p_j}{\wmass(\boldsymbol{\tm}(\tau_l^{\spr,x,y}-))}\right)\\ &\cdot f\left(\supp(\ssp_j(\tau_{l}^{\spr,x,y}-))_{\{i\}}\right)\\
       &=\sum_{x=1}^K\sum_{y=1}^J\mart^{\spr^x_y,j}_f(t)
   \end{align*} 
Therefore, $\mart_{\cdot}^j(f)$ is equal to $\sum_{y=1}^J \bar{\mart}_f^{\spr_y^k,j,m}(t)$ plus some other martingale terms that share no jump times with $\sum_{y=1}^J \bar{\mart}_f^{\spr_y^k,j,m}(t)$ (since the $\tau_{l}^{x,y}$'s are distinct).
It follows that the bound on the quadratic variation of $Y^j_t(1)$, given in the proof of Lemma 9.1 of that paper also holds for $\sum_{y=1}^J\bar{\mart}^{\spr^k_y,j,m}_1(t)=\sum_{y=1}^J\frac{1}{m}{\mart}^{\spr^k_y,j,m}_1(mt) ,$ and thus, the result of Lemma 9.1 holds for this martingale as well.
Namely, $\sum_{y=1}^J\bar{\mart}^{\spr^k_y,j,m}_1(\cdot)$ converges to $0$ in probability uniformly on compact sets as $m \rightarrow \infty.$
It follows that the limit of 
$\bar{S}^{k,m}_j(\cdot)= \frac{1}{m}{S}^{k,m}_j(m\cdot)$ is equal to the limit of
\begin{align*}
\sum_{n=1}^J\bar{\avg}^{\spr^k_n,j,m}_1(\cdot)&=\sum_{n=1}^J \frac{1}{m}{\avg}^{\spr^k_n,j,m}_1(m \cdot) \\
&= \int_0^{\cdot}\frac{p_j\langle 1,\ssp_j(s)\rangle }{\wmass(\bar{\boldsymbol{\tm}}^m(s))}d\sum_{n=1}^J\bar{V}_n^{k,m}\left(\bar{g}_n^{k,m}(s)\right)
&= \int_0^{\cdot}\frac{p_j\langle 1,\ssp_j(s)\rangle }{\wmass(\bar{\boldsymbol{\tm}}^m(s))}d\left( \bar{\servp}^{k,m}(s) + \bar{\xi}^{k,m}(s)\right)
\end{align*}
where $\bar{\xi}^{k,m}(s)$ is the difference between $\sum_{n=1}^J\bar{V}_n^{k,m}\left(\bar{g}^{k,m}_n(s)\right)$ and $\bar{\servp}^{k,m}(s)$ which is at most $\frac{1}{m}$.
From here the proof follows from the proof of Lemma 9.5 in \cite{loeserwilliams} with $[0,t]=[u,v]$ with a few small notes.
First, we notify the reader that, in this section of \cite{loeserwilliams}, a Skorokhod Representation with the relevant processes included has been taken in order to work with almost sure convergence. 
Because we are only looking at the limits in distribution, this works in our case as well.
Secondly, we note that in that proof the notation $\bar{\awmass}(s)$ is used in place of $\awmass(\boldsymbol{\flm}(s))$; $\bar{\awmass}^m(s)$ is used in place of $\awmass{(\bar{\boldsymbol{\tm}}^m(s))}$; $\bar{\wmass}(s)$ is used in place of $\wmass(\boldsymbol{\flm}(s))$; and $\bar{\wmass}^m(s)$ is used in place of $\wmass{(\bar{\boldsymbol{\tm}}^m(s))}$.
\end{proof}

\begin{cor}
Let $ \boldsymbol{\ssp}^m(\cdot)\rightarrow\boldsymbol{\fl}(\cdot),$ a fluid model solution that satisfies Definition \ref{notoverdef} such that $\boldsymbol{\fl}(t)>0$ for all $t \geq 0.$
     Then $\bar{g}_j^{k,m}(\cdot) \Rightarrow \int_0^{\cdot} \frac{\frac{p_j}{\sr_j}\flm_j(s)}{\awmass(\boldsymbol{\flm}(s))}ds,$ where $\boldsymbol{\flm}(\cdot)$ is the total mass process associated to $\boldsymbol{\fl}(\cdot),$ as defined in the beginning of \S \ref{diffresultsect}.
    \label{cjconvcor}
\end{cor}
\begin{proof}
    It is clear that $\bar{g}^{k,m}_j(\cdot)$ is $C$-tight because it is continuous, differentiable, and has a derivative bounded by $1$ for each $m \in \N.$
    Let $x(\cdot)$ be a subsequential limit of $\bar{g}^{k,m}_j(\cdot)$.
    Applying the continuous mapping theorem and the fact that $\bar{\spr}_j^{k,m}(\cdot)\Rightarrow\sr_j (\cdot),$ we conclude that $\bar{\spr}_j^{k,m}(\bar{g}^{k,m}_j(\cdot)) \Rightarrow \sr_jx(\cdot).$
  The result thus follows from Lemma \ref{sjkfluidlimit} and a standard every further subsequence argument.
\end{proof}

\subsection{The Diffusion-Scaled Difference Equation}
\label{diffusionscaledifferenceequationsect}
Now that we have decomposed our system into averaged and martingale parts, it is time to diffusion-scale each part of the decomposition.
 Throughout this section, we will use \eqref{diffusiondefinitioneqn} for our diffusion-scaling, where $\boldsymbol{\fl}(\omega)$ is the unique fluid model solution with initial condition $\bar{\boldsymbol{\ssp}}_0(\omega)$ for each $\omega.$

\begin{lem}
\label{mainprelimiteqnlemma}
    Using the diffusion scaling given in \eqref{diffusiondefinitioneqn}, for $f \in \mathscr{C}, t\geq 0,$
    \begin{align*}
        \langle f, \hat{\ssp}^m_j(t)\rangle &= \langle f,  \hat{\ssp}^m_j(0) \rangle-\int_0^t \langle f', \hat{\ssp}^m_j(s)\rangle ds +\hat{\mart}^{\ap_j,j,m}_{f}(t) - \sum_{i=1}^J\sum_{k=1}^K \hat{\othermart}^{\spr_i^k,j,m}_{f}(t)\\
        & - \int_0^t \frac{p_j\langle f,\hat{\ssp}_j^m(s-)\rangle}{\wmass(\bar{\boldsymbol{\tm}}^m(s-))}d\sum_{k=1}^K\sum_{l=1}^J\bar{\spr}^{k,m}_l(\bar{g}^{k,m}_l(s))\\
    &+\int_0^t\frac{p_j \langle f, \fl_j(s)\rangle }{\wmass(\bar{\boldsymbol{\tm}}^m(s-))}\left(\frac{\awmass(\hat{\boldsymbol{\tm}}^m(s-))}{\awmass(\boldsymbol{\flm}(s))}\right)d\sum_{k=1}^K\sum_{l=1}^J\bar{\spr}^{k,m}_l(\bar{g}^{k,m}_l(s))\\
    &-\int_0^t \frac{p_j \langle f, \fl_j(s)\rangle }{\awmass({\boldsymbol{\flm}}(s))} d\sum_{k=1}^K \sum_{l=1}^J\frac{1}{\sr_l} \hat{\spr}^{k,m}_l(\bar{g}_l^{k,m}(s))-\int_0^t \frac{p_j \langle f, \fl_j(s)\rangle }{\awmass({\boldsymbol{\flm}}(s))} d\sum_{k=1}^K\sum_{j=1}^J \hat{\epsilon}^{k,j,m}(s)\\
     &+\langle f, \pd_j\rangle \hat{\ap}_j^m(t).
    \numberthis \label{mainprelimiteqnnoghat}
    \end{align*}
    where, for $t \geq 0,$
\begin{align*}
    \hat{\othermart}^{\spr_i^k,j,m}_f(t) = \hat{\mart}^{\spr_i^k,j,m}_f(t)-\int_0^t \frac{p_j \langle f, \fl_j(s)\rangle }{\awmass({\boldsymbol{\flm}}(s))} d\sum_{l=1}^J\frac{1}{\sr_l} \hat{\mart}^{\spr_i^k,l,m}_1(s),
    \numberthis
    \label{othermartdef}
\end{align*}
where the subscript $1$ above represents the constant $1$ function and for $t \geq 0,$
\begin{align*}
    \hat{\epsilon}^{k,j,m}(t)&=\frac{1}{\sr_j}\hat{O}^{\tilde{V},j,k,m}(\bar{g}_j^{k,m}(t))-\frac{1}{\sr_j}\hat{O}^{V,j,k,m}(\bar{g}_j^{k,m}(t)).
    \numberthis
    \label{epsilondef}
\end{align*}
\end{lem}
\begin{proof}
Subtracting \eqref{fluidlimiteqn} from \eqref{prelimiteqn} term by term and using \eqref{diffusiondefinitioneqn}, we see that for $f \in \mathscr{C},$
\begin{align*}
    \langle f, \hat{\ssp}_j^m(t)\rangle & = \langle f, \hat{\ssp}_j^m(0)\rangle -\int_0^t \langle f', \hat{\ssp}_j^m(s) \rangle ds + \hat{\avg}_{f}^{\ap_j,j,m}(t) - \sum_{i=1}^J\sum_{k=1}^K\hat{\avg}_{f}^{\spr^k_i,j,m}(t)\\
    & +\hat{Y}^{\ap_j,j,m}_{f}(t) - \sum_{i=1}^J\sum_{k=1}^K \hat{Y}^{\spr_i^k,j,m}_f(t)
    \numberthis \label{firstdiffusioneqn}
\end{align*}
where for $j \in [J], k \in [K],$ \begin{equation}
\hat{\avg}_f^{\ap_j,j,m}(t)=\sqrt{m}\left(\frac{1}{m}{\avg}_f^{\ap_j,j,m}(mt)-\ar_j\langle f, \pd_j \rangle t
\right),
\end{equation}
\begin{equation}
\hat{\avg}_f^{\spr^k_i,j,m}(t)=\sqrt{m}\left(\frac{1}{m}{\avg}_f^{\spr^k_i,j,m}(mt)-\int_0^t \frac{p_j\langle f, {\fl_j}(s) \rangle}{\wmass(\boldsymbol{\flm}(s))
}\frac{p_i \flm_i(s)}{\awmass(\boldsymbol{\flm}(s))}ds\right).
\label{firstavgvequationinreformulationproof}
\end{equation}
In the above equations, we have used the identity $\sum_{i=1}^J\frac{p_i \flm_i(s)}{\awmass(\boldsymbol{\flm}(s))}= \frac{\wmass(\boldsymbol{\flm}(s))}{\awmass(\boldsymbol{\flm}(s))}$ to break up the third term on the right hand side of \eqref{fluidlimiteqn} and the fact that overloaded fluid model solutions are positive to remove indicator functions in \eqref{fluidlimiteqn}.
Furthermore, $\hat{Y}^{\ap_j,j,m}_f(t),$ $\hat{Y}^{\spr_j^k,i,m}_f(t)$ for $j,i \in [J], k \in [K]$ are as in Definition \ref{diffscaledmartdef}.

Following the outline given in \ref{outlineofmethod}, we further decompose $\hat{\avg}_f^{\ap_j,j,m}(\cdot)$ and $\hat{\avg}_f^{\spr_j^k,i,m}(\cdot)$ as was done in \eqref{phihatmtoy},\eqref{errortermtoy},\eqref{integralagainstrenewaldifftoy}.
\begin{align*}
    \hat{\avg}_f^{\spr^k_j,i,m}(t) &=\int_0^t \hat{\phi}_f^{\spr^k_j,i,m}(s, \bar{\boldsymbol{\ssp}}^m(s-))d\bar{\spr}^{k,m}_j(\bar{g}^{k,m}_j(s))
    \\&+\int_0^t {\phi}_f^{\spr^k_j,i}(s, \boldsymbol{\fl}(s)) d\left( \hat{\spr}^{k,m}_j(\bar{g}^{k,m}_j(s))+\sr_j\hat{g}^{k,m}_j(s)\right),
    \numberthis \label{decomposedVhat}
\end{align*}
where $\hat{\phi}_f^{\spr_j^k,i,m}(s, \bar{X}(s-)) = \sqrt{m}\left( \frac{p_i \langle f, \bar{\ssp}_i^m(s-)\rangle}{\wmass(\bar{\boldsymbol{\tm}}^m(s))}-\frac{p_i\langle f,\fl_i(s)\rangle}{\wmass(\boldsymbol{\flm}(s))} \right)$ is as discussed below \eqref{integralagainstrenewaldifftoy}, and we have used the result of Lemma \ref{sjkfluidlimit}, $\frac{p_i\flm_i(s)}{\awmass(\boldsymbol{\flm}(s))}ds = d \bar{\spr}^k_i(\bar{g}_i^k(s))$.
Applying \eqref{vavgdef}, expanding $\hat{\phi}_f^{\spr_j^k,i,m},$ and using the fact that overloaded fluid model solutions are nonzero, we have
\begin{align*}
    \hat{\avg}_f^{\spr^k_j,i,m}(t) &=\int_0^t p_i\left( \frac{\langle f,\hat{\ssp}_i^m(s-)\rangle}{\wmass(\bar{\boldsymbol{\tm}}^m(s-))}-\frac{\langle f,\fl_i(s)\rangle}{\wmass(\boldsymbol{\flm}(s))}\frac{\wmass(\hat{\boldsymbol{\tm}}^m(s-))}{\wmass(\bar{\boldsymbol{\tm}}^m(s-))}\right)d\bar{\spr}^{k,m}_j(\bar{g}^{k,m}(s))
    \\&+\int_0^t  \frac{p_i \langle f, \fl_i(s)\rangle }{L({\boldsymbol{\flm}}(s))} d\left( \hat{\spr}^{k,m}_j(\bar{g}^{k,m}_j(s))+\sr_j\hat{g}^{k,m}_j(s)\right).\numberthis
    \label{hvdiffexpansion}
\end{align*}

Examining the unscaled ${g}_j^{k,m}(\cdot)$ from \eqref{gjkdef}, we see that
\begin{align*}
    g_j^{k,m}(t)&:= \int_0^t 1_{\{c_j^{k,m}(s)=1\}}ds\\
    & = \sum_{l=1}^J\sum_{\tau_i^{k,l,m}\in (0,t]}1_{\{\ur_i^{k,l}\in I_j(\boldsymbol{\tm}^m(\tau_i^{k,l,m}-)\}}\ser_{\spr_j^{k,m}({g}_j^{k,m}(\tau_i^{k,l,m})-)+2}^{k,j,m}- 1_{\{c_j^{k,m}(t)=1\}}s^{k,m}(t).
\end{align*}
The above equation holds because the time spent on class $j$ by server $k$ up until time $t$ is the total amount of service time of jobs of class $j$ that have entered service at server $k$ up until that point, minus the remaining time of the job currently in service if server $k$ is working on a job of class $j$ at that time.
We also use the convention $\spr_j^k(0-):=-1$ above to simplify notation. 
With this convention, the first service time counted in the sum will be $\ser_1^{k,j,m},$ as desired.
(Following Remark \ref{ignoreindicators}, we have omitted the term that adds in possible service entries from arrivals to an empty system, as this will be zero on any realization where $\bar{\boldsymbol{\tm}}^m(t) \geq \boldsymbol{0}$ for all $t \in [0,T].$)

Decomposing following the method outlined in \S \ref{outlineofmethod}, we have
\begin{align}
    g_j^{k,m}(t)= \mart^{g_j^k,m}(t)+\avg^{g_j^k,m}(t)- 1_{\{c_j^{k,m}(t)=1\}}s^{k,m}(t)
    \label{gdecomp}
\end{align}
for $t \geq 0$, where
\begin{align*}
    \mart^{g_j^k,m}(t) = \sum_{l=1}^J\sum_{\tau_i^{k,l,m}\in (0,t]}\left(1_{\{\ur_i^{k,l}\in I_j(\boldsymbol{\tm}^m(\tau_i^{k,l,m}-))\}}\ser_{\spr_j^{k,m}({g}_j^{k,m}(\tau_i^{k,l,m})-)+2}^{k,j,m}-\frac{\frac{p_j}{\sr_j}\tm_j^m(\tau_i^{k,l,m}-)}{\wmass(\boldsymbol{\tm}^m(\tau_i^{k,l,m}-))}\right)
\end{align*}
and
\begin{align*}
    H^{g_j^k,m}(t) = \int_0^t \frac{\frac{p_j}{\sr_j}\tm_j^m(s-)}{\wmass(\boldsymbol{\tm}^m(s-))}d \sum_{l=1}^J\spr^{k,m}_l(g^{k,m}_l(s)).
\end{align*}
Further decomposing $\mart^{g_j^k,m}(\cdot),$ we see that for $t \geq 0,$
\begin{align*}
    \mart^{g_j^k,m}(t) &= \sum_{l=1}^J\sum_{\tau_i^{k,l,m}\in (0,t]}1_{\{\ur_i^{k,l}\in I_j(\boldsymbol{\tm}^m(\tau_i^{k,l,m}-))\}}\left(\ser_{\spr_j^{k,m}({g}_j^{k,m}(\tau_i^{k,l,m})-)+2}^{k,j,m}-\frac{1}{\sr_j}\right)\\&+\sum_{l=1}^J\sum_{\tau_i^{k,l,m}\in (0,t]}\frac{1}{\sr_j}\left(1_{\{\ur_i^{k,l}\in I_j(\boldsymbol{\tm}^m(\tau_i^{k,l,m}-))\}}-\frac{p_j\tm_j^m(\tau_i^{k,l,m}-)}{\wmass(\boldsymbol{\tm}^m(\tau_i^{k,l,m}-))}\right)\\
    \end{align*}
We remark at this point, that the first term on the right-hand side above counts up all of the $\ser_{i}^{k,j,m}- \frac{1}{\sr_j}$ for jobs that have entered service at server $k$ from class $j$.
Thus $\mart^{g_j^k,m}(\cdot)= -\frac{1}{\sr_j}O^{\tilde{\spr},k,j,m}(g_j^{k,m}(\cdot))+\frac{1}{\sr_j}\sum_{l=1}^J\mart_1^{\spr_l^k,j,m}(\cdot),$ as defined in \eqref{martingalepartoftildespr} and \eqref{vmartdef}.
Because we primarily work with the service \textit{completion} processes, we will introduce an error term for the small difference between the service completion and service entry martingales, and say
    \begin{align*}
    \mart^{g_j^k,m}(\cdot)& = -\frac{1}{\sr_j}O^{V,j,k,m}(g_j^{k,m}(\cdot))- \epsilon^{k,j,m}(\cdot) +\frac{1}{\sr_j}\sum_{l=1}^J {\mart}^{\spr_l^k,j,m}_1(\cdot) 
\end{align*}
where for $t\geq 0$
\begin{align*}
\epsilon^{k,j,m}(t) &:=\frac{1}{\sr_j}{O}^{\tilde{V},j,k,m}(g_j^{k,m}(t))-\frac{1}{\sr_j}O^{V,j,k,m}(g_j^{k,m}(t))\\
&=\frac{1}{\sr_j}1_{\{c_j^{k}(t)=1\}}\left(1-\sr_j\ser_{\spr_j^{k}(t)+1}^{k,j}\right) .\numberthis \label{ederiv}
\end{align*}
Diffusion-scaling, we conclude that for $t \geq 0,$
\begin{equation}
    \hat{\mart}^{g_j^k,m}(t) = -\frac{1}{\sr_j} \hat{O}^{V,j,k,m}(\bar{g}_j^{k,m}(t)) - \hat{\epsilon}^{k,j,m}(t)+\frac{1}{\sr_j}\sum_{l=1}^J \hat{\mart}^{\spr_l^k,j,m}_1(t). 
    \label{yghatfinalform}
    \end{equation}
Diffusion-scaling $H^{g_j^k}(\cdot),$ applying \eqref{gdecomp}, and following the same steps as in \eqref{decomposedVhat}, we see that
\begin{align*}
    \hat{g}^{k,m}_j(t)& = \hat{\mart}^{g_j^k,m}(t) +\int_0^t\sqrt{m}\left(\frac{\frac{p_j}{\sr_j}\bar{\tm}_j^m(s-)}{\wmass(\bar{\boldsymbol{\tm}}^m(s-))}-\frac{\frac{p_j}{\sr_j}\flm_j(s)}{\wmass(\boldsymbol{\flm}(s))}\right)d\sum_{l=1}^J \bar{\spr}^{k,m}_l(\bar{g}^{k,m}_l(s))\\
    &+\int_0^t\frac{\frac{p_j}{\sr_j}\flm_j(s)}{\wmass(\boldsymbol{\flm}(s))}d\sum_{l=1}^J\left( \hat{\spr}^{k,m}_l(\bar{g}^{k,m}_l(s))+ \sr_l\hat{g}^{k,m}_l(s)\right)- \frac{1}{\sqrt{m}}1_{\{c_j^{k,m}(t)=1\}}s^{k,m}(mt).
\end{align*}
Noting that, because we have a non-idling assumption, and following Remark \ref{ignoreindicators}, we may once again assume the queues are all nonempty for $t \geq 0,$ the service time given by server $k$ before time $t$ is $t,$ and thus
\begin{equation}
    \sum_{j=1}^J \bar{g}_j^{k,m}(t) =\frac{1}{m}\sum_{j=1}^J {g}_j^{k,m}(mt) = \frac{1}{m} mt= t = \sum_{j=1}^J \bar{g}_j^{k}(t), \hspace{5mm} t \geq 0.
\end{equation}
Therefore,
\begin{align*}
    0=\sum_{j=1}^J\hat{g}^{k,m}_j(t) & = \sum_{j=1}^J\hat{\mart}^{g_j^k,m}(t) +\int_0^t\sqrt{m}\left(\frac{\awmass(\bar{\boldsymbol{\tm}}^m(s-))}{\wmass(\bar{\boldsymbol{\tm}}^m(s-))}-\frac{\awmass({\boldsymbol{\flm}}(s))}{\wmass(\boldsymbol{\flm}(s))}\right)d\sum_{l=1}^J \bar{\spr}^{k,m}_l(\bar{g}^{k,m}_l(s))\\
    &+\int_0^t\frac{\awmass(\boldsymbol{\flm}(s))}{\wmass(\boldsymbol{\flm}(s))}d\sum_{l=1}^J\left( \hat{\spr}^{k,m}_l(\bar{g}^{k,m}_l(s))+\sr_l\hat{g}^{k,m}_l(s)\right)- \frac{1}{\sqrt{m}}s^{k,m}(mt).
\end{align*}
Now, we expand the integrand $\sqrt{m}\left(\frac{\awmass(\bar{\boldsymbol{\tm}}^m(s-))}{\wmass(\bar{\boldsymbol{\tm}}^m(s-))}-\frac{\awmass({\boldsymbol{\flm}}(s))}{\wmass(\boldsymbol{\flm}(s))}\right)$ as was done for $\hat{\avg}_f^{\spr^k_j,i,m}(t)$ in \eqref{hvdiffexpansion}:
\begin{align*}
    \sqrt{m} \left(\frac{\awmass(\bar{\boldsymbol{\tm}}^m(s-))}{\wmass(\bar{\boldsymbol{\tm}}^m(s-))}-\frac{\awmass(\boldsymbol{\flm}(s))}{\wmass(\boldsymbol{\flm}(s))}\right)=\left( \frac{\awmass(\hat{\boldsymbol{\tm}}^m(s-))\wmass(\boldsymbol{\flm}(s))-\awmass(\boldsymbol{\flm}(s))\wmass(\hat{\boldsymbol{\tm}}^m(s-))}{\wmass(\bar{\boldsymbol{\tm}}^m(s-))\wmass(\boldsymbol{\flm}(s))}\right)
\end{align*}
Next, we note that
\begin{align*}
     &-\frac{\wmass(\boldsymbol{\flm}(s))}{\awmass(\boldsymbol{\flm}(s))}d\sum_{l=1}^J\hat{\mart}^{g_l^k,m}(s)+ \frac{\wmass(\boldsymbol{\flm}(s))}{\awmass(\boldsymbol{\flm}(s))}d\frac{1}{\sqrt{m}}s^{k,m}(ms)
    \\&= \frac{\wmass(\boldsymbol{\flm}(s))}{\awmass(\boldsymbol{\flm}(s))}d\left(\sum_{l=1}^J\frac{1}{\sr_l} \hat{O}^{V,l,k,m}(\bar{g}_l^{k,m}(s))+
    \frac{1}{\sqrt{m}}s^{k,m}(ms)\right)\\
    &+\frac{\wmass(\boldsymbol{\flm}(s))}{\awmass(\boldsymbol{\flm}(s))}d\sum_{l=1}^J\hat{\epsilon}^{k,l,m}(s)-\frac{\wmass(\boldsymbol{\flm}(s))}{\awmass(\boldsymbol{\flm}(s))}d\sum_{l=1}^J\frac{1}{\sr_l}\sum_{i=1}^J \hat{\mart}^{\spr_i^k,l,m}_1(s) \\
\end{align*}
Rearranging and combining the four displays above along with \eqref{yghatfinalform}, we find that 
\begin{align*}
    &d\sum_{l=1}^J \left(  \hat{\spr}^{k,m}_l(\bar{g}^{k,m}_l(s))+\sr_l\hat{g}^{k,m}_l(s)\right)\\
    &= \frac{\wmass(\boldsymbol{\flm}(s))}{\awmass(\boldsymbol{\flm}(s))}d\left(\sum_{l=1}^J\frac{1}{\sr_l} \hat{O}^{V,l,k,m}(\bar{g}_l^{k,m}(s))+
    \frac{1}{\sqrt{m}}s^{k,m}(ms)\right)\\
    &+\frac{\wmass(\boldsymbol{\flm}(s))}{\awmass(\boldsymbol{\flm}(s))}d\sum_{l=1}^J\hat{\epsilon}^{k,l,m}(s)-\frac{\wmass(\boldsymbol{\flm}(s))}{\awmass(\boldsymbol{\flm}(s))}d\sum_{l=1}^J\frac{1}{\sr_l}\sum_{i=1}^J \hat{\mart}^{\spr_i^k,l,m}_1(s) \\
    &-\frac{\wmass(\boldsymbol{\flm}(s))}{\awmass(\boldsymbol{\flm}(s))}\left( \frac{\awmass(\hat{\boldsymbol{\tm}}^m(s-))\wmass(\boldsymbol{\flm}(s))-\awmass(\boldsymbol{\flm}(s))\wmass(\hat{\boldsymbol{\tm}}^m(s-))}{\wmass(\bar{\boldsymbol{\tm}}^m(s-))\wmass(\boldsymbol{\flm}(s))}\right)d\sum_{l=1}^J \bar{\spr}^{k,m}_l(\bar{g}^{k,m}_l(s)).
    \numberthis
    \label{ghattransformation}
\end{align*}
Then for $\hat{\avg}^{A,j,m},$ following the same steps but with \eqref{phiAdef}, we obtain
\begin{equation}
    \hat{\avg}^{A,j,m}_t(f) = \langle f, \pd_j\rangle \hat{\ap}_j^m(t), \hspace{5mm} t\geq 0. 
    \numberthis
    \label{hadiffexpansion}
\end{equation}
Lastly, we note that
\begin{equation}
    \left(\sum_{l=1}^J\frac{1}{\sr_l} \hat{O}^{V,k,l,m}(\bar{g}_l^{k,m}(s))+
    \frac{1}{\sqrt{m}}s^{k,m}(ms)\right)= \sum_{l=1}^J\frac{1}{\sr_l} \hat{\spr}^{k,m}_l(\bar{g}
    _l^{k,m}(s)),\hspace{5mm} s\geq0, \label{hash}
\end{equation}
because $\sr_l\frac{1}{\sqrt{m}}s^{k,m}(mt)$ is the remainder term $\hat{R}^{\spr, k,l,m}(g_l^{k,m}(mt))$ for whichever process $\spr_l^{k,m}(g_l^{k,m}(mt))$ is running at time $mt$ (this can be directly checked using \eqref{rtdefspr} and the diffusion scaling \eqref{remainderpartofEhat}).
Then, combining \eqref{firstdiffusioneqn}, \eqref{hvdiffexpansion}, \eqref{hadiffexpansion}, \eqref{ghattransformation}, and \eqref{hash}, one obtains 
  \begin{align}
        \langle f, \hat{\ssp}^m_j(t)\rangle &= \langle f,  \hat{\ssp}^m_j(0) \rangle-\int_0^t \langle f', \hat{\ssp}^m_j(s)\rangle ds +\hat{\mart}^{\ap_j,j,m}_f(t) - \sum_{i=1}^J\sum_{k=1}^K \hat{\othermart}^{\spr_i^k,j,m}_f(t)\\
        &- \int_0^t p_j\left( \frac{\langle f,\hat{\ssp}_j^m(s-)\rangle}{\wmass(\bar{\boldsymbol{\tm}}^m(s-))}-\frac{\langle f,\fl_j(s)\rangle}{\wmass(\boldsymbol{\flm}(s))}\frac{\wmass(\hat{\boldsymbol{\tm}}^m(s-))}{\wmass(\bar{\boldsymbol{\tm}}^m(s-))}\right)d\sum_{k=1}^K\sum_{l=1}^J\bar{\spr}^{k,m}_l(\bar{g}_l^{k,m}(s))\numberthis \label{whaaaat}
    \\
    &-\int_0^t \frac{p_j \langle f, \fl_j(s)\rangle }{\awmass({\boldsymbol{\flm}}(s))} d\sum_{k=1}^K \sum_{l=1}^J\frac{1}{\sr_l} \hat{\spr}^{k,m}_l(\bar{g}_l^{k,m}(s))-\int_0^t \frac{p_j \langle f, \fl_j(s)\rangle }{\awmass({\boldsymbol{\flm}}(s))} d\sum_{k=1}^K \sum_{l=1}^J\hat{\epsilon}^{k,l,m}(s)\\
     &+\int_0^t \frac{p_j \langle f, \fl_j(s)\rangle }{\awmass({\boldsymbol{\flm}}(s))} \left( \frac{\awmass(\hat{\boldsymbol{\tm}}^m(s))\wmass(\boldsymbol{\flm}(s))-\awmass(\boldsymbol{\flm}(s))\wmass(\hat{\boldsymbol{\tm}}^m(s))}{\wmass(\bar{\boldsymbol{\tm}}^m(s))\wmass(\boldsymbol{\flm}(s))}\right)d\sum_{k=1}^K\sum_{l=1}^J\bar{\spr}^{k,m}_l(\bar{g}^{k,m}_l(s))\label{line62}\\
     &+\langle f, \pd_j\rangle \hat{\ap}_j^m(t).
    \numberthis \label{almosttomaineq}
    \end{align}
All that is left to do is combine like terms. In particular, we combine \eqref{whaaaat} and \eqref{line62} to achieve \eqref{mainprelimiteqnnoghat}.

\end{proof}
\subsection{Mass Transport Version of Diffusion-Scaled Difference Equation}
We will now introduce a mass transport equation that will be satisfied by a large class of test functions integrated against $\hat{\boldsymbol{\ssp}}^m(\cdot).$
We will denote translation by $x \geq 0$ of a function $f: \R_+\rightarrow \R$ as follows:
\begin{equation}
t_xf(y):=
\begin{cases}
f(y-x), & y>x,\\
0, & y\leq x,
\end{cases}
\label{translationnotation}
\end{equation}
Furthermore, for a function $f: \R_+\rightarrow\R,$ we define
\begin{equation}
    \flcdf_f^{j,c}(t,x): = \langle t_xf, \fl(t)\rangle,
\end{equation}
and
\begin{equation}
    \pdcdf_f^{j,c}(x) := \langle t_xf, \pd_j\rangle.
\end{equation}
In \cite{loeserwilliams}, an alternate fluid model equation, (24), which can be thought of as a mass transport equation, is given in Lemma 4.1.
We write this equation in the notation of our paper,
\begin{align*}
\flcdf_1^{j,c}(t,x)&= \flcdf_1^{j,c}(u,t+x-u) + \int_u^t \pdcdf_1^{j,c}(t+x-s)  d \bar{\ap}_j(s) \\&- \sum_{l=1 }^J\sum_{k=1}^K\int_u^t \frac{p_j \flcdf_1^{j,c}(s,t+x-s)  }{\wmass(\boldsymbol{\flm}(s))}d \bar{\spr}_l^k(\bar{g}_l^k(s)) \numberthis
    \label{masstransportindicator}
\end{align*}
for $t \geq u \geq 0$, where the last terms are obtained using the limit of the service processes obtained in Lemma \ref{sjkfluidlimit} and the fact that $\bar{\ap}_j(s)= \ar_js$ for $s\geq 0$.
In \cite{loeserwilliams}, equation (24) is obtained from the fluid model equation in Lemma 4.1 by first obtaining the following equation for $g \in \mathscr{C}$, $0 \leq u \leq t,$ taking $g(x)=0$ for $x \leq 0,$
\begin{align*}
    \langle g(\cdot), \fl_j(t) \rangle &= \langle g(\cdot-t+u), \fl_j(u) \rangle -\int_u^t \frac{Kp_j \langle g(\cdot -t+s) , \fl_j(s) \rangle}{\awmass(\boldsymbol{\flm}(s))}ds + \int_u^t \ar_j \langle g (\cdot-t+s), \pd_j \rangle ds.\numberthis
    \label{masstransportequationfromotherpaper}
\end{align*}
Then, the authors used an approximation argument, in which they approximated $1_{(0,\infty)}$ from below and applied the monotone convergence theorem, to obtain \eqref{masstransportindicator}, but for the test function $1_{(0,\infty)}$ rather than $1.$
When one observes that, because the measures above are all continuous, $M_{1_{(0,\infty}}^{j,c}(t,x)= M_1^{j,c}(t,x)$ and $N_{1_{(0,\infty}}^{c}(t,x)= N_1^{j}(t,x)$ for all $t,x \geq 0,$ we obtain \eqref{masstransportindicator}.
Substituting $t_xf$ for $g$ in \eqref{masstransportequationfromotherpaper} and the limits $\ap_j(s)=\ar_js$ and $\bar{\spr}_l^k(\bar{g}_l^k(s))= \int_0^s \frac{p_l \flm_l(s)}{\awmass(\flm(s))}ds$ for $s\geq 0$, we obtain
\begin{align*}
    \flcdf_f^{j,c}(t,x)&= \flcdf_f^{j,c}(u,t+x-u)
      -\sum_{j=1 }^J\sum_{k =1}^K\int_u^t \frac{p_j \flcdf_f^{j,c}(s,t+x-s)}{\wmass(\boldsymbol{\flm}(s))}d\bar{\spr}_l^k(\bar{g}_l^k(s)) \\&+  \int_u^t \pdcdf_f^{j,c}(t+x-s) d\bar{\ap}_j(s) ,\numberthis
    \label{masstransportequation}
\end{align*}
$t \geq u \geq 0.$
We note that the above equation is the same as equation \eqref{masstransportindicator}, but has now been extended from $f = 1$ to any $f$ in $\mathscr{C} \cup \{1\}.$ 
It is worthwhile to do the martingale decomposition for the mass transport representation of the sequence of diffusion-scaled models, centered around the fluid limit mass transport equation \eqref{masstransportequation} for $f \in \mathscr{C}\cup \{1\}$. We do so now.
\begin{lem}
    Let $f\equiv 1$ or $f \in \mathscr{C}.$ Define
    $$ \hat{\flcdf}^{j,c,m}_f(t,x):= \langle t_x f(\cdot), \hat{\ssp}_j^m(t) \rangle.$$
Then, almost surely, for $t, x \geq 0,$
\begin{align*}
\hat{{\flcdf}}^{j,c,m}_f(t,x)& =  \hat{{\flcdf}}^{j,c,m}_f(0,t+x)+\hat{\mart}^{\ap_j,j,m}_{t_{t+x-\cdot}f}(t) - \sum_{i=1}^J\sum_{k=1}^K \hat{\othermart}^{\spr_i^k,j,m}_{t_{t+x-\cdot}f}(t)\\
        & - \int_0^t \frac{p_j\hat{{\flcdf}}^{j,c,m}_f(s-,t+x-s)}{\wmass(\bar{\boldsymbol{\tm}}^m(s-))}d\sum_{k=1}^K\sum_{l=1}^J\bar{\spr}^{k,m}_l(\bar{g}^{k,m}_l(s))\\
    &+\int_0^t\frac{p_j {{\flcdf}}^{j,c}_f(s,t+x-s)}{\wmass(\bar{\boldsymbol{\tm}}^m(s-))}\left(\frac{\awmass(\hat{\boldsymbol{\tm}}^m(s-))}{\awmass(\boldsymbol{\flm}(s))}\right)d\sum_{k=1}^K\sum_{l=1}^J\bar{\spr}^{k,m}_l(\bar{g}^{k,m}_l(s))\\
        &-\int_0^t \frac{p_j {{\flcdf}}^{j,c}_f(s,t+x-s)}{\awmass({\boldsymbol{\flm}}(s))} d\sum_{k=1}^K \sum_{l=1}^J\frac{1}{\sr_l} \hat{\spr}^{k,m}_l(\bar{g}_l^{k,m}(s))\\&-\int_0^t \frac{p_j {{\flcdf}}^{j,c}_f(s,t+x-s)}{\awmass({\boldsymbol{\flm}}(s))} d\sum_{k=1}^K \sum_{i=1}^J \hat{\epsilon}^{k,i,m}(s)\\
     &+\int_0^t\pdcdf^{j,c}_f(t+x-s) d\hat{\ap}_j^m(s).
    \numberthis \label{mainMeq}
    \end{align*}
\end{lem}
\begin{proof}
We will be following the same method as was done in \S \ref{outlineofmethod} and \S \ref{decomposingrenewaltermssect}-\ref{diffusionscaledifferenceequationsect} to obtain \eqref{mainprelimiteqnnoghat}, so we keep the following proof brief.
Applying \eqref{statespacedescriptorequation} and assuming nonzero paths as was done in the proof of Lemma \ref{mainprelimiteqnlemma}, following Remark \ref{ignoreindicators}, we see that for $t \geq 0,$
\begin{align*}
\langle t_xf, {\ssp}_j(t) \rangle 
&= \langle t_{t+x}f , \ssp_j(0)\rangle + \sum_{i=1}^{\ap_j(t)} t_{t+x}f(\iinta_i^j+\pat_i^j)\hspace{20mm}\\& -\sum_{k \in [K]}\sum_{l \in [J]}\sum_{\tau_i^{\spr,k,l} \in (0,t]} t_{t+x} f(T_{i,j}^{k,l}+\tau_i^{\spr,k,l}),\hspace{20mm}
\end{align*}
and using the decompositions given in \eqref{aavgdef}, \eqref{vavgdef}, \eqref{amartdef}, and \eqref{vmartdef}, we rewrite this as
\begin{align*}
    &=\flcdf_f^{j,c}(0,t+x) + \mart^{\ap_j,j}_{t_{t+x-\cdot}f}(t) + \avg^{\ap_j,j}_{t_{t+x-\cdot}f}(t)
    \\&-\sum_{k \in [K]}\sum_{l \in [J]}\mart^{\spr^k_l,j}_{t_{t+x-\cdot}f}(t)-\sum_{k \in [K]}\sum_{l \in [J]}\avg^{\spr^k_l,j}_{t_{t+x-\cdot}f}(t),
\end{align*}
\begin{align*}
    &\hspace{35mm}=\flcdf_f^{j,c}(0,t+x) + \mart^{\ap_j,j}_{t_{t+x-\cdot}f}(t) + \int_0^t \pdcdf^{j,c}_f(t+x-s)d \ap_j(s) \\&\hspace{35mm}-\sum_{k \in [K]}\sum_{l \in [J]}\mart^{\spr^k_l,j}_{t_{t+x-\cdot}f}(t)-\sum_{k \in [K]}\sum_{l \in [J]}\int_0^t  \frac{p_j \flcdf_f^{j,c}(s-,t+x-s) }{\sum_{n=1}^J p_n \langle 1, \ssp_n(s-) \rangle}d\spr_l^k(g_l^k(s)).
\end{align*}
Subtracting off \eqref{masstransportequation} with $u=0$ and following the calculation in the proof of Lemma \ref{mainprelimiteqnlemma} with $t_{x+t-\cdot}f$ in place of $f$, \eqref{mainMeq} follows.

\end{proof}
\section{Proof of Tightness}
\label{proofoftightnesssect}

In this section, we prove Theorem \ref{tightnessresult}.
 We first use Lemma \ref{ornsteinuhlenbecktightnesscondition} to reduce C-tightness to C-tightness of a function of the martingale terms, fluid-scaled terms, and deterministic terms, which we we call $U_f^{j,m}(r,t)$.
We then prove C-tightness of $U_f^{j,m}(r,t)$.
\begin{lem}
  For $f \in \mathscr{C}\cup \{1\},$ $0 \leq r \leq t,$ define
\begin{align*}
    U_f^{j,m}(r,t)&:= \hat{M}_f^{j,c,m}(0,t)+ \hat{{\mart}}^{\ap_j,j,m}_{t_{t-\cdot }f}(r)- \sum_{i=1}^J\sum_{k=1}^K \hat{\othermart}^{\spr_i^k,j,m}_{t_{t-\cdot }f}(r)\\
    &-\int_0^r \frac{p_j {{\flcdf}}^{j,c}_f(s,t-s)}{\awmass({\boldsymbol{\flm}}(s))} d\sum_{k=1}^K \sum_{l=1}^J\frac{1}{\sr_l} \hat{\spr}^{k,m}_l(\bar{g}_l^{k,m}(s))-\int_0^r \frac{p_j {{\flcdf}}^{j,c}_f(s,t-s)}{\awmass({\boldsymbol{\flm}}(s))} d\sum_{k=1}^K\sum_{i=1}^J \hat{\epsilon}^{k,i,m}(s)\\
     &+\int_0^r\pdcdf^{j,c}_f(t-s) d\hat{\ap}_j^m(s).
     \numberthis
     \label{ujmdef}
\end{align*}
Then, if $U_f^{j,m}(r,t)$ is compactly contained, in other words for each $M\in\N,$ $\epsilon>0,$ there exists $m_0\in \N$ and $K_{\epsilon} \in \R_+$ such that
\begin{equation}
    m\geq m_0\implies P^m(\sup_{t\leq M}\sup_{r \leq t}|U^{j,m}_{f}(r,t)|\geq K_{\epsilon})\leq \epsilon \label{compactcontainmentcond}
\end{equation}
then if we define
\begin{equation}
    R_f^{j,m}(r,t): = \hat{M}_f^{j,c,m}(r,t-r), \hspace{5mm}t\geq0,0 \leq r \leq t, \label{rjdef}
\end{equation}
$R_f^{j,m}(\cdot,\cdot)$ satisfies the condition, and $\awmass(\hat{\boldsymbol{\tm}}^m(
\cdot))$ satisfies the condition with the $\sup_{r \leq t}$ removed.
If we assume further that $U_f^{j,m}(\cdot,\cdot)$ is C-tight as a multiparameter process, then $R_f^{j,m}(\cdot,\cdot)$ is also C-tight as a multiparameter process.
\label{Lcompactcontainmentlemma}
\end{lem}
Because, for $t\geq 0,$
$ \awmass(\hat{\boldsymbol{\tm}}^m(t))= \sum_{j=1}^J \frac{p_j}{\sr_j}R_{1}^{j,m}(t,t)$, and for $f \in \mathscr{S},$ $f-f(0)1 \in \mathscr{C},$ and thus
\begin{equation}
    \langle f,\hat{\ssp}_j^m(t) \rangle=\langle f-f(0),\hat{\ssp}_j^m(t) \rangle+f(0)\langle 1,\hat{\ssp}_j^m(t) \rangle = R_{f-f(0)}^{j,m}(t,t)+f(0)R_{1}^{j,m}(t,t),
    \label{zerodecomeqn}
\end{equation}
Lemma \ref{Lcompactcontainmentlemma} and the Mitoma tightness criterion reduce Theorem \ref{tightnessresult} to tightness of $U_f^{j,m}(\cdot,\cdot)$.
\begin{proof}[Proof of Lemma \ref{Lcompactcontainmentlemma}]
Applying \eqref{rjdef} and \eqref{mainMeq}
\begin{align*}
{R}_f^{j,m}(r,t)& = {R}_f^{j,m}(0,t)+\hat{\mart}^{\ap_j,j,m}_{t_{t-\cdot}f}(r) - \sum_{i=1}^J\sum_{k=1}^K \hat{\othermart}^{\spr_i^k,j,m}_{t_{t-\cdot}f}(r)\\
        & - \int_0^r \frac{p_jR_f^{j,m}(s-,t)}{\wmass(\bar{\boldsymbol{\tm}}^m(s-))}d\sum_{k=1}^K\sum_{l=1}^J\bar{\spr}^{k,m}_l(\bar{g}^{k,m}_l(s))\\
    &+\int_0^r\frac{p_j {{\flcdf}}^{j,c}_f(s,t-s)}{\wmass(\bar{\boldsymbol{\tm}}^m(s-))}\left(\frac{\awmass(\hat{\boldsymbol{\tm}}^m(s-))}{\awmass(\boldsymbol{\flm}(s))}\right)d\sum_{k=1}^K\sum_{l=1}^J\bar{\spr}^{k,m}_l(\bar{g}^{k,m}_l(s))\\
    &-\int_0^r \frac{p_j {{\flcdf}}^{j,c}_f(s,t-s)}{\awmass({\boldsymbol{\flm}}(s))} d\sum_{k=1}^K \sum_{l=1}^J\frac{1}{\sr_l} \hat{\spr}^{k,m}_l(\bar{g}_l^k(s))-\int_0^r \frac{p_j {{\flcdf}}^{j,c}_f(s,t-s)}{\awmass({\boldsymbol{\flm}}(s))} d\sum_{k=1}^K \sum_{i=1}^J\hat{\epsilon}^{k,i,m}(s)\\
     &+\int_0^r\pdcdf_f^{j,c}(t-s) d\hat{\ap}_j^m(s).
    \end{align*}
For the second line we note that, almost surely, $R_f^{j,m}(s-,t)= \lim_{r\rightarrow s^-}\flcdf_f^{j,c,m}(r,t-r)=\flcdf_f^{j,c,m}(s-,t-s)$ for each $f \in \mathscr{C} \cup \{1\}$.
To see this, choose $a_n \downarrow 0$ and $\omega\in \Omega$.
Then if we take the random time $\sigma$ to be the last arrival or service departure time before time $s,$ then $\sigma(\omega) < s$ because interarrival and service times are positive. Thus, because masses in $\bar{\ssp}^m(\cdot)$ move to the left at rate $1,$ for $a_n < s- \sigma(\omega)$,
the masses that are past $t-s+a_n$ at time $s-a_n$ are the same as the masses that are past $t-s$ at time $s$.
The result then follows immediately when $f = 1$.
In the case that $f \in \mathscr{C},$ it follows from continuity of $f.$

Then we see that, applying \eqref{ujmdef},
\begin{align*}
    {R}_f^{j,m}(r,t)& = U_f^{j,m}(r,t)- \int_0^r \frac{p_jR_f^{j,m}(s-,t)}{\wmass(\bar{\boldsymbol{\tm}}^m(s-))}d\sum_{k=1}^K\sum_{l=1}^J\bar{\spr}^{k,m}_l(\bar{g}^{k,m}_l(s))\\
    &+\int_0^r\frac{p_j {{\flcdf}}^{j,c}_f(s,t-s)}{\wmass(\bar{\boldsymbol{\tm}}^m(s-))}\left(\frac{\awmass(\hat{\boldsymbol{\tm}}^m(s-))}{\awmass(\boldsymbol{\flm}(s))}\right)d\sum_{k=1}^K\sum_{l=1}^J\bar{\spr}^{k,m}_l(\bar{g}^{k,m}_l(s)).
    \numberthis \label{finalrjeq}
\end{align*}
It follows that
\begin{align*}
    |{R}_f^{j,m}(r-,t)|& \leq |U_f^{j,m}(r-,t)|+ \int_{0}^r \frac{p_j}{\wmass(\bar{\boldsymbol{\tm}}^m(s-))}|R_f^{j,m}(s-,t)|d\sum_{k=1}^K\sum_{l=1}^J\bar{\spr}^{k,m}_l(\bar{g}^{k,m}_l(s))\\
    &+\int_0^r\frac{p_j {{\flcdf}}^{j,c}_f(s,t-s)}{\wmass(\bar{\boldsymbol{\tm}}^m(s-))\awmass(\boldsymbol{\flm}(s))}\left|\awmass(\hat{\boldsymbol{\tm}}^m(s-))\right|d\sum_{k=1}^K\sum_{l=1}^J\bar{\spr}^{k,m}_l(\bar{g}^{k,m}_l(s)).
\end{align*}
Applying the same Gr{\"o}nwall Inequality argument as in the proof of Lemma \ref{ornsteinuhlenbecktightnesscondition}, we conclude that
\begin{align*}
    |R_f^{j,m}(r-,t)| &\leq x(r) \\&+ \int_0^r x(s) e^{\int_s^r\frac{p_j}{\wmass(\bar{\boldsymbol{\tm}}^m(y-))}d\sum_{k=1}^K\sum_{l=1}^J\bar{\spr}^{k,m}_l(\bar{g}^{k,m}_l(y))}\frac{p_j}{\wmass(\bar{\boldsymbol{\tm}}^m(s-))}d\sum_{k=1}^K\sum_{l=1}^J\bar{\spr}^{k,m}_l(\bar{g}^{k,m}_l(s))
\end{align*}
for $$ x(\cdot) = |U_f^{j,m}(\cdot-,t)|+\int_{0}^{\cdot}\frac{p_j {{\flcdf}}^{j,c}_f(s,t-s)}{\wmass(\bar{\boldsymbol{\tm}}^m(s-))\awmass(\boldsymbol{\flm}(s))}\left|\awmass(\hat{\boldsymbol{\tm}}^m(s-))\right|d\sum_{k=1}^K\sum_{l=1}^J\bar{\spr}^{k,m}_l(\bar{g}^{k,m}_l(s)).$$
Expanding and changing the order of integration, we obtain

\begin{align*}
    &|R_f^{j,m}(r-,t)|\\
    &\leq |U_f^{j,m}(r-,t)|+\int_{0}^r\frac{p_j {{\flcdf}}_f^{j,c}(s,t-s)}{\wmass(\bar{\boldsymbol{\tm}}^m(s-))\awmass(\boldsymbol{\flm}(s))}\left|\awmass(\hat{\boldsymbol{\tm}}^m(s-))\right|d\sum_{k=1}^K\sum_{l=1}^J\bar{\spr}^{k,m}_l(\bar{g}^{k,m}_l(s)) \\&+ \int_0^r |U_f^{j,m}(s-,t)| e^{\int_s^r\frac{p_j}{\wmass(\bar{\boldsymbol{\tm}}^m(x-))}d\sum_{k=1}^K\sum_{l=1}^J\bar{\spr}^{k,m}_l(\bar{g}^{k,m}_l(x))}\frac{p_j}{\wmass(\bar{\boldsymbol{\tm}}^m(s-))}d\sum_{k=1}^K\sum_{l=1}^J\bar{\spr}^{k,m}_l(\bar{g}^{k,m}_l(s))\\
    &+\int_0^r\left|\awmass(\hat{\boldsymbol{\tm}}^m(y-))\right|\frac{p_j {{\flcdf}}^{j,c}_f(y,t-y)}{\wmass(\bar{\boldsymbol{\tm}}^m(y-))\awmass(\boldsymbol{\flm}(y))}\int_{y}^re^{\int_s^r\frac{p_j}{\wmass(\bar{\boldsymbol{\tm}}^m(x-))}d\sum_{k=1}^K\sum_{l=1}^J\bar{\spr}^{k,m}_l(\bar{g}^{k,m}_l(x))} \\ & \cdot\frac{p_j}{\wmass(\bar{\boldsymbol{\tm}}^m(s-))}d\sum_{k=1}^K\sum_{l=1}^J\bar{\spr}^{k,m}_l(\bar{g}^{k,m}_l(s))d\sum_{k=1}^K\sum_{l=1}^J\bar{\spr}^{k,m}_l(\bar{g}^{k,m}_l(y))  
\end{align*}
and thus, defining

\begin{align*}
&\tilde{U}_f^{j,m}(r-,t):=|U_f^{j,m}(r-,t)|\\&+\int_0^r |U_f^{j,m}(s-,t)| e^{\int_s^r\frac{p_j}{\wmass(\bar{\boldsymbol{\tm}}^m(x-))}d\sum_{k=1}^K\sum_{l=1}^J\bar{\spr}^{k,m}_l(\bar{g}^{k,m}_l(x))}\frac{p_j}{\wmass(\bar{\boldsymbol{\tm}}^m(s-))}d\sum_{k=1}^K\sum_{l=1}^J\bar{\spr}^{k,m}_l(\bar{g}^{k,m}_l(s))\\
\end{align*}

and
\begin{align*}
    &h^{j,m}_{f,r-,t}(s):=\frac{p_j {{\flcdf}}^{j,c}_f(s,t-s)}{\wmass(\bar{\boldsymbol{\tm}}^m(s-))\awmass(\boldsymbol{\flm}(s))}\\&+\frac{p_j {{\flcdf}}^{j,c}_f(s,t-s)}{\wmass(\bar{\boldsymbol{\tm}}^m(s-))\awmass(\boldsymbol{\flm}(s))}\int_{s}^re^{\int_y^r\frac{p_j}{\wmass(\bar{\boldsymbol{\tm}}^m(x-))}d\sum_{k=1}^K\sum_{l=1}^J\bar{\spr}^{k,m}_l(\bar{g}^{k,m}_l(x))}\frac{p_j}{\wmass(\bar{\boldsymbol{\tm}}^m(y-))}d\sum_{k=1}^K\sum_{l=1}^J\bar{\spr}^{k,m}_l(\bar{g}^{k,m}_l(y))
\end{align*}
Then we may conclude that
\begin{align*}
    |R_f^{j,m}(r-,t)| \leq \tilde{U}_f^{j,m}(r-,t) + \int_0^rh_{f,r-,t}^{j,m}(s) |\awmass(\hat{\boldsymbol{\tm}}^m(s-)|d\sum_{k=1}^K\sum_{l=1}^J\bar{\spr}^{k,m}_l(\bar{g}^{k,m}_l(s)).
    \numberthis
\label{finalRjineq}
\end{align*}
Finally, we see that
\begin{align*}
|\awmass(\hat{\boldsymbol{\tm}}^m(t-))|& \leq\sum_{j=1}^J \frac{p_j}{\sr_j}|R_{1}^{j,m}(t-,t)|\\
& \leq\sum_{j=1}^J \frac{p_j}{\sr_j} \tilde{U}_{1}^{j,m}(t-,t) +\int_0^t\sum_{j=1}^J \frac{p_j}{\sr_j}h_{{1},t-,t}^{j,m}(s) |\awmass(\hat{\boldsymbol{\tm}}^m(s-)|d\sum_{k=1}^K\sum_{l=1}^J\bar{\spr}^{k,m}_l(\bar{g}^{k,m}_l(s)) 
    \end{align*}
Applying Lemma \ref{ornsteinuhlenbecktightnesscondition}, compact containment of $|\awmass(\hat{\boldsymbol{\tm}}^m(t-))|$ follows from the condition \eqref{compactcontainmentcond} holding for $\tilde{U}_{f}^{j,m}(t-,t)$, $h_{{f},t-,t}^{j,m}(s),$ and $\sum_{k=1}^K\sum_{l=1}^J\bar{\spr}^{k,m}_l(\bar{g}^{k,m}_l(s)),$ specifically when $f \equiv 1.$
This follows from tightness of the fluid model, which was proved in \cite{loeserwilliams}, Lemma \ref{nonzerolemma}, and compact containment for $U_f^{j,m}(\cdot,\cdot).$
After establishing compact containment (condition \eqref{compactcontainmentcond}) of $|\awmass(\hat{\boldsymbol{\tm}}^m(t-))|,$ $\tilde{U}_{f}^{j,m}(r-,t)$, $h_{f,r-,t}^{j,m}(\cdot),$ and $\sum_{k=1}^K\sum_{l=1}^J\bar{\spr}^{k,m}_l(\bar{g}^{k,m}_l(s)),$ compact containment of $R_f^{j,m}(\cdot,\cdot)$ then follows from \eqref{finalRjineq} and Lemma \ref{ornsteinuhlenbecktightnesscondition}.
Then, if $U_f^{j,m}(\cdot,\cdot)$ is C-tight, C-tightness of $R_f^{j,m}(\cdot,\cdot)$ follows from the same Lemma with \eqref{finalrjeq}.
\end{proof}
We must now prove compact containment for $\{U_f^{j,m}(\cdot,\cdot)\}_{m=1}^{\infty}.$
We begin by examining the convergence of the fifth term in $U_f^{j,m}(\cdot,\cdot)$.
\begin{lem}
    Let $\{H^m(\cdot)\}$ be a sequence of processes in $D(\R_+, \R)$ such that $H^m(\cdot) \Rightarrow H(\cdot)$ for some process $H(\cdot).$
    Then the process
    $\int_0^{\cdot}H^m(s-)d\hat{\epsilon}^{k,j,m}(s)\Rightarrow 0$ for $j \in [J], k \in [K].$
    \label{epsilonconvlem}
\end{lem}
\begin{proof}
Recall from its definition \eqref{epsilondef} that for each $k \in [K], j \in [J], m \in \N,$ $\hat{\epsilon}^{k,j,m}(s)$ is a constant multiple of the difference of $\hat{O}^{\spr,k,j,m}(\bar{g}_j^{k,m}(\cdot))$ and $\hat{O}^{\tilde{\spr},k,j,m}(\bar{g}_j^{k,m}(\cdot)).$
    With respect to their individual filtrations, both $\hat{O}^{\spr,k,j,m}$ and $\hat{O}^{\tilde{\spr},k,j,m}$ have stochastically bounded quadratic variations, and thus satisfy the UCV condition given in \cite{protterkurtz}.
    In particular, using the predictable quadratic variation calculated for these terms in the proof of Lemma \ref{martingaleconv}, the fact that the expectation of the quadratic variation is the same as the expectation of the predictable quadratic variation and Assumption \ref{basicassumptions}, the condition (7.8) given in that paper is straightforward to check.
(For more details about the use of \cite{protterkurtz} in this paper, see the proof of Corollary \ref{martingaleconvcor}).
It follows that
\begin{align*}
&\left(\int_0^{\cdot} H^m(s-)d\hat{O}^{\spr,k,j,m}(\bar{g}_j^{k,m}(s)),\int_0^{\cdot} H^m(s-) d\hat{O}^{\tilde{\spr},k,j,m}(\bar{g}_j^{k,m}(s))\right)\\&\Rightarrow\left(\int_0^{\cdot} H(s-) d\hat{O}^{\spr,k,j}(\bar{g}_j^{k}(s)),\int_0^{\cdot} H(s-)  d\hat{O}^{\tilde{\spr},k,j}(\bar{g}_j^{k}(s))\right),
\end{align*}
where the limits $\hat{O}^{\spr,k,j}(\bar{g}_j^{k}(s))$ and $\hat{O}^{\tilde{\spr},k,j}(\bar{g}_j^{k}(s))$ are as established in Lemma \ref{martingaleconv}.
However, from \eqref{ederiv}, the convergence
$$\hat{O}^{\spr,k,j,m}(\bar{g}_j^{k,m}(\cdot))-\hat{O}^{\tilde{\spr},k,j,m}(\bar{g}_j^{k,m}(\cdot)) =\frac{1}{\sqrt{m}}  \frac{1}{\sr_j}1_{\{c_j^{k,m}(m\cdot)=1\}}\left(1-\sr_j\ser_{{\spr}_j^{k,m}(m\cdot)+1}^{k,j,m}\right) \Rightarrow 0$$
follows from the argument used in the proof of Lemma \ref{dremainderconv} to bound the analogous quantity \eqref{tvcalc}, with $\spr_j^k(\cdot)$ in place of $E(\cdot)$ and $\ser_l^{k,j,m}$ in place of $x_l^{i,m}$, noting that $\bar{g}^{k,m}_j(t) \leq t$ 
 $\forall t\geq 0.$
It follows from the last two displays that
$$\int_0^{\cdot} H^m(s-) d\hat{\epsilon}^{k,j,m}(s) \Rightarrow \int_0^{\cdot} H(s-)d\left(\hat{O}^{\spr,k,j}(s)-\hat{O}^{\tilde{\spr},k,j}(s)\right) =\int_0^{\cdot}H(s-) d0=0.$$
\end{proof}
Now, we must handle the multiparameter martingale terms.
This will require a very small extension to the Bickel-Wichura tightness criterion.
Ususally, one needs to check the tightness criterion for neighboring rectangles that abut on any arbitrary face.
This small extension clarifies that checking along a fixed face, for example, along the second face for blocks in $\R^3$, is actually enough.
In the following lemma, we use some notation from \cite{BickelWichura1971} that may not be entirely standard, so we include some of their definitions here.
\begin{defi}[$i$-neighbors, from \cite{BickelWichura1971}]
A block $B$ in $[0,T]^d$ is a subset of $[0,T]^d$ of the form $(s,t]=\Pi_{p}(s_p,t_p].$ Disjoint blocks $B$ and $C$ are $i$-neighbors if they abut and have the same $i$th face.
    
\end{defi}
Also, in the following extension to the Bickel-Wichura tightness condition, the modulus of continuity $M''$ is the extension of the $M''$ modulus discussed by Billingsley in his book Convergence of Probability Measures \cite{billingsley} to multiparameter processes, which is carefully defined in the beginning of section 2 of \cite{BickelWichura1971}.
\begin{lem}
    Let $X$ be a multiparameter process on $[0,T]^d$ and $\mu$ be a finite nonnegative measure on $[0,T]^d.$
    Let $\beta>1,\gamma >0.$
    Then we say that $(X,\mu) \in \mathscr{C}_i(\beta,\gamma)$ if, for all blocks $B,C$ in $[0,T]^d$ that are $i$-neighbors,
    \begin{equation}
        P(min\{|X(B)|,|X(C)|\}\geq \lambda) \leq \lambda^{-\gamma}\mu((B\cup C))^{\beta}\label{extendedBWcondition}
    \end{equation}
    If, for some $1\leq i\leq d,$ $(X,\mu) \in \mathscr{C}_i(\beta,\gamma)$ then there is a constant $L_d(\beta,\gamma)$ such that, for each $\lambda >0,$
    $$P\{M''(X) \geq \lambda\}\leq L_d(\beta, \gamma) \lambda^{-\gamma}.$$
    It follows that any sequence of processes $\{X_n\}_{n=1}^{\infty}$ such that there is an $i$ in $\{1,...,d\}$ and a finite nonnegative measure $\mu$ on $[0,T]^d$ such that $(X_n,\mu) \in \mathscr{C}_i(\beta,\gamma) $ for each $n,$ then $\{X_n\}_{n=1}^{\infty}$ is tight.
    \label{bickelwichuraextlemma}
\end{lem}

\begin{proof}
    The extension is immediate from the proof method, which uses induction on the dimension for $d\geq 2$.
    In particular, step 5 of the proof of Theorem 1 in Bickel-Wichura proves the extension Lemma above for $i=1$ using the original Bickel-Wichura theorem on dimension $d-1$
    Because the ordering of the indices is arbitrary, step 5 works as a proof of the extension for any $1 \leq i \leq d.$
\end{proof}
\begin{lem}
    For each $T>0, f \in \mathscr{C}\cup \{1\},$ $i,j \in [J],$ $k \in [K]$, the multi-index processes $\{\hat{\mart}_{t_{t-\cdot}f}^{\ap_j,j,m}(r): 0 \leq t \leq T, 0 \leq r \leq t\}$  and $\{\hat{\mart}_{t_{t-\cdot}f}^{\spr_j^k,i,m}(r): 0 \leq t \leq T, 0 \leq r \leq t\}$ are tight and converge to the process defined in Theorem \ref{Ydefthm}.
    \label{timechangemartingalesaretight}
\end{lem}
\begin{proof}
Applying the quadratic covariation calculations in Corollary \ref{martingaleconvcor} with the specific functions \eqref{phiAdef} and \eqref{truefunctioninsprmart} and simplifying using Lemma \ref{sjkfluidlimit}, we see that, fixing $t_1,...,t_a$, the processes
    $\{\hat{\mart}_{t_{t_b-\cdot}f_c}^{\ap_j,j,m}(\cdot)\}_{b\in[a],c\in[n],j\in [J]}$ and  $\{\hat{\mart}_{t_{t_b-\cdot}f_c}^{\spr_j^k,i,m}(\cdot)\}_{b\in[a],c\in[n],i,j\in [J], k \in [K]}$
   converge jointly in distribution to a process that has the finite dimensional distributions given in Theorem \ref{Ydefthm}.
    This gives convergence of finite dimensional distributions to the desired limit process.
    The meat of this proof will be in checking tightness.
We begin with the case of $f \in \mathscr{C}.$
Applying Corollary \ref{martingaleconvcor}, we see that, fixing $t,$ the processes
$\hat{\mart}_{t_{t-\cdot}f}^{\ap_j,j,m}(\cdot)$  and $\hat{\mart}_{t_{t-\cdot}f}^{\spr_j^k,i,m}(\cdot)$
 converge to continuous processes.
 Thus, it suffices to show that there is a bound on the modulus of continuity in the $t$ variable that is uniform in $r,m.$
We use the fact that $(\hat{\mart}^{\ap_j,j,m}_{t_{t-\cdot} f-t_{s -\cdot} f}(\cdot))^2-\langle\hat{{\mart}}^{\ap_j,j,m}_{t_{t-\cdot}f-t_{s-\cdot}f}\rangle_{\cdot}$ is a martingale to obtain for each $r \in [0,T]$
\begin{align*}
E[(\hat{\mart}^{\ap_j,j,m}_{t_{t-\cdot} f}(r)-\hat{\mart}^{\ap_j,j,m}_{t_{s -\cdot} f}(r))^2]&=
E[(\hat{\mart}^{\ap_j,j,m}_{t_{t-\cdot} f-t_{s -\cdot} f}(r))^2] = E[\langle\hat{{\mart}}^{\ap_j,j,m}_{t_{t-\cdot}f-t_{s-\cdot}f}\rangle_{r}].
\end{align*}
Then, using the form of the quadratic variation, an explicit calculation of which is included in the proof of Corollary \ref{martingaleconvcor} equation \eqref{quadraticvarcalc},
\begin{align*}
   E[\langle\hat{{\mart}}^{\ap_j,j,m}_{t_{t-\cdot}f-t_{s-\cdot}f}\rangle_{r}]& \leq E\left[\frac{1}{m} \sum_{i=1}^{m \bar{\ap}_j^m(r )}||f'||^2|t-s|^24\right]\\
   & \leq 4 ||f'||^2|t-s|^2\sup_mE[\bar{\ap}_j^m(T)].
\numberthis\label{bdgcalc}
\end{align*}
Applying Assumption \ref{basicassumptions}, particularly that $\sup_mE[\bar{\ap}_j^m(T)]< \infty,$ and the Kolmogorov continuity condition, we achieve a modulus of continuity in the $t$ variable that is uniform in $r,m$.
Following the same argument for $\hat{\mart}^{\spr_j^k,i,m}_{t-\cdot f}(r),$ we obtain for each $r \in [0,T],$
\begin{align*}
   E[(\hat{\mart}^{\spr_j^k,i,m}_{t_{t-\cdot} f}(r)-\hat{\mart}^{\spr_j^k,i,m}_{t_{s -\cdot} f}(r))^2]
   & \leq 4 ||f'||^2|t-s|^2\sup_mE[\bar{\spr}_j^{k,m}(T)].
\end{align*}
Therefore, the same proof suffices for $\hat{\mart}^{\spr_j^k,i,m}_{t_{t-\cdot} f}(r).$

For $f \equiv 1,$ the proof is trickier because the family of martingales does not vary continuously in the $t$ variable. In particular, there are jumps at points where $t= \pat_i^j +\iinta_i^j$ for some $i.$ Therefore, we must carefully prove tightness from first principles. We use Lemma \ref{bickelwichuraextlemma}, which will be more amenable to a process with jumps.
We note that, while we work on the domain $\{0 \leq r\leq t \leq T\}$ rather than a square domain such as $[0,T]^2,$ this will make no meaningful difference to the proof.
In particular, observe that one can extend the domain to the entire square by analyzing the stopped martingales $\hat{\mart}_{t_{t-\cdot}f}^{\ap_j,j,m}(\cdot\wedge t)$  and $\hat{\mart}_{t_{t-\cdot}f}^{\spr_j^k,i,m}(\cdot\wedge t)$ for each fixed $t$ which have the same global supremum and modulus of continuity as our processes.

For the remainder of this proof, we fix $j,k$ to minimize notation.
It suffices to show that
$$X^{A,m}(r,t):= \frac{1}{\sqrt{m}}\sum_{i=1}^{\lfloor mr \rfloor} \left(1_{\{\pat_i^j > t-\iinta_i^j/m\}} -\pd_j ((t-\iinta_i^j/m, \infty))\right)$$
and
$$X^{V,m}(r,t):=\frac{1}{\sqrt{m}}\sum_{i=1}^{\lfloor mr \rfloor}\left( 1_{\{T_i^{j,k} > t-\tau_i^{\spr,k,j}\}} -\frac{\langle 1_{(t-\tau_i^{\spr,k,j,m},\infty)}, \bar{\ssp}_j^m(\tau_i^{\spr,k,j,m}-)\rangle}{\wmass(\boldsymbol{\bar{\tm}^m}(\tau_i^{\spr,k,j,m}-))}\right)$$
is tight, as $\hat{\mart}_{t_{t-\cdot}f}^{\spr_j^k,j,m}(r)= X^{m,A}(\bar{\ap}_j^m(r),t),$ $\hat{\mart}_{t_{t-\cdot}f}^{\ap_j,j,m}(r)= X^{m,V}(\bar{\spr}_j^{k,m}(r),t)$ and $\bar{\ap}_j^m(\cdot),\bar{\spr}_j^{k,m}(\cdot)$ time-changes that converge to deterministic, continuous, increasing processes.
(It is straightforward to check that the composition processes will inherit the C-tightness of the processes before the time-change in this situation. However, for a nice proof of the extension of the random time change theorem to the two-parameter setting, see \cite{TalrejaWhitt2009}, Theorem 2.4).
Applying Lemma \ref{bickelwichuraextlemma}, we will check the condition for blocks that are 1-neighbors.
Now, if the increment in $r$ is less than or equal to $\frac{1}{2m}$ for both $B$ and $C,$ then the change in $r$ for one of the increments must be contained in some $(n/m,(n+1)/m],$ i.e. the increment must be zero.

If at least one of $B$ or $C$ has an increment in $r$ that is larger than $\frac{1}{2m},$ then we can use moments to bound the probability.
Because the rectangles are adjacent along the $t$ edge, they are of the form $B=(r_1,r_2]\times(t_1,t_2],C=(r_2,r_3]\times(t_1,t_2].$
We will be able to satisfy this condition using moment bounds.
In particular, we define $$\xi_i^A(t_1,t_2):= 1_{\{\pat_i^j \in (t_1-\iinta_i^j/m,t_2-\iinta_i^j/m]\}}-\pd_j((t_1-\iinta_i^j/m,t_2-\iinta_i^j/m])$$
and 
$$\xi_i^V(t_1,t_2):=1_{\{T_i^{j,k} \in (t_1-\tau_i^{\spr,k,j,m},t_1-\tau_i^{\spr,k,j,m}]\}} -\frac{\langle 1_{(t_1-\tau_i^{\spr,k,j,m},t_2-\tau_i^{\spr,k,j,m}]}, \bar{\ssp}_j^m(\tau_i^{\spr,k,j,m}-)\rangle}{\wmass(\boldsymbol{\bar{\tm}^m}(\tau_i^{\spr,k,j,m}-))}$$
for $i\in \N.$
Then, we have that for $\lambda >0,$ using Chebychev's inequality and the fact that $min(x,y)^4\leq x^2y^2$ for any $x,y\in \R,$ we obtain
$$P(\min(|X^A(B)|,|X^A(C)| \geq \lambda) \leq \lambda^{-4}E[(X^A(B))^2(X^A(C))^2].$$
Then, using the conditional independence of $\xi_i,\xi_k,$ we find that
\begin{align*}
\lambda^{-4}E[(X^A(B))^2(X^A(C))^2]&\leq \lambda^{-4}E[(X^A(B))^2(X^A(C))^2]\\
& \leq \frac{\lambda^{-4}}{m^2}(\sum_{\lfloor m r_1\rfloor+1}^{\lfloor m r_2 \rfloor}E[\xi_i^A(t_1,t_2)^2])(\sum_{\lfloor m r_2\rfloor+1}^{\lfloor m r_3 \rfloor}E[\xi_i^A(t_1,t_2)^2])\numberthis\label{linefrom}\\
& \leq \frac{\lambda^{-4}}{m^2}C^2(m r_2-m r_1 +1)(t_2-t_1)^{\alpha}(m r_3-mr_2 +1)(t_2-t_1)^{\alpha}\\
&  \leq \lambda^{-4} C^2( r_2- r_1 +\frac{1}{m})(t_2-t_1)^{\alpha}(r_3-r_2 +\frac{1}{m})(t_2-t_1)^{\alpha},
\end{align*}
where $C,\alpha$ are the H\"{older} constant and exponent for the cdf of $\pd_j.$
Now, because we are in the case where at least one of the increments in $r$ is larger than $\frac{1}{2m},$ we may assume without loss of generality that $r_2-r_1 \geq r_3-r_2$ and $r_2-r_1 \geq \frac{1}{2m}.$
We conclude that
\begin{align*}
    P(\min(|X^A(B)|,|X^A(C)| \geq \lambda)&\leq \lambda^{-4} 9C^2(r_2-r_1)^2(t_2-t_1)^{2\alpha}\\
    & \leq \lambda^{-4}9C^2 ((r_2-r_1)(t_2-t_1))^{ 2 \alpha}\\
    &\leq \lambda^{-4}9C^2 ((r_3-r_1)(t_2-t_1))^{ 2 \alpha}
\end{align*}
Taking $\gamma = 4,$ $\beta = 2\alpha,$ and $\mu$ to be the Lebesgue measure times $9C^2,$ we have \eqref{extendedBWcondition}.
The same proof works for this case for $X^V,$ it just takes a bit more work to calculate the moments of the $\xi_i^V(t_1,t_2)$.
In particular, once we show that $E[\xi_i^V(t_1,t_2)^2]\leq C(t_2-t_1)^{\alpha}$ for each $i,$ the proof from \eqref{linefrom} forward is identical.
So we bound these moments now.
First, observe that, using the tower property to condition on $\mathscr{F}_{\tau_i^{V,k,l,m}-}$ inside the expectation,
\begin{align*}
     &E[\xi_i^V(t_1,t_2)^2]\\& \leq 4E\left[\left(\langle1_{(t_1-\tau_i^{\spr,k,l,m},t_2-\tau_i^{\spr,k,l,m}]},\bar{{\ssp}}^m_j(\tau_i^{\spr,k,l,m}-)\rangle+\langle1_{(t_1-\tau_i^{\spr,k,l,m},t_2-\tau_i^{\spr,k,l,m}]},\bar{{\ssp}}^m_j(\tau_i^{\spr,k,l,m}-)\rangle^2\right)\right],
    \end{align*}
and therefore, using the fact that for $0 \leq s \leq T,$ \\
    $\langle 1_{(x-s,y-s]},\bar{\ssp}_j^m(s-) \rangle\leq \langle 1_{(x,y]},\bar{\ssp}_j^m(0) \rangle +\frac{1}{m}\sum_{i=1}^{m\bar{\ap}_j^m(T)}1_{\{x-\iinta_i^{j,m}/m < \pat_i^j \leq y-\iinta_i^{j,m}/m \}}$, 
\begin{align*}
&\leq E\left[\left(\langle 1_{[(t_1,t_2]},\bar{\ssp}_j^m(0) \rangle +\frac{1}{m}\sum_{i=1}^{m\bar{\ap}_j^m(T)}1_{\{t_1 < \pat_i^j+\iinta_i^{j,m}/m \leq t_2\}}\right)\right]
\\&+E\left[\left(\langle 1_{(t_1,t_2]},\bar{\ssp}_j^m(0) \rangle +\frac{1}{m}\sum_{i=1}^{m\bar{\ap}_j^m(T)}1_{\{t_1 < \pat_i^j+\iinta_i^{j,m}/m \leq t_2\}}\right)^2\right].
\end{align*}
Using the independence of $\{\pat_i^j\}_{i=1}^{\infty},\{\iinta_i^{j,m}\}_{i=1}^{\infty}$ from $\bar{\boldsymbol{\ssp}}^m(0),$ we bound the above quantity with
\begin{align*}
& \leq E[J]+E[K]+3E[J^2]+3E[K^2] 
\end{align*}
where $J=\langle 1_{[(t_1,t_2]},\bar{\ssp}_j^m(0) \rangle$ and $K=\frac{1}{m}\sum_{i=1}^{m\bar{\ap}_j^m(T)}1_{\{t_1 < \pat_i^j+\iinta_i^{j,m}/m \leq t_2\}}.$
Using the continuity assumptions on $\pd_j$ and $\{\bar{\ssp}_j^m(0)\}_{m=1}^{\infty},$
we obtain a bound of the form
\begin{equation}
    E[\xi_i^V(t_1,t_2)^2] \leq C(t_2-t_1)^{\alpha}
\end{equation}
for some $C,$ as required.

\end{proof}
\begin{proof}[Proof of Theorem \ref{Ydefthm}]
    We now check that these finite dimensional distributions specify a unique continuous process for each family of martingales and test function $f_c \in \mathscr{C}\cup \{1\}$.
    To do this, we apply the multiparameter version of the Kolmogorov continuity condition.
    Examining a multiparameter increment $(r_1,r_2]\times (t_1,t_2],$ we see that for $p \geq 2$, it follows from the BDG inequalities that for either family of martingales (we will just denote the family with a $\circ$ for now) and any $j$
    \begin{align*}
        &E[(\hat{\mart}_{t_{t_1-\cdot}f}^{\circ,j}(r_1)-\hat{\mart}_{t_{t_2-\cdot}f}^{\circ,j}(r_1)-(\hat{\mart}_{t_{t_1-\cdot}f}^{\circ,j}(r_2)-\hat{\mart}_{t_{t_2-\cdot}f}^{\circ,j}(r_2)))^p]\\
        &\leq C_p( E[\langle (\hat{\mart}_{t_{t_1-\cdot}f}^{\circ,j}-\hat{\mart}_{t_{t_2-\cdot}f}^{\circ,j})(\cdot-r_1)-(\hat{\mart}_{t_{t_1-\cdot}f}^{\circ,j}-\hat{\mart}^{\circ,j}_{t_{t_2-\cdot}f})(r_1)\rangle_{r_2-r_1}^{p/2}]
    \end{align*}
    The right-hand side above has different bounds depending on the choice of $f$ and the family of martingales. For $f \in \mathscr{C},$ it is straightforward to check that the RHS is bounded by $C_f(t_2-t_1)^p(r_2-r_1)^{p/2},$ where $C_f$ depends only on the function $f$ and the Lipschitz constants of the fluid limits for the arrival and service processes.
    For $f \equiv 1,$ we have $C (t_2-t_1)^{p\alpha/2}(r_2-r_1)^{p/2},$ where $\alpha,C$ are calculated using the H\"{o}lder coefficients and constants of the cdfs of $\pd_j$ and $\fl_j(0).$
    We note that, for the $\spr_j^k$ martingales, one needs to use the fact that $\fl_j(s) [u,v) \leq \fl_j(0) [s+u,s+v)+\int_0^t \ar_j \pd_j[r+u,r+v)dr.$
    Applying the criterion, we see that there exists a unique H\"{o}lder continuous modification by taking $p$ such that $p,p/2>1$ or $p\alpha/2,p/2>1,$ respectively.

\end{proof}
\begin{lem}
\label{Uconvlemma}
    For each $T>0, f_1,...,f_n \in \mathscr{C}\cup \{1\}$, the multi-index process $\{U_{f_a}^{j,m}(r,t): 0 \leq r \leq t \leq T, a\in [n], j \in \J \}$ is C-tight and converges to the process defined in Definition \ref{Udefdef}.
\end{lem}
\begin{proof}
Convergence for the first term on the right hand side of \eqref{ujmdef} follows from Assumption \ref{initialconditionsassumption}.
We now examine the fourth and sixth terms of \eqref{ujmdef}.
C-tightness will follow from the martingale convergence given in Lemmas \ref{dremainderconv} and \ref{martingaleconv} combined with the decomposition of diffusion-scaled renewal processes given in \eqref{martingalepartofEhat} and \eqref{remainderpartofEhat} with a small adjustment to the proofs to allow for the integrands to vary in $t.$
In particular, in the proof of Lemma \ref{dremainderconv}, the bound \eqref{tvcalc} becomes 
\begin{equation*}
    \frac{\max\{{\ser_l^{j,m}}:l \leq \spr_j^{k,m}(mT)\}}{\sqrt{m}}\sup_{t\leq T}TV(\flcdf_f^{j,c}(\cdot,t-\cdot))_{[0,t]}
\end{equation*}
for the integrals in the fourth term and
\begin{equation*}
    \frac{\max\{{\inta_l^{j,m}}:l \leq \ap_j^m(mT)\}}{\sqrt{m}}\sup_{t\leq T}TV(\pdcdf_f^{j,c}(t-\cdot))_{[0,t]}
\end{equation*}
for the sixth term.
We note that the integral equation derived for \eqref{masstransportindicator} in Lemma 6.1 of \cite{loeserwilliams} also holds for the more general \eqref{masstransportequation}, and using this equation it is straightforward to find a Lipschitz constant for $\flcdf_f^{j,c}(\cdot,t-\cdot)$ that is uniform over $t \in [0,M]$ and verify that $\sup_{t \leq M}TV(\flcdf_f^{j,c}(\cdot,t-\cdot))_{[0,t]}< \infty$.
Because $\pdcdf_f^{j,c}(\cdot)$ is decreasing and takes the value $1$ at zero, $TV(\pdcdf_f^{j,c}(t-\cdot))_{[0,t]}\leq 1$ for each $M>0.$
Examining the remaining martingale parts, the same argument as was used in the Proof of Theorem \ref{Ydefthm} will suffice using the finite dimensional distributions given by Lemma \ref{martingaleconv}.
\end{proof}

\section{Proof of Theorems 3.2, 3.3, and 3.4}
 \begin{proof}[Proof of Theorems \ref{uniquenessoftotalmassthm} and \ref{lhatconvergencethm} ]
Applying Lemma \ref{Lcompactcontainmentlemma} and Lemma \ref{Uconvlemma}, we see that $\{\boldsymbol{R}_f^{j,m}(r,t)\}_{m=1}^{\infty}$ is a C-tight sequence of multiparameter processes for each $j \in \J,$ $f \in \mathscr{C}\cup \{1\}$ and satisfies \eqref{finalrjeq} for each $m \in \N.$
It follows from convergence of the fluid model and Lemma \ref{sjkfluidlimit} that a subsequential limit of $\{\boldsymbol{R}_f^{j,m}(r,t)\}_{m=1}^{\infty},$ $\boldsymbol{R}_f^{j}$ satisfies \eqref{rjlimitdef} in each coordinate $j$.
Therefore, to complete the proof of Theorem \ref{lhatconvergencethm}, we must show that, if $\awmass(\hat{\boldsymbol{\tm}}(x)), U_f^j(\cdot,\cdot),$ and the fluid limit total mass vector are known, this uniquely specifies any solution $H_f^j(\cdot,\cdot)$ of \eqref{rjlimitdef}.
Fix $t \geq 0, j \in \J,$ $f \in \mathscr{C}\cup \{1\},$ $U_f^j(\cdot,t)$, and let $H_f^j(\cdot,\cdot)$ be a solution to \ref{rjlimitdef}.
Then, because $U_f^j$ is a semimartingale, it follows that $H_f^j(\cdot,t)$ is a semimartingale.
Therefore, applying integration by parts for stochastic integrals with the deterministic integrating factor
$$ b_j(r):=\exp\left(\int_0^r \frac{Kp_j}{\awmass(\boldsymbol{\flm}(s))}ds\right)$$
 we see that, for $0 \leq r \leq t,$
\begin{align}
    &b_j(r)H_f^j(r,t)\\&=H_f^j(0,t)+\int_0^rb_j(x)dH_f^j(x,t) \label{hsub}\\ 
    &+\int_0^rb_j(x)H_f^j(x,t)\frac{Kp_j}{\awmass(\boldsymbol{\flm}(x))}dx.
\end{align}
Substituting the RHS of \eqref{rjlimitdef} in for the integrator in \eqref{hsub}, we obtain
\begin{align*}
    H_f^j(r,t)&=H_f^j(0,t)\exp\left(-\int_0^r \frac{Kp_j}{\awmass(\boldsymbol{\flm}(s))}ds\right)+\int_0^r \exp\left(-\int_x^r \frac{Kp_j}{\awmass(\boldsymbol{\flm}(s))}ds\right)dU_f^j(x,t)\\
    &+\int_0^r\exp\left(-\int_x^r \frac{Kp_j}{\awmass(\boldsymbol{\flm}(s))}ds\right)\frac{Kp_j {{\flcdf}}^{j,c}_f(x,t-x)}{\awmass(\boldsymbol{\flm}(x))^2}\awmass(\hat{\boldsymbol{\tm}}(x)) dx.
    \numberthis \label{Hsolutionintegratingfactor}
\end{align*}
Therefore, for each $f \in \mathscr{C}\cup \{1\},$ $H_f(r,t)$ is an explicit function of $\awmass(\hat{\boldsymbol{\tm}}(x)), U_f(\cdot,\cdot),$ and the fluid limit total mass vector.
This completes the proof of theorem \ref{lhatconvergencethm}.

We continue to the proof of Theorem \ref{uniquenessoftotalmassthm}.
Applying \eqref{Hsolutionintegratingfactor} with $f=1,$ we see that it suffices to show that \eqref{Lhatlimit}, 
\begin{align*}
    \sum_{j=1}^J\frac{p_j}{\sr_j} H_1^j(t,t)&=\sum_{j=1}^J\frac{p_j}{\sr_j}H_1^j(0,t)b(t)^{-1}+\sum_{j=1}^J\frac{p_j}{\sr_j}\int_0^t \exp\left(-\int_x^t \frac{Kp_j}{\awmass(\boldsymbol{\flm}(s))}ds\right)dU_1^j(x,t)\\
    &+\int_0^t\sum_{j=1}^J\frac{p_j}{\sr_j}\exp\left(-\int_x^t \frac{Kp_j}{\awmass(\boldsymbol{\flm}(s))}ds\right)\frac{Kp_j {{\flcdf}}^{j,c}_f(x,t-x)}{\awmass(\boldsymbol{\flm}(x))^2}\sum_{i=1}^J\frac{p_i}{\sr_i} H_1^i(x,x)dx
\end{align*}
has a unique solution for the process $ \{\sum_{j=1}^J\frac{p_j}{\sr_j} H_1^j(t,t): t\geq 0\}.$
Uniqueness of solutions $\sum_{j=1}^J\frac{p_j}{\sr_j} H_1^j(t,t)$ then follows from Gr\"onwall's inequality because, fixing a realization of $\boldsymbol{U}_1$, the first two terms on the right hand side above will be the same.
 \end{proof}
\label{proofofdiffusiontheoremsection}
\begin{proof}[Proof of Theorem \ref{diffusionapproximationresult}]

    For $\boldsymbol{f} \in \mathscr{C}^J,$ we use the notation $\boldsymbol{X}_{\boldsymbol{f}}(\cdot):= \langle \boldsymbol{f}, \boldsymbol{X}(\cdot) \rangle.$
    Applying equation \eqref{mainprelimiteqnnoghat}, we see that our system is a ``good sequence of diffusion-scaled renewal driven systems" with
    \begin{itemize}
        \item $\hat{\boldsymbol{X}}^m(\cdot) = \hat{\boldsymbol{X}}^m_{\boldsymbol{f}}(\cdot)$,
        \item $A= J + JK$,
        \item $(E_1(\cdot),...,E_J(
        \cdot)
        ) = (\ap_1(\cdot),...,\ap_J(
        \cdot)),$ and $(E_{kJ+1}(\cdot),...,E_{kJ+J}(
        \cdot))=(\spr_1^k(\cdot),...,\spr_J^k(\cdot))$ for $k \in [K],$
        \item $(g_1^m(\cdot),...,g_J^m(
        \cdot)
        ) = (\cdot,...,        \cdot),$ and $(g_{kJ+1}^m(\cdot),...,g_{kJ+J}^m(
        \cdot))=(\bar{g}_1^{k,m}(\cdot),...,\bar{g}_J^{k,m}(\cdot))$ for $k \in [K],$
        \item $(c_1^m,...,c_J^m) = (1,...,1)$ and $(c_{kJ+1}^m,...,c_{kJ+J}^m) = \left(\frac{1}{\sr_1},...,\frac{1}{\sr_J}\right)$ for $k \in [K],$
        \item $b_{j}^i = 1_{\{i = j\}}\langle f_i,\pd_j\rangle$ for $i,j \in [J]$ and $b_{j}^{kJ+i}(\cdot) = \frac{p_j \langle f, \fl_j(\cdot)\rangle}{\awmass(\boldsymbol{\flm}(\cdot)) }$ for $k \in [K],i,j \in [J],$
        \item ${\mart}_{i,j}^m(\cdot)= 1_{\{i=j\}}{\mart}_{f_j}^{\ap_j,j,m}(\cdot) $ for $i, j \in [J]$ and ${\mart}_{kJ+i,j}^m(\cdot) = \othermart^{\spr_i^k,j,m}_{f_j}(\cdot)$ for $i,j \in [J], k \in [K],$
        \item $\boldsymbol{r}^{i} = \boldsymbol{0}$
        \item $\boldsymbol{h}^{i,m}(t) = \boldsymbol{0}$ for $i,j \in [J], t \geq 0$ and $\boldsymbol{h}^{i,m}(t) = \frac{\boldsymbol{p}}{\wmass(\bar{\boldsymbol{\tm}}^m(t-))}$ for $J \leq i \leq J+KJ, t \geq 0,$
        \item and 
        \begin{align*}
            \hat{\boldsymbol{J}}^m(\cdot) &= -\int_0^\cdot\langle \boldsymbol{f}', \hat{\boldsymbol{\ssp}}^m(s) \rangle ds+\int_0^{\cdot}\frac{\boldsymbol{p} \langle \boldsymbol{f}, \boldsymbol{\fl}(s)\rangle }{\wmass(\bar{\boldsymbol{\tm}}^m(s-))}\left(\frac{\awmass(\hat{\boldsymbol{\tm}}^m(s-))}{\awmass(\boldsymbol{\flm}(s))}\right)d\sum_{k=1}^K\sum_{l=1}^J\bar{\spr}^{k,m}_l(\bar{g}^{k,m}_l(s))\\
            &-\int_0^{\cdot} \frac{\boldsymbol{p} \langle \boldsymbol{f}, \boldsymbol{\fl}(s)\rangle }{\awmass({\boldsymbol{\flm}}(s))} d\sum_{k=1}^K \sum_{j=1}^J\hat{\epsilon}^{k,j,m}(s)
        \end{align*}
    \end{itemize}
    We note that it is easy to check, examining the form of the martingale decompositions, that the change in each martingale at each jump time of the associated renewal process is independent of the next interevent time for that renewal process.
    Applying Theorem \ref{diffusionclt}, and recalling Remark \ref{simultaneouseventsremark}, and Assumption \ref{assumptions}, Theorem \ref{diffusionapproximationresult} is proved if we show the following
    \begin{enumerate}[(i)]
    \item For $k \in [K], j \in [J]$ covariance matrix of $\hat{\boldsymbol{\othermart}}^{\spr_j^k,m}_{\boldsymbol{f}}(\cdot)$ converges to $\int_0^{\cdot} D_{k,j}^{\boldsymbol{f}}(s)ds$ as $m \rightarrow \infty.$ For $j \in [J],$ the covariance matrix of $\hat{\boldsymbol{Y}}^{\ap_j,m}_{\boldsymbol{f}}(\cdot)$ converges to the matrix with $\ar_j(\langle f_j^2, \pd_j\rangle- \langle f_j,\pd_j \rangle^2)(\cdot)$ in the $(j,j)$ spot for $j \in [J]$ and $0$ for $(i,l) \in [J]\times[J],$ $(i,l) \neq (j,j)$ for any $j \in [J]$. Furthermore, for $T>0,$ $\lim_{m \rightarrow \infty}E[\sup_{t \in [0,T]}|\hat{\boldsymbol{Y}}^{\spr_j^k,m}_{\boldsymbol{f}}(t)-\hat{\boldsymbol{Y}}^{\spr_j^k,m}_{\boldsymbol{f}}(t-)|^2]=0$ and $\lim_{m \rightarrow \infty}E[\sup_{t \in [0,T]}|\hat{\boldsymbol{Y}}^{\ap_j,m}_{\boldsymbol{f}}(t)-\hat{\boldsymbol{Y}}^{\ap_j,m}_{\boldsymbol{f}}(t-)|^2]=0.$
        \label{checkingassumptionsoftheorem4}
        \item $\bar{g}_j^{k,m}(\cdot) \Rightarrow \int_0^{\cdot} \frac{\frac{p_j}{\sr_j}\flm_j(s)}{\awmass(\boldsymbol{\flm}(s))}ds$ for $j \in [J], k \in [K],$
        \label{checkingassumptionsoftheorem1}
        \item $b_{j}^i = 1_{\{i = j\}}\langle f_i,\pd_j\rangle$ for $i,j \in [J]$ and $b_{j}^{kJ+i}(\cdot) = \frac{p_j \langle f_j, \fl_j(\cdot)\rangle}{\awmass(\boldsymbol{\flm}(s)) }$ for $k \in [K],i,j \in [J],$ are of locally finite variation,
        \label{checklocallyfinitevariationsitem}
        \item $\boldsymbol{h}^{i,m}(\cdot) \Rightarrow \frac{\boldsymbol{p}}{\wmass(\boldsymbol{\flm}(\cdot))}$ for $J \leq i \leq J+KJ, t \geq 0,$
        \label{checkingassumptionsoftheorem2}
        \item $$\hat{\boldsymbol{J}}^m(\cdot) \Rightarrow -\int_0^\cdot\langle \boldsymbol{f}', \hat{\boldsymbol{\ssp}}(s) \rangle ds+\int_0^{\cdot}\frac{\boldsymbol{p} \langle \boldsymbol{f}, \boldsymbol{\fl}(s)\rangle }{\wmass(\boldsymbol{\flm}(s))}\left(\frac{\awmass(\hat{\boldsymbol{\tm}}(s))}{\awmass(\boldsymbol{\flm}(s))}\right)d\sum_{k=1}^K\sum_{l=1}^J\bar{\spr}^{k}_l(\bar{g}^{k}_l(s)).$$
        \label{checkingassumptionsoftheorem3}
        
    \end{enumerate}
If \eqref{checkingassumptionsoftheorem4}-\eqref{checkingassumptionsoftheorem3} are true, then applying Theorem \ref{diffusionclt} to \eqref{mainprelimiteqnnoghat} along with the explicit forms of the fluid limit quantities given in Lemma \ref{sjkfluidlimit} and Corollary \ref{cjconvcor}, we find that subsequential limits will satisfy the equation system of equations given in \eqref{diffusionlimitequation}.

We begin with \ref{checkingassumptionsoftheorem1}.
This was proved in Corollary \ref{cjconvcor}.
For \ref{checklocallyfinitevariationsitem}, we see that in \cite{loeserwilliams} Lemma 8.1, it is shown that a function $\langle f, \fl_j(\cdot) \rangle$ is Lipschitz continuous for $f \in \mathscr{C}, j \in [J].$
In \cite{loeserwilliams} Lemma 6.1, we see that each $\flm_j(\cdot),$ $j \in [J]$ satisfies an integral equation (40).
It follows from the form of this equation that each $\flm_j(\cdot)$ has finite variation on $[0,T]$ if $\awmass(\cdot)$ is bounded away from zero on that interval.
Since $\awmass(\cdot)$ is continuous and nonzero, that will be the case.
Putting these facts together, we have shown \ref{checklocallyfinitevariationsitem}.
We continue to \ref{checkingassumptionsoftheorem2}.
This is immediate from Theorem \ref{fluidlimittheorem} and Lemma \ref{nonzerolemma} (which also implies that $\wmass(\bar{\boldsymbol{\tm}}^m(s))$ is eventually bounded away from $\boldsymbol{0}$).
For \ref{checkingassumptionsoftheorem3} we see that the limit of the first term follows from the same argument as was used in the proof of Theorem \ref{diffusionclt} to obtain convergence of the term $\int_0^t \boldsymbol{r}^{i}(\hat{\boldsymbol{X}}^m(s))ds .$
Convergence of the second term in $\hat{\boldsymbol{J}}^m$ follows from the same argument as was used for convergence of the term $\int_0^{\cdot}\boldsymbol{h}^{i,m}(s)\hat{\boldsymbol{X}}^m(s-)d\bar{{E}}^m_i(g^m_i(s))$ with $\boldsymbol{h}^{i,m}(s) = \frac{\boldsymbol{p}\langle \boldsymbol{f},\boldsymbol{\fl}(s)\rangle}{\awmass(\boldsymbol{\flm}(s))\wmass(\bar{\boldsymbol{\tm}}^m(s))}$, $\hat{\boldsymbol{X}}^m(s-) = \awmass(\hat{\boldsymbol{\tm}}^m(s-))$, and the time changed renewals appropriately substituted.
The convergence of the error terms with integrator $\hat{\epsilon}^{k,j,m}$ was proved in Lemma \ref{epsilonconvlem}.
Therefore, the heart of this proof is checking \ref{checkingassumptionsoftheorem4}.
First, we note that because the $\boldsymbol{f}$'s are bounded, the jumps of ${\boldsymbol{Y}}^{\ap_j,m}_{\boldsymbol{f}}(\cdot)$ and ${\boldsymbol{Y}}^{\spr_j^k,m}_{\boldsymbol{f}}(\cdot)$ are uniformly bounded, and thus the jumps of $\hat{\boldsymbol{Y}}^{\ap_j,m}_{\boldsymbol{f}}(\cdot)$ and $\hat{\boldsymbol{Y}}^{\spr_j^k,m}_{\boldsymbol{f}}(\cdot),$ which are $1/\sqrt{m}$ times the jumps of ${\boldsymbol{Y}}^{\ap_j,m}_{\boldsymbol{f}}(\cdot)$ and ${\boldsymbol{Y}}^{\spr_j^k,m}_{\boldsymbol{f}}(\cdot)$, satisfy the condition given for the jumps.
We turn our attention to the convergence of the predictable quadratic covariation matrices. 
We compute these now.
Applying Corollary \ref{martingaleconvcor}, \eqref{truefunctioninapmart}, the fact that $\bar{\ap}_j'(t) = \ar_j,$ and the fact that $\hat{\mart}^{\ap_j,i,m}_{f_i} = 0$ for $i \neq j,$ we have that $\langle \hat{\mart}^{\ap_j,i,m}_{f_i},\hat{\mart}^{\ap_j,l,m}_{f_l}\rangle = 0 $ if $i\neq j$ or $l \neq j,$ and when $i = j =l$ we have
$$\langle \hat{\mart}^{\ap_j,j,m}_{f_j},\hat{\mart}^{\ap_j,j,m}_{f_j}\rangle_{\cdot}\Rightarrow \ar_j (\cdot)\left(\langle f_j^2, \pd_j\rangle -\langle f_j, \pd_j\rangle^2 \right) .$$
Again applying Corollary \ref{martingaleconvcor} with the function \eqref{truefunctioninsprmart} and Lemma \ref{sjkfluidlimit}, we have
\begin{equation}
    \langle \hat{\mart}^{\spr_j^k,i,m}_{f_i},\hat{\mart}^{\spr_j^k,l,m}_{f_l}\rangle_{\cdot} \Rightarrow \int_0^{\cdot} \left(1_{\{i=l\}}\frac{p_i \langle f_i^2,\fl_i(s)\rangle}{\wmass(\boldsymbol{\flm}(s))}- \frac{p_i\langle f_i,\fl_i(s)\rangle p_l\langle f_l,\fl_l(s)\rangle }{\wmass(\boldsymbol{\flm}(s))^2}\right)\frac{p_j\flm_j(s)}{\awmass(\boldsymbol{\flm}(s))}ds .
    \label{ysprpredquad}
\end{equation}
Then, noting that
\begin{align*}
    &\langle \hat{\othermart}^{\spr_j^k,i,m}_{f_i},\hat{\othermart}^{\spr_j^k,l,m}_{f_l}\rangle_t \\&= \Bigg\langle \hat{\mart}^{\spr_j^k,i,m}_{f_i}-\int_0^{\cdot} \frac{p_i \langle f_i, \fl_i(s)\rangle }{\awmass({\boldsymbol{\flm}}(s))} d\sum_{n=1}^J\frac{1}{\sr_n} \hat{\mart}^{\spr_j^k,n,m}_1(s),\\&\hat{\mart}^{\spr_j^k,l,m}_{f_l}-\int_0^{\cdot} \frac{p_l \langle f_l, \fl_l(s)\rangle }{\awmass({\boldsymbol{\flm}}(s))} d\sum_{n=1}^J\frac{1}{\sr_n} \hat{\mart}^{\spr_j^k,n,m}_1(s)\Bigg\rangle_t\\
    & = \Bigg\langle\hat{\mart}^{\spr_j^k,i,m}_{f_i},\hat{\mart}^{\spr_j^k,l,m}_{f_l}\Bigg\rangle_t\\
    &-\sum_{n=1}^J \int_0^t \frac{p_i \langle f_i, \fl_i(s)\rangle }{\awmass({\boldsymbol{\flm}}(s))}\frac{1}{\sr_n} d\Bigg\langle \hat{\mart}^{\spr_j^k,n,m}_1,\hat{\mart}^{\spr_j^k,l,m}_{f_l}\Bigg\rangle_t\\
    &-\sum_{n=1}^J \int_0^t \frac{p_l \langle f_l, \fl_l(s)\rangle }{\awmass({\boldsymbol{\flm}}(s))}\frac{1}{\sr_n}d\Bigg\langle\hat{\mart}^{\spr_j^k,i,m}_{f_i}, \hat{\mart}^{\spr_j^k,n,m}_1\Bigg\rangle_t\\
    &+ \sum_{n=1}^J \sum_{x=1}^J\int_0^t\frac{p_i \langle f_i, \fl_i(s)\rangle }{\awmass({\boldsymbol{\flm}}(s))}\frac{1}{\sr_n}\frac{p_l \langle f_l, \fl_l(s)\rangle }{\awmass({\boldsymbol{\flm}}(s))}\frac{1}{\sr_x}d\Bigg\langle\hat{\mart}^{\spr_j^k,n,m}_1,\hat{\mart}^{\spr_j^k,x,m}_1\Bigg\rangle_t,
    \numberthis \label{othermartcovdoneout}
\end{align*}
the form of $D^{\boldsymbol{f}}_{k,j}$ follows from \eqref{ysprpredquad} and \eqref{othermartcovdoneout} and a standard real analysis argument in which one takes a Skorokhod Representation and notes that the Lebesgue-Stieltjes measure induced by the predictable quadratic covariations above converges in the weak topology to the measure induced by the limiting function.
\end{proof}
 \noindent {\bf Acknowledgements} The research reported in this paper was supported in part by NSF RTG grant DMS-2134107.
 The author would also like to acknowledge the feedback and discussion provided by Amarjit Budhiraja and Zachary Bezemek, and Ben Seeger, which greatly improved this paper.
 Furthermore, the feedback of some very diligent reviewers was invaluable.
 Lastly, the author used ChatGPT for grammar, copyediting, and phrasing suggestions. However, no mathematical results or proofs were generated by the tool, and the author reviewed and edited all AI-assisted text.
\newpage

\bibliographystyle{amsplain}
\bibliography{diffusionreferences.bib}

\end{document}